\documentclass[11pt, reqno]{amsart}

\usepackage{amsfonts, amssymb, amscd}
\usepackage{graphicx}
\usepackage{hyperref}
\usepackage{parskip}
\usepackage{slashed}
\usepackage{fullpage}
\usepackage{color}
\newcommand{\Blue}[1]{{\color{blue} #1}}
\newcommand{\Red}[1]{{\color{red} #1}}

\newcommand{\baoping}[1]{{\color{blue}#1}}

\newtheorem{thm}{Theorem}[section]
\newtheorem{cor}[thm]{Corollary}
\newtheorem{lem}[thm]{Lemma}
\newtheorem{conjec}[thm]{Conjecture}
\newtheorem{prop}[thm]{Proposition}

\theoremstyle{remark}
\newtheorem{rem}[thm]{Remark}

\theoremstyle{definition}
\newtheorem{defin}[thm]{Definition}

\newcommand{\R}{\mathbb{R}}

\newcommand{\la}{\left<}
\newcommand{\ra}{\right>}
\newcommand{\lp}{\left(}
\newcommand{\rp}{\right)}

 \newcommand{\XT}{\frac{\la x \ra}{t^\alpha}\geq 1}
  \newcommand{\XTZ}{\frac{\la x \ra}{t^{\alpha_0}}\geq 1}
 \newcommand{\XTT}{\frac{\la x \ra}{t^\alpha}\leq 1}
  \newcommand{\XTTZ}{\frac{\la x \ra}{t^{\alpha_0}}\leq 1}
 \newcommand{\jx}{\left< x\right>}
  \newcommand{\XS}{\frac{\jx}{s}}
  \newcommand{\Gt}{\frac{\gamma}{\tau}}

  \newcommand{\gbb}{\gamma t^\beta >1}

 \newcommand{\bb}{\mathcal{B}_n}

\newcommand{\pwb}{\phi_{wb}}
\newcommand{\tpwb}{\tilde{\phi}_{wb}}
\newcommand{\pwl}{\phi_{wl}}
\newcommand{\pwls}{\phi_{wls}}
\newcommand{\bp}{\textbf{p}}
\newcommand{\be}{\begin{equation}}
\newcommand{\ee}{\end{equation}}
\newcommand{\A}{\mathcal{A}}
\numberwithin{equation}{section}

\title{The Large Time Asymptotics of Nonlinear Multichannel Schr\"odinger Equations }

\author{Baoping Liu}
\address{Beijing International Center for Mathematical Research\\
 Peking University\\
 Beijing\\
  China}
\email{baoping@math.pku.edu.cn}
\author{Avy Soffer}
\address{Department of Mathematics\\
Rutgers University\\
110 Frelinghuysen Rd.\\
Piscataway, NJ, 08854, USA}
\email{soffer@math.rutgers.edu}

\thanks{
B. Liu is supported in part by the NSFC12071010 and NSFC11631002.
A.Soffer is supported in part by Simons Foundation Grant number 851844
}

\usepackage{comment}
\begin{document}

\begin{abstract}
We consider the Schr\"odinger equation with a general interaction term, which is localized in space. The interaction may be $x,t$ dependent and non-linear. Purely non-linear parts of the interaction are localized via the radial Sobolev embedding.
Under the assumption of radial symmetry and boundedness in $H^1(\R^3)$ of the solution, uniformly in time.
we prove it is \emph{asymptotic} in $L^2$ (and $H^1$) in the strong sense, to a free wave and a weakly localized solution. 
The general properties of the localized solutions are derived.
The proof is based on the introduction of phase-space analysis of the nonlinear dispersive dynamics and relies on a new class of (exterior) a priory propagation estimates.
This approach allows a unified analysis of general linear time-dependent potentials and non-linear interactions.

\end{abstract}

 \maketitle

\tableofcontents

\section{Introduction}

We consider global solutions to the  nonlinear Schr\"{o}dinger equation, with spherically symmetric solutions.
\begin{equation}\left\{
\begin{aligned}i\partial_t \phi +\Delta \phi & =\textbf{N}(\phi)= \mathcal{N}(|\phi|,|x|,t)\phi, \\
\phi(0,x)& =\phi_0\in H^1_{rad}(\R^3)\cap L^2_{rad}(\R^3, |x|^\frac12dx),
\end{aligned}\right. \label{Main-eq}
\end{equation}
under the assumptions 

(H1)\label{H1} The solution satisfies the global $H^1$ bound. 
\begin{equation}\sup_{t\in [0,\infty)} \|\phi\|_{H^1(\R^3)} <\infty. \label{global-bound}\end{equation}


(H2)\label{H2} Uniform decay of the interaction term: for some $\alpha\in (\frac13, 1)$ and $F(\lambda)$ being a smooth characteristic function of the interval $[1,+\infty)$, there exist constants $\beta_0=\beta_0(\alpha, F)>1$ and $C=C( \alpha, F)$ such that
\begin{equation} \left|F(\frac{|x|}{t^\alpha}\geq 1)\mathcal{N}(|\phi|,x,t) \right|\leq C  t^{-\beta_0}, \quad\forall t\geq 1. \label{N-decay}\end{equation}

(H2')\label{H2'} Uniform decay of the interaction term: there exists $\alpha=\alpha(\phi)\in (\frac13, 1)$, such that for any $F(\lambda)$ being a smooth characteristic function of the interval $[1,+\infty)$, there exist constants $\beta_0=\beta_0(\alpha, F,\phi)>1$ and $C=C( \alpha, F,\phi)$ such that
\begin{equation} \left|F(\frac{|x|}{t^\alpha}\geq 1)\mathcal{N}(|\phi|,x,t) \right|\leq C  t^{-\beta_0}, \quad\forall t\geq 1. \label{N-decay}\end{equation}


(H3) The interaction $\mathcal{N}(|\phi|, x, t)$ is analytic in $\phi$, and $\bar{\phi}$.  If there is potential $V(x)$ or $V(t,x)$, we assume that they are regular up to $N_0$ ($N_0\in \mathbb{Z}^+$ or $N_0=\infty)$ and 
\begin{equation}(x\cdot \nabla)^N V(x,t) \in L^{\infty}, \quad   N=0,1,...N_0\label{regularity-V}\end{equation}


Typical examples of interactions  include
\begin{equation}\mathcal{N}(|\phi|,x,t)\phi=a  |\phi|^{p}\phi - b \frac{|\phi|^{m}\phi}{1+|\phi|^{m-n}} +V(|x|,t)\phi+W(|x|,t)f(|\phi|)\phi \label{ex-N}\end{equation}
with  $|V(x, t)|\leq (1+|x|)^{-q}$ for all $t\geq 0.$  Here 
$a,b\in \mathbb{R}^+$,  $p\in (\frac{4}{3}, 4)$,   $m>\frac43, n<4, $  $q>1.$    Notice that the first term is the defocusing power type nonlinearity (energy-subcritical and mass supercritical). The second term is the focusing  saturated nonlinearity which is energy subcritical  for $|\phi|\gg 1$, and mass supercritical for $|\phi|\ll 1$. The third term is linear with radially symmetric time-dependent potential.
Using the radial Sobolev embedding and the global $H^1$ bound \eqref{global-bound}, we see that assumption \eqref{N-decay} is verified for general nonlinearity and potential with a suitable choice of the parameters.

The above conditions imply that in the monomial case we only consider the inter-critical cases. For some detailed properties of the solution further conditions are needed, in particular higher-order decay for large distances.







 When we take $\alpha\in (\frac13, 1)$, we need $p\geq 3.$
By a bootstrap argument it extends to $p\geq 1.$
If we take $\alpha\in (\frac12, 1)$, we need $p\geq 2;$  (nice case because it includes the cubic NLS)
and by bootstrap it extends to less than $1.$
Bootstrap here refers to the observation that estimates on the nonlinear term can be improved by iterating the starting estimate.
This is not possible with an interaction term that is an explicit function of $x,t.$


The potential term $V$ should satisfy:
$$
|V(|x|,t)|\leq C(1+|x|)^{-q}, \text{for all}\, t.
$$
 If we take $\alpha\in (\frac13, 1)$, we need $q\geq 3.$
 We will stick to this assumption, although it is only needed for the last part of the work, where we prove the localization and regularity properties of the localized part of the solution.
 We also require that $V$ is bounded and sufficiently regular.

  For some refined estimates on the structure of the localized part of the solution, we need
 that for all positive integers $N,$ $|(x\cdot\nabla_x)^N V(x,t)| \lesssim 1$ uniformly in $t.$

 If we take $\alpha\in (\frac12, 1)$, we need $q\geq 2.$. This is an interesting case, since the inverse square is a critical case. 
 But this case will be a borderline to some properties of the localized solution. It should be noted that the regularity of the localized part is sensitive to the decay rate of the interaction. For example, the eigenfunctions of the hydrogen atom are not smooth, since they depend on $|x|$ but not $|x|^2.$





\medskip

{\bf Comment about the models}

 In higher dimensions ($\R^d,  d>1$), non-linear monomial terms (in the supercritical mass case) lead to solutions that are typically one channel.
 This is because the generic solution either blows up in a finite time or disperses. The possible solitons are unstable.
 This is why in \cite{RSS} a class of nonlinearities which are {\bf saturated} (See e.g. \cite{Segev} and cited references) were introduced and studied.
 In general, in the physics literature, saturated nonlinearities are relevant for large data in higher dimensions \cite{Su-Su,Segev,Malomed}.

 In the models we consider, the localized solutions are {\bf coherent structures}, which may be more general than solitons.

 We also introduce linear time-dependent interaction terms. These models, even at the linear level appear in a fundamental way in many situations, like quantum control, atomic + radiation models, and linearized theories.
 Yet, the dispersive theory for time-dependent potentials has been limited so far.

\medskip

Our main result is the following.
\begin{thm}
Let $\phi(t)$ be a global solution to equation \eqref{Main-eq} satisfying assumption (H1)(H2), then we have the following asymptotic decomposition
\begin{equation}
\lim_{t\rightarrow +\infty}\|\phi(t)- e^{i\Delta t}\Omega^*_{f}\phi_0 - \pwl(t)\|_{H^1(\R^3)}=0. \label{Main-result}
\end{equation}
Here {\bf $\Omega^*_f$ is the bounded nonlinear scattering wave operator}, mapping the initial data to the asymptotic free wave;
 $\phi_{wl}$ is the   weakly localized part of the solution with the following properties
\begin{enumerate}
\item It is  localized in the region $|x|\leq t^{\frac12}$,  in the following sense
\begin{equation}(\pwl, |x|\pwl )\lesssim t^{\frac12}.\end{equation}
\item If we further have assumption (H3), then $\pwb$ is smooth, and for $k\geq 1$, \begin{equation} \|(x\cdot\nabla_x)^k \pwl\| _{L^2_x}\lesssim 1, \quad k\leq N_0\end{equation}
here more assumptions on the reularity of the interaction terms enter. it is true for even power type, saturated but analytic. 
The cubic NLS is covered but is a borderline case for this part of the results.

\item If the solution $\phi(t)$ is time periodic, then \begin{equation} \|x \pwl\|_{L^2_x} \lesssim 1.\end{equation}
\end{enumerate}
All estimates hold uniformly in time for $t\geq 0.$
\end{thm}
\begin{rm}
In the special case where the interaction term is a time-independent potential,
one gets a new direct proof of Asymptotic Completeness, see \cite{Li-Sof1}.
Radial symmetry is not assumed in this case.
\end{rm}

  To explain our result, let us first review the background of this problem.

One of the most interesting phenomena in the study of nonlinear dispersive PDEs is the existence of  \textit{coherent structures},  which include \textit{solitons, breathers, kinks, black-holes, vortices...}. They are solutions to nonlinear PDEs that remain spatially localized for all time. These special solutions arise as an outcome of balancing between the linear dispersion and nonlinear attractive effect; they seem to be a universal phenomenon in many physical systems such as fluids, plasma, string theory, supergravity, etc.~\cite{DP, Manton, SSNLS, RR, Yang}.

Coherent structures usually come with remarkable stability with respect to small perturbation and collision, see~\cite{Tao-soliton, Martel5} and references therein.
Moreover, they  play a fundamental role in understanding the long-time dynamics for general solutions.  In fact, it is conjectured that the generic asymptotic states are given by coherent structures that move independently (freely) and free radiation~\cite{Soffer}. This statement is called \textit{asymptotic completeness}, sometimes also goes by the name \textit{soliton resolution conjecture}~\cite{Tao-soliton}. It is one of the most challenging and exciting topics in dispersive equations.

 Historically, the soliton resolution was shown only for integrable equations, for data with sufficient regularity and decay. See earlier results for KdV in~\cite{KdVresolution}, mKdV in~\cite{mKdVresolution, mKdVresolution-2}, and heuristics about 1d cubic NLS in~\cite{NLSresolution-1, NLSresolution-2}   using the inverse
scattering approach.    Recently, there is a collection of works revisiting long-time dynamics for integrable models, with the goal to construct solutions with (optimal) rough data.
 These include the rigorous proof of soliton resolution for 1d-focusing NLS~\cite{NLSresolution}, derivative NLS~\cite{dNLS},  mKdV~\cite{mKdV} and sine-Gordon~\cite{SG}.  It is worth noticing that the result of mKdV and sine-Gordon also allows emergence of topological solitons such as breathers, kinks and anti-kinks.
Despite great success in integrable models, soliton resolution remains largely open for general dispersive equations.  In fact, there are very few PDE tools that can distinguish generic data from general data. Thus, the presence of a few exotic solutions that do not resolve into solitons and radiation seems to prevent us from tackling all the other cases.

An exciting progress appeared in the remarkable work of Duyckaerts–Kenig–Merle~\cite{DKM2, DKM}, where they   classified the large energy dynamics for the 3d energy critical focusing wave equation with radial data, first for a well–chosen sequence of times, and then for general times.  The key ingredients in their argument are the concentration compactness (in particular profile decomposition) and the `channel of energy argument.  These works break the ice for understanding soliton resolution for wave-type equations and stimulate an explosion of related results.  For radial NLW in all odd dimensions, ~\cite{DKM-o1, DKM-o2, DKM-o3} established the resolution for the general time. For equivariant wave maps $\R^{2+1}\rightarrow \mathbb{S}^2$, \cite{DKMM} proved the resolution for the equivariant class $k=1$ and \cite{JL} proved it for all equivariant classes.
 For radial NLW in even dimensions and for non-radial NLW in dimension $3\leq d\leq 5$,  \cite{4d, JK, DJKM} established the resolution along a well-chosen sequence of times. Since there is a blow up, this type of soliton resolution is expected to be unstable. In fact, \cite{DKM} conjectured that  for focusing energy critical wave  equation, the collection of  data where such resolution holds is the   boundary of two open sets. Nevertheless, the powerful ideas developed by  Duyckaerts–Kenig–Merle can be implemented to show stable soliton resolution for other models~\cite{Exterior-1, Exterior-2, CR1, CR2} when blow up does not happen. 

Now we turn our attention to nonlinear Schr\"{o}dinger equation. To our knowledge, the only attempts in this direction are the works of Tao~\cite{Tao04, Tao07, Tao08}. In~\cite{Tao04}, by using a pseudo-differential partition of a wave into incoming and outgoing waves, combined with the Duhamel formula and harmonic analysis tools, Tao showed that global radial solution to 3d cubic focusing NLS will decouple into a radiation term plus a ‘weakly bounded’ component which enjoys various additional decay and regularity properties, with error going to zero in $\dot{H}^1$ norm. In~\cite{Tao07},  he considered the NLS equation with non-linearity of the power type (supercritical mass and subcritical energy) and proved the existence of a compact attractor in $H^1$.  In~\cite{Tao08}, for $d\geq 11$ with defocusing nonlinearity and focusing potential,  this was further improved to be a global compact attractor, i.e. the compact set does not depend on the energy bound of the solution.   
The related result in this direction is the work of Roy~\cite{Roy} on fourth-order NLS. Recently, there is a complete resolution result by Kim-Kwon-Oh~\cite{CSS} on the self-dual Chern-Simons-Schr\"odinger equation in the equivariant setting. 

Our work is inspired by~\cite{Tao04, Tao07, Tao08}. In particular, we try to attack the general problem but with a different method, which originates from the study of the quantum $N$-body problem.

 In the resolution of asymptotic completeness for the quantum problem $N$-body~\cite{SSAnnals,  SSInvention, SSDuke, SSjams}, Sigal and the second author introduced the idea of {\bf phase-space propagation observables(PROB)}, which  have locally positive commutators with the Hamiltonian $N$-body.
They produce
various propagation estimates that help to control asymptotic propagation into possible scattering channels. 

For a non-linear problem or a time-dependent potential case, we face new difficulties.
 It is impossible to exclude an arbitrarily small frequency from scattering as in the linear problem. The interaction terms can create another channel of asymptotic behavior, by focusing at frequency zero. Hence we need to perform a \emph{second microlocalization}.  By carefully breaking the phase space into different regions, the microlocal version PROB still has a positive commutator with $iH_0=-i\Delta$.
 For this, it is necessary to introduce several new classes of PROB, which are adapted to nonlinear interaction terms, zero frequency modes, and solutions which are spreading in space in a non-ballistic (non-free) way.

   We could then get new propagation estimates that control the solution at different phase-space regions. We prove that channel wave operators exist in the strong $L^2$ sense by Cook's method.   For the left-over part, we show that it is regular using minimal/maximal velocity bound~\cite{MV, SSpreprint} and Duhamel formula.

  [Furthermore, by building a new PROB and proving {\bf exterior propagation estimates}, we are able to show that the expectation of the dilation operator is bounded on the weakly localized part of the solution.
  These estimates rely on the decay of the interaction at large distances. We then use PROBs that are localized in the regions where $|x|\geq t^{\alpha}$ and microlocalize on incoming/outgoing waves. In this way one can get new a-priory estimates that hold, in particular in the transition region $|x|~t^{\alpha}.$

 \emph{ The endgame is a sharp localization in phase-space and time of the various parts of the solution, including the localized part.}
The departure point in this work is the application of phase-space methods to the localized part of the solution as well.

The emerging picture is that one can use propagation estimates to exclude from the asymptotic state all regions, except a thin set.
This thin set, the {\bf propagation set}, is seen to be supported on self-similar functions. These are functions where up to a phase are functions of $t,|x|/t^{\alpha}$ for some $\alpha.$ The free channel is characterized by the set $x/2t\sim p$ and the weakly localized states by $x\sim t^{\alpha}; p\sim t^{-\alpha}; \quad\quad 0\leq \alpha\leq 1/2.$ Here $p$ stands for the momentum operator.

These unusual microlocalization properties of the weakly localized states allow a detailed understanding of the asymptotic behavior of the equation for all data and channels.
In particular, the microlocal properties of the weakly localized states point to {\bf self-similar} solutions as they expand.]

\begin{rem}If the soliton resolution holds in $H^1$,  then we must have a global $H^1$ bound (\ref{global-bound}). So it is natural to impose  such assumption in our problem.  A similar bound was also imposed in the works of Tao~\cite{Tao04, Tao07, Tao08}, and on the energy critical wave equations~\cite{cote, DJKM, DKM1, DKM,  DKM3, DKM-o3,DKMM}.   In the latter case, it involves the scale-invariant critical norm.   We also remark that such an a priori bound is achievable for all solutions if we are able to remove possible blow-up scenarios~\cite{NLWpotential, Exterior-1,Exterior-2, RSS}. In particular, saturated non-linearity~\cite{RSS} was chosen so that any initial data with finite energy will lead to a global solution satisfying the global $H^1$ bound.

However, it is not clear to what extent there are stable multichannel solutions if the blow-up channels are open.
\end{rem}


\begin{rem}  We will call the decomposition (\ref{Main-result}) \textit{generic asymptotic completeness} to distinguish it from the case where we can explicitly characterize the bound states.
\end{rem}

\section{Notations and Preliminaries}

\subsection{Notations}   We write $X\lesssim Y$,  $Y\gtrsim X$  or $X=O(Y)$ to indicate $X\leq CY$ for
some constant $C>0$.  If $C$ depends upon some additional parameters, we will indicate them with subscripts.

 For any  interval $\Omega\subset \R$,  $F(\lambda\in\Omega)$ is   the smooth characteristic function of $\Omega$. 
In practice, we choose $F(\lambda\geq 1)$ to be a smooth increasing function, $F(\lambda)=0$ for $\lambda\leq \frac12$, $F(\lambda)=1$ for $\lambda\geq 1$.  For $b>4a>0$, we write  $F(\lambda \geq a)=F(\lambda/a\geq 1),$
 $F(\lambda \leq a) =1-F(\lambda\geq a)$, and $F(a\leq \lambda \leq b)= F(\lambda\geq a)-F(\lambda\geq b)$.

$P_N, P_{\leq N}$ are the standard Littlewood-Paley operators, i.e.  denote $\varphi$ as a  smooth characteristic  function of $[-1,1]$,  then $\mathcal{F}(P_{\leq N}f) =\varphi(\frac{|\xi|}{N})\hat{f}$ and   $P_N=P_{\leq N}-P_{\leq \frac{N}{2}}$.

By $\jx$, we mean a smooth function matching $|x|$ at large distance, and constant at origin. In practice, it is enough to  use the following construction.
Consider $\alpha\in C^\infty, \alpha\geq 0,  supp\, \alpha=(1,2) $ and $ \int_1^2\alpha(r)dr=1$.   Now define $\beta(r)$  on $\mathbb{R}^+$ as follows
\begin{equation}\beta(r)=\left\{\begin{aligned} r, &\hspace{0.3cm} r \geq 2, \\\int_1^r\int_1^s  \alpha (t)dt ds +C, &\hspace{0.3cm} 1\leq r \leq 2,
\\
C, &  \hspace{0.3cm}0\leq r\leq 1.  \end{aligned}\right. \label{Beta}\end{equation}
Here the constant $C$ is chosen so that $\beta$ is smooth.
Then $\beta\in C^\infty(\R)$ is a positive function and $\beta', \beta''\geq 0$.   Then we take $\jx:=g(x)=\beta(|x|). $

It is easy to verify the following identities\begin{align}  g_{ij} = \left\{\begin{aligned} \frac{\delta_{ij}}{|x|} -\frac{x_ix_j}{|x|^3}, &\hspace{0.5cm} |x|\geq 2, \\   (\frac{\delta_{ij}}{|x|} -\frac{x_ix_j}{|x|^3}) \beta' +\frac{x_ix_j}{|x|^2} \beta'', &\hspace{0.5cm}1\leq |x|\leq 2,  \\0, & \hspace{0.5cm} 0\leq |x|\leq 1 . \end{aligned}\right.\hspace{1cm}  \Delta g = \left\{\begin{aligned} \frac{2}{|x|}, &\hspace{0.5cm} |x|\geq 2, \\ \frac{2}{|x|} \beta' + \beta'' ,& \hspace{0.5cm}1\leq |x|\leq 2,   \\0, & \hspace{0.5cm} 0\leq |x|\leq 1.  \end{aligned}\right.  \label{delta-g}
 \end{align}
\begin{align}
\Delta^2 g= \left\{\begin{aligned}0 , &\hspace{0.5cm} |x|\geq 2 \mbox{ or } 0\leq |x| \leq 1, \\ \frac{4}{|x|} \beta^{(3)} +\beta^{(4)},& \hspace{0.5cm}1\leq |x|\leq 2 . \end{aligned}\right. \label{delta2-g}
\end{align}
We use $(\cdot,\cdot)$ for standard inner product on $L^2$, i.e. $(f,g)=\int_{\R^3} g\bar{f} dx$.   For a self adjoint operator $A$,  $D(A)$ is the domain of $A$,  and $\|A\|$ is the operator norm of $A$.

For any $\phi$ or $\phi(t)\in D(A)$,  the {\bf  expectation of $A$ }
is defined as  $$\la A \ra_t = (A(t)\phi(t),\phi(t)).$$  The subscript will be ignored if it is clear from the context,  and we simply write $\la A\ra$.

$A=O(\la x\ra^{-a}s^{-b})$ means that $\forall \phi, \psi\in L^2$,  $|( \jx^a A\phi, \psi ) |\lesssim  s^{-b}\|\phi\|_{L^2}\|\psi\|_{L^2}$,

$A=O_1(\la x\ra^{-a}s^{-b})$ means that $\forall \phi, \psi\in H^1$,  $|( \jx^a A\phi, \psi )|\lesssim  s^{-b}\|\phi\|_{H^1}\|\psi\|_{H^1}$.

Here $\phi, \psi$ are chosen so that the quadratic form $(\jx^a A\phi, \psi)$ is well-defined, and the constant $C$ doesn't depend on $\phi, \psi$.

Throughout the paper,  $\textbf{p}=-i\nabla$,  and we will use the following operators
\begin{align}
\gamma_0=&\frac{1}{2}( \frac{x}{|x|} \cdot \bp + \bp \cdot \frac{x}{|x|} ),\label{gamma0}\\
\gamma=&\frac{1}{2}( \nabla \jx \cdot \bp + \bp \cdot \nabla \jx ),\label{gamma}\\
\A= &\frac{1}{2}(x\cdot \bp + \bp\cdot x).  \label{dilation}
\end{align}
Notice that $\gamma_0$ is the multiplier used to obtain Morawetz estimate.  $\gamma$ is a mollification of $\gamma_0$ near the origin; as we will see  in section~\ref{sec:del-r-commutator}, $[-i\Delta, \gamma]$ is compactly supported.  $\A$ is the \textbf{dilation generator}. 


 \subsection{Vector flow of $\gamma$.}
Let us  study the flow of $e^{i a \gamma}, a\in \R$.   In fact,  recall $g(x)=\jx,$ and consider function $y(a,x):\R\times \R^n\rightarrow \R^n$,  such that
\begin{equation}\frac{d}{da}y(a,x) =  (\nabla g)(y(a,x))=\beta'(|y|)\frac{y}{|y|}, \hspace{1cm}  y(0,x)=x.\end{equation}
Since $y(a,Rx)=Ry(a,x)$ for any rotation $R\in SO(3)$,  we have $|y(a,x)|$ is a radial function and denote it as $z(a,r)$. Now $z(a,r)$ satisfies
\begin{equation}\frac{d}{da}z(a,r) = \beta'(z), \hspace{1cm}  z(0,r)=r. \label{eq-z}\end{equation}

For $\lambda>1$, we define
  $B(\lambda)=\int_2^\lambda \frac{1}{\beta'(s)} ds$, then  $B(\lambda)$ is   strictly increasing   and\footnote{Here we used the fact $0< \beta'(s) < c_0  (s-1)^2, s\rightarrow 1+$ for some constant $c_0$. Also the role of $1$ is
not  important, we can replace it by $\inf\{supp \alpha\}$. }  \be  \lim_{\lambda\rightarrow 1+}B(\lambda)=-\infty,\hspace{1cm}\lim_{\lambda\rightarrow +\infty}B(\lambda)=+\infty. \ee
Now we have the solution to equation (\ref{eq-z})
\be z(a,r) = B^{-1}(B(r)+a), \hspace{0.2cm} r>1; \hspace{1cm}  z(a,r)=r, \hspace{0.2cm}r\leq 1.\ee
Hence we have the following description about $z(a,r)$  when $r>1$,
\be \lim_{a\rightarrow -\infty}z(a,r)=1, \hspace{1cm}\lim_{a\rightarrow +\infty}z(a,r)=+\infty.\ee
As $a$ goes from $-\infty$ to $+\infty$,  $z(a,r)$ increases from $1$ till $z(a_0,r)=2$, and we have the explicit formula $z(a,r)=a+2-a_0$ for $a>a_0.$  Here $a_0=-B(r)$. In particular notice if $r\geq 2$, we have $a_0=2-r$ and $z(a,r)=a+r, $ for $a\geq 0.$

 Now we claim that the mapping $f\rightarrow U(a)f:=f(y(a,x))|det(\frac{dy(a,x)}{dx})|^{\frac12}$, gives the flow $e^{ia\gamma}f$.
  In fact, notice that operators $\left(U(a)\right)_{a\in\R}$ are unitary in $L^2$.  Stone's theorem guarantees that $U(a)=e^{iaG}$ such that
 \be G f=-i\lim_{a\rightarrow 0}\frac{U(a)-U(0)}{a} f=\gamma f\ee
 The second equality follows by direct computation, noticing that
 \be \frac{d }{da}  det(\frac{dy(a,x)}{dx})=\Delta g(y(a,x))det(\frac{dy(a,x)}{dx}), \hspace{1cm}  det(\frac{dy(0,x)}{dx})=1\ee
 which implies
 $\left(\frac{d}{da} det(\frac{dy(a,x)}{dx})\right)|_{a=0}=\Delta g(x)$   This also implies $det(\frac{dy(a,x)}{dx})\not=0, \forall a\in \R, x\in \R^3$. This is because
 \begin{enumerate}
 \item If$|x|\leq 1$, then $|y(a,x)|=|x|\leq 1$, hence $\Delta g(y(a,x))=0$, and $det(\frac{dy(a,x)}{dx})=det(\frac{dy(0,x)}{dx}))=1,\forall a\in \R$.
 \item If $|x|>1$, then $|y(a,x)|> 1, \forall a$.  And $0< \Delta g(y(a,x) \leq C$ for some constant $C$.  Hence $det(\frac{dy(a,x)}{dx}) = e^{\int_0^a \Delta g(y(s,x))ds}\not=0, \forall a\in \R.$
 \end{enumerate}

\begin{rem}\label{support}{Now we discuss the support of functions under the flow $e^{ia\gamma}$. Given radial function $f$ such that $supp f=\Omega\subset \R^3$. Then $supp\{e^{ia\gamma}f\} \subset \{x\in \R^3,  y(a,x)\in \Omega\}$.     Since $|y(a,x)|$ increases with respect to $a$, we have the following heuristic statement: for $a>0$, $e^{ia\gamma}$ moves the support of a function closer to the origin.

In particular, if $\Omega\subset\{|x|\geq 2\}$.  Then for $a<0$,  $ supp\{e^{ia\gamma}f\} \subset \{|x|\geq 2\}$;  For $a > 0$, then $ supp\{e^{ia\gamma}f\} \subset \{|x|\geq \max\{2-a, 1\}\}$. This follows from $|y(a,x)|\leq |x|+a$.}
\end{rem}

The flow properties generated by the operators $\gamma$ and $A$, are relevant in general to understanding and estimating products and commutators of operators of the type
$$
F_1(x/M^{\alpha} \sim1)F_2(M^{\beta}\gamma)F_1(x/M^{\alpha} \sim1)
$$
and similar expressions, that are used to construct PROB (propagation Observables) and to microlocalize.


\subsection{ $\Delta$ and $\gamma$  in exterior region}
In the radial case, we can rewrite $\gamma$  and $\Delta$ using shperical coordinates.
\begin{align}
\gamma  = \frac{1}{i} ( a(r)\partial_r + b(r)),\hspace{1cm}
\Delta =\partial_r^2 +\frac{2}{r}\partial_r,
\end{align}
 with $ a(r)= \beta'(r), b(r) =\frac12 (\Delta g)(r) $. Notice that $a(r)=1, b(r)=\frac{1}{r}$ for $r\geq 2.$

 Direct computation gives
 \begin{align}-\gamma^2 =  a(r)^2 \partial_r^2 +(a(r)a'(r)+2a(r)b(r))\partial_r+ a(r)b'(r)+ b(r)^2,
  \end{align}
hence we get for $|x|=r>2,$
\begin{equation}-\Delta = \bp^2=\gamma^2 , \label{gamma-delta}\end{equation}
when applying to radial functions located in exterior region $|x|\geq 2$.

\subsection{The commutator $[-i\Delta, \gamma]$}\label{sec:del-r-commutator}
For a  $C^2$   function $h$ and  operator \be A_h= \frac{1}{2}[(\nabla h)\cdot  \bp +\bp \cdot (\nabla h)] =-i[\nabla h\cdot \nabla +\frac12 \Delta h],\ee  by direct computation we get\footnote{Here we use the convention that repeated indices are summed from 1 to $3$.}
\begin{align}
[-i\Delta, h]= & 2A_h,\\
[-i \Delta, A_h] = & -\frac12 [4 \nabla_{j}\cdot h_{jk}\cdot\nabla_k +\Delta^2 h]. \label{delta-A}
\end{align}
If $h=\frac12|x|^2$, then
\begin{align}
[-i\Delta, \frac12|x|^2]=A,  \hspace{1cm} [-i\Delta, A]= -2\Delta.
\end{align}
If  $h=\jx$, we get $[-i\Delta, \jx]=2\gamma$ and
\begin{align}\la [-i\Delta, \gamma] \ra =& 2  \int_{|x|\geq 2} \frac{1}{|x|} (|\nabla \phi|^2 -|\frac{x}{|x|}\cdot\nabla \phi|^2) dx \notag\\
& +2  \int_{1\leq |x|\leq 2} \frac{\beta'(|x|)}{|x|} (|\nabla \phi|^2 -|\frac{x}{|x|}\cdot\nabla \phi|^2)  + \frac{|x\cdot \nabla \phi|^2}{|x|^2}  \beta''(|x|)dx \notag\\
&-\frac12  \int_{1\leq |x|\leq 2}  \left( \frac{4}{|x|}\beta^{(3)}(|x|) +\beta^{(4)}(|x|) \right) |\phi|^2 dx   \label{D-gamma}\end{align}
 When restricted to   radial functions, $[-i\Delta, \gamma]$ is   supported in $[1,2]$ in the sense that
\begin{align}
[-i\Delta, \gamma]f=0,  \hspace{1cm} f(x)=f(|x|),  \text{ supp}\{f\}\cap \{|x|\in [1,2]\}=\emptyset. \label{Delta-r-support}
\end{align}

\subsection{Basic Estimates}

\begin{prop}[Strichartz estimate~\cite{Keel-Tao}] $(i\partial_t +\Delta )\psi=f$, then
\be \|\psi\|_{L^2_tL^6_x \cap L^\infty_t L^2_x } \lesssim \|\psi_0\|_{L^2} +\|f\|_{L^2_tL^{\frac65}_x +L^1_tL^2_x} \label{strichartz}\ee
\end{prop}

\begin{prop}[Radial Sobolev Embedding]  If  $f\in H^1_{rad}(\R^3)$, then we have
\be
|f(r)|\lesssim r^{-1} \|f\|_{H^1}, \hspace{1cm}r\geq 1.
\ee
\be
|f(r)|\lesssim r^{-1/2} \|f\|_{H^1}, \hspace{1cm}r\leq 1. \label{radialSob}
\ee

\end{prop}
\begin{prop}[Hardy's inequality~\cite{Frank, Herbst}]For $s>0$ and $1<p<\frac{n}{s}$, we have
\begin{equation}\|\frac{f(x)}{|x|^s}\|_{L^p(\R^n)}\lesssim \||\nabla|^s f\|_{L^p(\R^n)}. \label{Hardy}\end{equation}
 \end{prop}
 So  $|x|^{-s}|\bp|^{-s}$ and $|\bp|^{-s} |x|^{-s}$ are bounded operators from $L^2(\R^3)$ to itself for $s<\frac32.$
 We should note that to use it for $L^p, p\not=2$, then  $|x|^{-s}|\bp|^{-s}: L^p\rightarrow L^p$ and $|\bp|^{-s} |x|^{-s}: L^{p'}\rightarrow L^{p'}$ for  $p\in (1, \frac{n}{s})$.


 \begin{prop} \label{prop:MVB} Let $\phi(t)=e^{it\Delta}\phi_0$ be the free wave with initial data $\phi_0\in L^2$. Assume $\hat{\phi}_0\in C_0^\infty$
 \begin{enumerate}
 \item(Minimal Velocity Bound)
 If $supp\,  \hat{\phi}_0 \subset \{|\xi|\geq v\}$, then for any $v_1<v$, and any $m$,
  \be
\int_{|x|\leq v_1t} |\phi(t,x)|^2dx =O(t^{-m}), \hspace{1cm} t\rightarrow +\infty. \label{MinVB}
 \ee
 \item(Maximal velocity bound) If $supp\,  \hat{\phi}_0 \subset \{|\xi|\leq v\}$, then for any $v_1>v$ and any $m$,
   \be
\int_{|x|\geq v_1t} |\phi(t,x)|^2dx =O(t^{-m}), \hspace{1cm} t\rightarrow +\infty. \label{MaxVB}
 \ee
 \end{enumerate}
  \end{prop}
We choose to present the minimal/maximal velocity bound just for free solutions,  which follow by the method of nonstationary phase.    For general Hamiltonian $H=-\Delta +V(x)$, such type of results first appeared in~\cite{MV, SSpreprint}, where the condition of Fourier support is replaced by $E_{\Omega}i[H, A]E_\Omega \geq v^2 E_\Omega$ for some interval $\Omega\subset \R$, and $E_{\Omega}$ is the corresponding spectral projection of $H$.

\section{Phase space operators and Commutator formula}

\subsection{Phase space calculus} Now we introduce functions of self-adjoint operators and prove commutator expansion formulas.  The idea of phase-space operators was introduced and used in the proof of asymptotic completeness for N-body problems~\cite{SSAnnals, SSInvention,  SSDuke, SSjams}.   Here we adopt most of the notation and lemmas of ~\cite{SSAnnals}.

Given a self-adjoint operator $A$, and some function $f:\R\rightarrow \R$,
we can define the operator $f(A)$  using Helffer-Sj\"{o}strand formula,
\begin{equation}f(A) = -\frac{1}{2\pi}\int_{\R^2} (z-A)^{-1} \partial_{\bar{z}}\tilde{f}(z)dxdy,\label{HS-formula}\end{equation}
here $z=x+iy$, and  $\tilde{f}$ is an almost analytic extension of  $f$ to $\mathbb{C}$, in the sense that $\partial_{\bar{z}}\tilde{f}(z) =\mathcal{O}(|\Im z|^m), \Im z\sim 0. $    For
the convenience of commutator estimate, we also use the representation

\begin{equation}f(A)=\int_{\R}\hat{f}(s)e^{iAs}ds.\label{rep-formula}\end{equation}

Initially, this representation only is valid $f$ with $\hat{f}\in L^1$.  In the following, we will prove lemmas first for functions in Schwartz class using this representation, and then extend the obtained results to functions in $\bb$\footnote{In the class $\bb$, $\hat{f}$ means  the distributional Fourier transform of $f$. Notice that if $f=f(\lambda\geq 1)$, then $\hat{f}$ is the product of a Schwartz function with $\xi^{-1}$.}
\be \mathcal{B}_n:= \{f\in C_b^\infty(\R)\left| \int_{\R}|\hat{f}(s)| |s|^{k} ds <\infty, \hspace{0.3cm} \mbox{ for } 1\leq k\leq n\}\right.\ee
Here $C_b^\infty$ represents class of smooth and bounded functions.   We equip $\bb$ with the topology induced by $\|f\|_{\bb}=\sum_{k=1}^n  \int_{\R}|\hat{f}(s)| |s|^{k} ds$, and therefore the Schwartz class $\mathcal{S}(\R)$ is a dense subset.
 Typical examples in $\bb$ are the smooth characteristic functions for intervals in $\R$.

 We denote $ad_A^{(k)}(B)$ to   be \be ad_A^{(k)}(B)=[ad_A^{(k-1)}(B), A] = [[B,A],A],\ldots, A],\ee
 which are initially defined as forms on $D(A)\cap D(B)$.
  Then the commutator expansion formula is the following.
\begin{lem}~\label{BFA-lemma} Let $A,B$ be self-adjoint operators on the same Hilbert Space $\mathcal{H}$. Assume $D(A)\cap D(B)$ is dense in $\mathcal{H}$, and
$ad_A^{(k)}(B)$ extends to a bounded operator for all $1\leq k\leq n$.

For any function $f\in \bb$, let $[B, f(A)]$ be defined as a form on $D(A^n)$. Then
\begin{align}
[B, f(A)]= &\sum_{k=1}^{n-1} \frac{1}{k!}f^{(k)}(A)ad_A^{(k)}(B) +R_n(f), \label{BFA}\\
[B, f(A)]=& \sum_{k=1}^{n-1} \frac{1}{k!}(-1)^{k-1} ad_A^{(k)}(B)f^{(k)}(A) - R^*_n(f). \label{BFA-adjoint}
\end{align}
in the form sense with remainder $R_n(f)$ satisfying
\be\|R_n(f)\|\leq c_n \|ad_A^{(n)}(B)\| \int |\hat{f}(s)| |s|^{n}ds. \label{ineq:Rn}\ee
Consequently, $[B, f(A)]$ defines an operator on $D(A^{n}).$
\end{lem}

\begin{proof} We first verify everything for function  $f \in \mathcal{S}(\R)$  and bounded operator $B$.
From the representation formula (\ref{rep-formula})
\begin{align}
[B, f(A)]
=& \int_{\R} \hat{f}(s) (B e^{iAs} - e^{iAs}B)ds \notag\\
=& \int_{\R} \hat{f}(s) e^{iAs}\int_0^s \frac{d}{du} (e^{-iAu }Be^{iAu})du\,  ds\notag \\
=& i \int_{\R} \hat{f}(s) e^{iAs}\int_0^s  e^{-iAu }[B,A]e^{iAu} du \,ds. \label{BFA-commute}
\end{align}
Then  denote $h(u) =i e^{-iAu }[B,A]e^{iAu} $ to be the bounded operator on $D(A)$, by Taylor expansion we have
\be  h(u)= \sum_{k=0}^{n-2} \frac{1}{k!} h^{(k)}(0)u^k +\int_0^u\int_0^{s_1}\ldots \int_0^{s_{n-2}} h^{(n-1)}(s_{n-1}) ds_{n-1} ds_{n-2}\ldots ds_1  \ee
$h^{(k)}(u)=i^{k+1} e^{-iAu }ad_{A}^{k+1}(B)e^{iAu}$.  So we have (\ref{BFA}) with the expression for $R_n(f)$
\begin{equation}
R_n(f)=i^n \int_{\R} \hat{f}(s) \int_0^s \int_0^{s_0} \int_0^{s_1}\ldots \int_0^{s_{n-2}} e^{iA(s-s_{n-1})}ad_A^n(B)e^{iAs_{n-1}} ds_{n-1} ds_{n-2}\ldots ds_1 ds_0 ds .\label{R-formula}
\end{equation}
which implies (\ref{ineq:Rn}) since $ad^n_A(B)$ is bounded.
(\ref{BFA-adjoint}) follows from (\ref{BFA}) by taking adjoint.
Now given $f\in \bb$
\be  A(f)   =[B, f(A)] - \sum_{k=1}^{n-1} \frac{1}{k!}f^{(k)}(A)ad_A^{(k)}(B) \ee
is well defined on $D(A^n)$. Take $f_j \in S(\R),  f_j\rightarrow f\in \bb$,  which implies~\cite{SSpreprint}
 \be \|(1+|\lambda|)^{-n} (f_j(\lambda)-f(\lambda))\|_{L^\infty}\rightarrow 0.\ee
 Here we used that the $L^{\infty}$ norm is bounded by the $L^1$ norm of the Fourier Transform of the function.
 Hence $A(f_j)\rightarrow A(f)$ in the form sense.

Now $A(f_j)=R_n(f_j)$ and $R_n(f_j)\rightarrow R_n(f)$ with $R_n(f_j)$ uniformly bounded. By passing to the limit,  we get $A(f)=R_n(f)$, which holds in the form sense on $D(A^n)$.  Both $R_n(f)$ and $\sum_{k=1}^{n-1} \frac{1}{k!}f^{(k)}(A)ad_A^{(k)}(B)$ are bounded operator on $D(A^n)$,  so   $[B, f(A)] $  defines an operator on $D(A^n)$.

The result for unbounded operator $B$ follows by approximating  $B$ by  $\frac{B}{1+\epsilon B}$ and continuity argument.
\end{proof}

\begin{rem}

(1) A more useful version of Lemma~\ref{BFA-lemma} applies to functions $f, g$ with a scaling parameter $s\in (0,\infty)$
\begin{equation}[g(B), f(s^{-1}A)]=\sum_{k=1}^{n-1} \frac{1}{k!}s^{-k}f^{(k)}(s^{-1}A)ad_A^{(k)}(g(B)) +s^{-n}R_n(f), \label{BFAs}\end{equation}
with
\be  \|R_n(f)\|\leq \tilde{c}_n \|ad_A^{(n)}(g(B))\|  \int |\hat{f}(\lambda)| |\lambda|^{n}d\lambda. \ee
and we also have  similar formulas as (\ref{BFA-adjoint}).

(2) We also record the double commutator identity, which follows from  (\ref{BFA-commute}).
 \begin{align}  [f_1(A), [f_2(A), B]] =  -  \int_{\R}  \int_{\R}   \hat{f_2}(s)   \hat{f_1}(\lambda) \int_0^s  \int_0^{\lambda}        e^{i(s+\lambda -u -v)A}  ad_A^2(B)   e^{i(v+u)A} \,du \,dv \,d\lambda       \,ds \label{fafab}
\end{align}

(3) Lemma~\ref{BFA-lemma} is usually proved using the representation formula (\ref{rep-formula}) as in~\cite{SSAnnals, SSpreprint}.   Alternatively, it can be verified using the Helffer-Sj\"{o}strand formula (\ref{HS-formula}) as in~\cite{HS}.
\end{rem}

\subsection{Commutator estimates}   We first record simple symmetrization formulas, which can be checked by direct computation.
\begin{lem} Let $A,B,C$ be self-adjoint operators. Suppose $[A,C]=0$, then we have
\begin{align}
AB^2A = & BA^2B +	[[ A,B], B]A + 	B[[ A,B], A]
\label{AB2A}\\
A^2BC^2= & (AC)B(AC) + A^2[B,C]C+ AC[A,B]C\\
A^2B + BA^2 = & 2ABA + [A, [A,B]] \label{ABBA}\\
ABC - CBA =& A[B,C] + C[A, B] \\
A^2BC^2 + C^2BA^2
= &  2(AC)B(AC) +R(A,B,C) \label{A2BC2}\end{align}
where the remainder term in (\ref{A2BC2}) involves double commutator
\be  R(A,B,C) =  A[[A,B],C]C + C[[C,B],A]A+ A [C, [C,B]]A +C [A,[A,B]]C \ee
\end{lem}
There are various representations of the remainder term in (\ref{A2BC2})
  but we choose to write it in a way such that double commutators are always squeezed by operators $A$ or $C$.  From (\ref{ABBA})(\ref{A2BC2}), we notice that symmetrization errors always involve double commutators.
\begin{rem}
The equation \ref{A2BC2} will be used repeatedly to symmetrize terms that are formally positive (at the symbolic phase space level).
The result is a positive operator plus a correction that we call $Sym$, coming from the second commutator.
This second commutator term is typically higher-order due to the scaling properties of the operators used.
\end{rem}

 \begin{lem}
 For any function  $f\in C^\infty$, such that $|f^{(\alpha)}(x)|\lesssim \jx^{K-|\alpha|}, $
 we have
 \begin{equation}[\gamma, f(\XS)] =  -i s^{-1} \beta'(r)^2 f'(\XS) = O(\jx^{K-1}s^{-1}). \label{gammaF1}\end{equation}
Let  $F_i, i=1,2$ be  smooth characteristic functions of $[1,+\infty)$,   we have
\begin{align}
[-i\Delta, F_1(\XS)] =& 2s^{-1}\sqrt{F_1'(\XS)}\gamma \sqrt{ F_1'(\XS)}=O_1(s^{-1}) \label{deltaf1} \\
[F_2(\Gt), f(x)] =&{\tau}^{-1} \sqrt{F_2'(\Gt)}[\gamma, f]\sqrt{ F_2'(\Gt)} + O(\jx^{K-3}\tau^{-3}) \label{f2f}\\
[F_1(\XS), F_2(\Gt)] =& \tau^{-1} \sqrt{F_2'(\Gt)} [F_1(\XS),\gamma]\sqrt{F_2'(\Gt)}+O(s^{-3}\tau^{-3}) \label{f1f2}
\end{align}
When applying to radial functions,  we have for $s>4$
\begin{equation}
 F_1(\XS)[-i\Delta, F_2(\Gt)]F_1(\XS)=0. \label{DeltaF2}
\end{equation}
\end{lem}

\begin{proof}  (\ref{gammaF1})  follow from direct computation.  (\ref{deltaf1}) is obtained from direct computation and symmetrization.
(\ref{f2f}) and  (\ref{f1f2}) are similar, so we prove (\ref{f1f2}) here. Using (\ref{BFAs}) together with its adjoint form, also notice (\ref{ABBA})  we get
\begin{align}
&2 [F_1(\XS), F_2(\Gt)] \\=&  \sum_{k=1}^{2}\left[ \frac{1}{k!}\tau^{-k}F_2^{(k)}  ad^{(k)}_{\gamma}(F_1)  +(-1)^k\frac{1}{k!}\tau^{-k}  ad^{(k)}_{\gamma}(F_1 )F_2^{(k)} \right] + R_3-R_3^*\\
=& \tau^{-1}\left[F_2' ad_{\gamma}(F_1)+ad_{\gamma}(F_1) F_2' \right] +\frac12 \tau^{-2} \left[F_2'' ,  ad^{(2)}_{\gamma}(F_1)\right]+ R_3-R_3^*\\
=&2\tau^{-1} \sqrt{F_2' } [F_1,\gamma]\sqrt{F_2' } + \tau^{-1}[\sqrt{F_2'},[ \sqrt{F_2'}, ad_{\gamma}(F_1)]] +\frac12 \tau^{-2} \left[F_2'' ,  ad^{(2)}_{\gamma}(F_1)\right]+ R_3-R_3^*\\
=& 2\tau^{-1} \sqrt{F_2'(\Gt)} [F_1(\XS),\gamma]\sqrt{F_2'(\Gt)}+O(s^{-3} \tau^{-3})
\end{align}
From the   third to the fourth identity, we estimate the last 3 terms by using  (\ref{BFA-commute})(\ref{R-formula})(\ref{fafab}) and reduce the problem to   estimating
\begin{equation}ad_{\gamma}^k (F_1(\XS))= O(s^{-k}),\label{adk-bound}\end{equation}
which can be checked by direct computation.

Now we verify   (\ref{DeltaF2}) only when we interpret it as an operator acting on radial functions on    both sides, i.e. $( \phi, T\psi)=0, \forall \phi, \psi\in D(T)$.  Actually, this follows directly from (\ref{gamma-delta}). However, we provide another proof using the support condition of the vector flow for $\gamma$. 

In fact, from (\ref{BFA-commute}) and (\ref{Delta-r-support}),  for radial functions $\phi, \psi$ we have
  \begin{align*}
&(\phi, F_1(\XS)[-i\Delta, F_2(\Gt)]F_1(\XS) \psi )\\= & \lp \phi(x),  F_1(\XS)\int_{\R}\int_0^{\lambda}  \hat{F}(\lambda)   e^{i(\lambda-\lambda_1) \Gt}
ad_{\gamma}(-i\Delta)  e^{i\lambda_1 \Gt}   d\lambda_1 d\lambda \,\,F_1(\XS)\psi(x) \rp\\
 =& \int_{\R}\int_0^{\lambda} \hat{F}(\lambda) \lp e^{i(\lambda_1-\lambda)\Gt}  F_1(\XS)\phi(x), ad_{\gamma}(-i\Delta) e^{i \lambda_1 \Gt} F_1(\XS)\psi(x) \rp  \,d\lambda_1 \,d\lambda
  \end{align*}
  Now let us discuss the spatial support of functions in the inner product.  From Remark~\ref{support}  we know for $s> 4$
\be   supp \{e^{i a \gamma} F_1(\XS)\psi(x)\}  \subset \{|x|>2\}, \hspace{1cm} a\leq 0. \ee
We can discuss the cases $\lambda \geq  0 $ and $\lambda<0$,  and in each case  either $e^{i(\lambda_1-\lambda)\Gt}  F_1(\XS)\phi(x)$ or  $e^{i \lambda_1 \Gt} F_1(\XS)\psi(x)$ has support different from $ad_{\gamma}(-i\Delta)$, which is located in $[1,2]$.  This finishes the proof.
\end{proof}
\begin{rem}\label{rem:boundary} By performing the commutator expansion for $[F_1(\XS), F_2(\Gt)]$
\be [F_1(\XS), F_2(\Gt)]=   \sum_{k=1}^{n-1} \frac{1}{k!\tau^k}F_2^{(k)}(\frac{\gamma}{\tau})ad_\gamma^{(k)}(F_1(\XS)) +O((s\tau)^{-n}),\ee
we see that up to error terms,
 $[F_1(\XS), F_2(\Gt)]$ is essentially supported on the boundary region of $F_1', F_2'$.
\end{rem}

 \section{ The gamma operator limit}\label{sec:gamma-limit}

\subsection{ Positive Commutator Argument}
Let us start with the general commutator argument.
Take    $\phi(t)$ to be  a solution to equation (\ref{Main-eq}), then we have
 \begin{align}
\frac{d}{dt}\left<A\right>_t
  =\la  D_HA\ra  + i([\mathcal{N}, A]\phi, \phi)
=\la  D_HA\ra +    2 \textit{Im} (A\phi, {\bf N}(\phi)) \label{DtA}
\end{align}
Here,
$ D_HA=[-i\Delta, A] +\frac{\partial A}{\partial t}$
and $D_H$  is called the \textit{Heisenberg derivative}.
Notice that if $\phi$ satisfies the free Schr\"{o}dinger equation,  or if $\phi$ is a solution to (\ref{Main-eq}) but $A=f(x,t)$, then $\frac{d}{dt}\la A\ra = \la D_HA\ra.$

Now we impose the following assumptions on the propagation observable  $A$ in $t\in [t_0, +\infty), t_0\geq 0$.
\begin{enumerate}
\item[(A1)] $\la A\ra_t$ is uniformly bounded in time. \label{AS1}
\item[(A2)] $\textit{Im}(A(t)\phi(t), {\bf N}(\phi)(t)) = O_1((t+1)^{-1-\epsilon})$ \label{AS2}
\item[(A3)]$ D_HA=\pm \sum_{k=1}^KB_k^*B_k + C$, and  $  \la C\ra_t \in L^1(dt)$\label{AS3}
\end{enumerate}
 then from (\ref{DtA}) we have  \be \sum_{k=1}^K\int_{t_0}^T\la B_k^*B_k \ra_tdt = \pm \left[\la  A\ra_T-\la A\ra_{t_0} -\int_{t_0}^T\la C \ra_t dt + \int_{t_0}^T O(t^{-1-\epsilon})dt \right].\ee
So we conclude that
\begin{equation}\sum_{k=1}^K\int_{t_0}^\infty \la B_k^*B_k\ra_t dt <\infty, \label{B-PE} \end{equation}
and by Cauchy criterion, this further implies  $\lim_{T \rightarrow \infty }\la A\ra_T$ exists.

\begin{rem}To implement the commutator argument, the key is to check assumption (A3).  It follows immediately if we verify the following assumption.
\begin{enumerate}
\item[(A4)] $ D_HA =\pm  \sum_{k=1}^KB_k^*B_k   +O_1((t+1)^{-1-\epsilon})$\label{AS4}
\end{enumerate}
In general, (A4) does not hold and we might have to run a bootstrap argument,  i.e. we first get estimate involving $\la B_k^*B_k\ra_t$ weaker than (\ref{B-PE}),  then use it to prove $  \la C\ra_t \in L^1(dt)$.  We will demonstrate it in the proof of Theorem~\ref{thm:r-limit}.
\end{rem}
\begin{defin}
An operator $A$ satisfying (A4) is called a {\bf Propagation Observable} (PROB).
\end{defin}

 \subsection{Existence of $\gamma$ limit} From now on we  take $t_0=100$, i.e. it is a large time such that $t_0^\alpha\gg 2$, for any $\alpha\in (\frac13, 1)$.  The following theorem says that $\gamma$ in the exterior domain has a limit.
  \begin{thm}[Existence of $\gamma$ limit]\label{thm:r-limit} Given any solution $\phi$ to \eqref{Main-eq} satisfying assumption (H1)(H2).
Then for any $\alpha \in (\frac13,1)$, and any smooth characteristic function $F$ of interval $[1,+\infty)$,  we have
  \begin{equation}\Gamma:= \lim_{T\rightarrow +\infty}\la F(\XT)\gamma F(\XT) \ra_{t=T} \label{gamma-limit}\end{equation}exists.
  \end{thm}
 Note that if we use the assumption (H2'), just take the $\alpha$ in the assumption, the result still holds.
  \begin{rem} If the limit $\Gamma=0$, we will call the solution \textbf{weakly localized state} (WLS), and denote it as $\phi_{wls}.$




  \end{rem}
  \begin{proof}  Let us simplify the notation by writing $F$ instead of $F(\XT)$  when it is not essential.

 Denoting $A(t)=F\gamma F$, then from global $H^1$ bound (\ref{global-bound}) on the solution, we see that assumption (A1) holds.
  By the non-linearity decay assumption (\ref{N-decay}) we have  
  \be  \left|(A(t)\phi(t), \textbf{N}(\phi)(t)) \right| = \left| (\gamma F \phi,   F\mathcal{N} \phi) \right| \lesssim t^{-\beta_0}, \hspace{1cm} t\geq 1. \ee
  So assumption (A2) holds true.

 Now we focus on analyzing $D_HA(t)$. From the
 commutator expansion formula (\ref{BFA})(\ref{BFA-adjoint}) ,  we get
\begin{align}
[-i\Delta, F\gamma F]
=&  F [-i\Delta, \gamma ]F +  [-i\Delta, F] \gamma F   + F\gamma [-i\Delta,  F] \\
=&  F [-i\Delta, \gamma ]F  +\left(\frac{1}{t^\alpha}F' [-i\Delta, \la x \ra]  + \frac{1}{2t^{2\alpha}}F''[ [-i\Delta, \la x\ra], \la x\ra]  \right)\gamma F \\ & +F\gamma \left(\frac{1}{t^\alpha} [-i\Delta, \la x\ra] F' -  \frac{1}{2t^{2\alpha}}[ [-i\Delta, \la x\ra], \la x \ra] F''\right) \\
=&\underbrace{F [-i\Delta, \gamma ]F}_{I_1}  +\underbrace{\frac{2}{t^\alpha}\left( F'\gamma^2F +F\gamma^2F' \right)}_{I_2}+\underbrace{ \frac{1}{t^{2\alpha}}[F''[ \gamma, \jx]\gamma F - F\gamma [\gamma, \jx]F''] }_{I_3}.
\end{align}
For $I_1$, we see from (\ref{Delta-r-support}) that it vanishes for $t\in [t_0, +\infty)$.

For $I_2$, we use (\ref{A2BC2}) and (\ref{gammaF1})  to get
\begin{align}
F'\gamma^2F +F\gamma^2F'  = 2 \sqrt{FF'}\gamma^2 \sqrt{FF'} + \sqrt{F_1}[\sqrt{F_2}, [\sqrt{F_3}, \gamma^2] ]\sqrt{F_4},
\end{align}
where the second term, with $F_1, F_2, F_3, F_4\in \{F,F'\}$  represents a collection of  terms of the same type, and they give $O(t^{-2\alpha})$.
Hence we conclude that
\be I_2=\frac{4}{t^\alpha} \sqrt{FF'}\gamma^2 \sqrt{FF'} + O(t^{-3\alpha}).\ee
For $I_3$, we use  (\ref{gammaF1}) to obtain
\begin{align}
 I_3= &  \frac{1}{t^{2\alpha}}[F''[ \gamma, \jx]\gamma F - F\gamma [\gamma, \jx]F''] \\
=&  \frac{1}{t^{2\alpha}}\left\{F''[ \gamma, \jx][\gamma,F] -F[\gamma, [ \gamma, \jx]F'']\right\}
= O(t^{-3\alpha}). \label{I3}
\end{align}
Hence we get for $t\geq t_0,$
 \begin{align}
 D_H(F\gamma F) =&  [-i\Delta, F\gamma F] + \frac{d}{dt}(F\gamma F) \notag \\
 =&  \frac{4}{t^\alpha}   \sqrt{FF'}\gamma^2 \sqrt{FF'}
 -\frac{\alpha}{t}\left[ F'\frac{\la x\ra}{t^\alpha}\gamma F + F\gamma F'\frac{\la x \ra }{t^\alpha}\right]+  O(t^{-3\alpha}).\label{DHFrF}
 \end{align}
Moreover, the second term is of order $O(t^{-1})$, so assumption (A3) does not hold. So we verify it using a bootstrap argument.

Take $F_1$ to be another smooth characteristic function of $[1,+\infty)$,  which will be determined later. Denoting $\tilde{A} =\frac{1}{t^\epsilon}F_1(\XT)\gamma F_1(\XT)$, with $\epsilon+\alpha<1$ and repeating the previous discussion for $\tilde{A}$, we can see that Assumptions (A1)(A2)(A4) holds true for $[t_0, +\infty)$. In fact, we only need to check the validity of (A4).
\begin{align} D_H\tilde{A} =& \frac{4}{t^{\alpha+\epsilon}}   \sqrt{F_1F_1'}\gamma^2 \sqrt{F_1F_1'}
 -\frac{\alpha}{t^{1+\epsilon}} \left[ F'\frac{\la x\ra}{t^\alpha}\gamma F + F\gamma F'\frac{\la x \ra }{t^\alpha}\right] - \frac{\epsilon}{t^{1+\epsilon}}F_1\gamma F_1 + O(t^{-3\alpha-\epsilon})\\
 =&  \frac{4}{t^{\alpha+\epsilon}}   \sqrt{F_1F_1'}\gamma^2 \sqrt{F_1F_1'}   + O_1(t^{-1-\epsilon})
 \end{align}
 Hence we conclude that \be \int_{t_0}^\infty    \frac{4}{t^{\alpha+\epsilon}} \la \sqrt{F_1F_1'}\gamma^2 \sqrt{F_1F_1'}\ra  dt<\infty,\ee
  which implies
  \begin{equation} \int_{t_0}^\infty \frac{1}{t^{\alpha+\epsilon}}\|\gamma  \sqrt{F_1F_1'} \phi\|^2_{L^2} dt<\infty. \label{F1-estimate}\end{equation}
  Let us take $F_1$ such that $\sqrt{F_1F_1'}(\lambda)=cF'(\lambda)\lambda$ where $c$ is the renormalization constant such that $F_1(+\infty)=1$.
  Now in (\ref{DHFrF}), denote $C= -\frac{\alpha}{t}\left[ F'\frac{\la x\ra}{t^\alpha}\gamma F + F\gamma F'\frac{\la x \ra }{t^\alpha}\right]$, we have
  \begin{align}
\left | \int_{t_0}^T\la C\ra_t dt\right| \lesssim  &    \int_{t_0}^T     \frac{1}{t}\left|( \gamma F'\frac{\la x\ra}{t^{\alpha}} \phi , F\phi ) \right|   dt\\
\lesssim & \left(\int_{t_0}^T \frac{1}{t^{\alpha+\epsilon}} \|\gamma F' \frac{\la x\ra}{t^\alpha}\phi\|^2_{L^2} dt\right)^{\frac12}\left( \int_{t_0}^T \frac{ \|\phi\|_{L^2}^2}{t^{2-\alpha-\epsilon}}dt\right)^\frac12
  \end{align}
  Hence assumption (A3) holds   for $A(t)=F\gamma F$, and we  conclude that
   \begin{equation}\int_{t_0}^\infty    \frac{1}{t^{\alpha}}\la \sqrt{FF'}\gamma^2 \sqrt{FF'}\ra  dt<\infty. \label{PE}\end{equation}
   and $ \lim_{T\rightarrow +\infty}\la F\gamma F\ra_{T} $ exists.
  \end{proof}

\begin{rem}
The support of the function $\sqrt{FF'}$ is essentially the same as $F'.$ So, we will sometimes use the generic notation $\tilde F'$ instead of  $\sqrt{FF'}.$
\end{rem}
\begin{rem}
The above argument is an example of what we refer to as Iterating Propagation Estimates.
\end{rem}

\subsection{Discussion of $\gamma$-limit} In this subsection, we will  obtain more information on the $\gamma-$limit.

We first recall the propagation estimate (\ref{PE}), which by \eqref{AB2A} further implies  that   for $\alpha> \frac13$  and $F$ a smooth characteristic function of $[1,+\infty)$
 \begin{equation}\int_{t_0}^\infty     \frac{1}{t^{\alpha}}\left\|\gamma F'(\XT)\phi\right\|_{L^2_x}^2  dt +   \int_{t_0}^\infty     \frac{1}{t^{\alpha}}\left\| F'(\XT)\gamma \phi\right\|_{L^2_x}^2  dt <\infty \label{PE-1}\end{equation}
 
 \begin{rem}\label{rem:PE} Even though \eqref{PE-1} is written for $F'$ with $F$ being a smooth, increasing characteristic function,  it also holds when we replace $F'$ by any bounded function $\tilde{F}$ that is supported in the compact set of $(0,\infty)$.  In fact, we can take a smooth characteristic function $G$ such that $ G' \geq  c |\tilde{F}|$, here $c$ is a renormalizing constant. Then we have 
 \begin{equation}\int_{t_0}^\infty \int_{\R^3}\frac{1}{t^\alpha}|\tilde{F}\gamma \phi|^2dxdt \leq c^{-1}\int_{t_0}^\infty \int_{\R^3}\frac{1}{t^\alpha}|G'\gamma \phi|^2dxdt<\infty \label{G'}\end{equation}
This further implies that $\int_{t_0}^\infty \int_{\R^3}\frac{1}{t^\alpha}|\gamma \tilde{F} \phi|^2dxdt<\infty$ since the difference with (\ref{G'}) is of size $O(t^{-3\alpha}).$
 
 \end{rem}

\subsubsection{Negative $\gamma$-limit} We first show that $\gamma-$limit can not be negative if the solution global.  To do that, we use a few steps.

 \begin{prop}\label{0limit-frf} For any smooth characteristic function $F(\XT), \alpha\in (\frac13,1)$,  we have
 \be   \left|\int_{t_0}^T  \la F'\gamma F'\ra dt\right| \lesssim T^\alpha, \hspace{1cm} \forall T\geq t_0,\label{integrate-boundary}
 \ee
 and \be\lim_{t\rightarrow +\infty}\la F'\gamma F'\ra_t =0.\label{0limit-boundary}
 \ee
 \end{prop}
\begin{proof}  Let us
take the characteristic function $G_1$,  with $F_1'(\lambda)=c (F')^2(\lambda)$ with $c$ being the renormalize constant such that $F_1(+\infty)=1$.
\begin{align}
 \la t^\alpha F_1(\XT)\ra_T -  \la t^\alpha F_1(\XT)\ra_{t_0}= &\int_{t_0}^T
 t^\alpha \la D_HF_1\ra   +  \alpha t^{\alpha-1}\la F_1\ra dt  \\ =&  \int_{t_0}^T 2 \la \sqrt{F_1'}\gamma \sqrt{F_1'}\ra  +O(t^{\alpha-1})dt
\end{align}
which implies (\ref{integrate-boundary}).

To prove \eqref{0limit-boundary}, we first show that $\la F'\gamma F'\ra$  has a limit as $t\rightarrow +\infty$.
In fact, $F'\gamma F'$ satisfies assumption (A1)(A2). Now we compute the Heisenberg derivative
\begin{align}
D_H(F'\gamma F')=     \frac{2}{t^\alpha}\left( F''\gamma^2F' +F'\gamma^2F'' \right) + \frac{1}{t^{2\alpha}}[F'''[ \gamma, \jx]\gamma F' - F'\gamma [\gamma, \jx]F''']
  +  \frac{d(F'\gamma F')}{dt}
\end{align}
The term in the middle is of size $O(t^{-3\alpha})$, similar as (\ref{I3}).

Using (\ref{PE-1}) we have
\be |\la  \frac{d(F'\gamma F')}{dt}\ra| =|\la \frac{\alpha}{t}(F''\frac{\jx}{t^\alpha}\gamma F' + F'\gamma \frac{\jx}{t^\alpha} F'') \ra|\leq \frac{1}{t}\|\frac{\jx}{t^\alpha} F''\phi\|_{L^2}\|\gamma F'\phi\|_{L^2}\in L^1(dt).\ee

For $ \frac{2}{t^\alpha}\left( F''\gamma^2F' +F'\gamma^2F'' \right) $, since $F''$ is not positive, we cannot symmetrize and obtain positive definite terms. 
We estimate it directly.  In fact, from \eqref{PE-1} and Remark~\ref{rem:PE}, we have 
 \begin{align}
|\la \frac{1}{t^\alpha}\left( F''\gamma^2F' +F'\gamma^2F'' \right) \ra
| \leq \frac{1}{t^\alpha}\|\gamma F'\phi\|_{L^2_x}\|\gamma F''\phi\|_{L^2_x} \in L^1(dt).
\end{align}
Hence $F'\gamma F'$ satisfies Assumption (A3) with $B=0$, and we conclude that $\lim_{t\rightarrow +\infty}\la F'\gamma F'\ra_t$ exists.

Denote $lim_{t\rightarrow +\infty}\la \sqrt{F_1'}\gamma \sqrt{F_1'}\ra =\Gamma_0$. If $\Gamma_0\not=0$, we find $T_0$ such that $t\geq T_0$, $\la \sqrt{F_1'}\gamma \sqrt{F_1'}\ra $ is sign definite and
\be |\la \sqrt{F_1'}\gamma \sqrt{F_1'}\ra| \geq C=\frac{1}{2}|\Gamma_0|.\ee
We reach contradiction with (\ref{integrate-boundary}) for $T$ large enough. Hence $\lim_{t\rightarrow +\infty}\la \sqrt{F_1'}\gamma \sqrt{F_1'}\ra_t=0$ and we finish the proof.
\end{proof}

\begin{prop}\label{x-growth}For any solution $\phi$ to equation (\ref{Main-eq}) satisfying the bound (\ref{global-bound}), we have
\begin{equation}\la |x|\ra_t\leq Ct, \hspace{1cm}\forall t\geq 1 .\label{x-bound}\end{equation}
\end{prop}
\begin{proof}
Since $\la |x|\ra \lesssim \la \jx\ra \lesssim \la |x|\ra +\|\phi\|_{L^2}^2$, it suffices to verify
\begin{equation}(\phi(t), \jx \phi(t))\leq Ct, \hspace{1cm} \text{ for } t\geq 1.\label{jxjx}\end{equation}

Take $\chi(x)$ to be a bump function near the origin.
From direct computation and the global $H^1$ bound (\ref{global-bound}), we have
\be\frac{d}{dt}\la \chi(\frac{x}{R})\jx \ra =\la [-i\Delta,  \chi(\frac{x}{R})\jx]\ra   <\infty. \ee  
So we get
\be \la \chi(\frac{x}{R})\jx \ra \leq Ct, \ee
with constant independent of $R$. Now (\ref{jxjx}) follows by taking $R\rightarrow +\infty$.
 \end{proof}

\begin{thm}\label{Non-negative} For global solutions to equation (\ref{Main-eq}) satisfying the bound (\ref{global-bound}), the $\gamma$-limit is always non-negative.
\end{thm}
\begin{proof} Proof by contradiction. Assume $\Gamma=\lim_{t\rightarrow +\infty} \la F(\XT)\gamma F(\XT)\ra_t <0$. Consider $\la F\frac{\jx}{t}F\ra$ which is uniformly bounded on $t\in [1,+\infty)$  due to (\ref{x-bound}), we have
\begin{align}
&\frac{d}{dt}\la F\frac{\jx}{t}F\ra  = \la D_H( F\frac{\jx}{t}F)\ra     \\
=& \frac{1}{t}\left[2\la\sqrt{FF'\frac{\jx}{t^\alpha}}\gamma \sqrt{FF'\frac{\jx}{t^\alpha}}\ra+2\la F \gamma F\ra  - \frac{\alpha}{t^{1-\alpha}} \la F'\frac{\jx^2}{t^{2\alpha}} F\ra - \frac{1}{t}\la  F\jx F\ra   \right]\notag
\end{align}
We notice for $t$ large enough,  in the bracket of the RHS, there are  two negative terms  $\la F \gamma F\ra$ and $- \frac{1}{t}\la  F\jx F\ra$ together with terms converging to $0$. Hence, by taking $t\geq T_0$ for large $T_0$,  we get $\frac{d}{dt}\la F\frac{\jx}{t}F\ra \leq -\frac{C}{t}$, which implies
\be \left|\la F\frac{\jx}{t}F\ra_{T} -\la F\frac{\jx}{t}F\ra_{T_0}\right| \geq C (ln T -ln T_0).\ee
By taking $T\rightarrow +\infty$, we reach the contradiction by noticing (\ref{x-bound}). \end{proof}

\subsubsection{Weakly localized state} A-priory, the definition of WLS seems to depend on the choice of $F$ and $\alpha$,  we will show that it is actually independent of these choices. Toward proving such a statement, we also obtain characterization of WLS in terms of the growth speed for $\la |x|\ra$.

 \begin{prop}\label{slowgrowth} For a given characteristic function $F$ and $\alpha\in (\frac13,1)$, assume we have the decay assumption on nonlinearity (\ref{N-decay}) with $\beta_0\geq
\frac32$.
 Let $\phi(t)$ be  a solution to equation (\ref{Main-eq}) satisfying
 \be \lim_{t\rightarrow +\infty} \la F(\XT)\gamma F(\XT)\ra =0,\ee  then we have
\begin{equation}\la |x|\ra_t\leq  C \max(t^{2-3\alpha}, t^{\alpha}), \hspace{1cm}\forall t\geq 1 .\label{refined-x-bound}\end{equation}
 \end{prop}
\begin{proof}
Let us estimate $\la  F\gamma F\ra$  on $[t_0, +\infty)$.   There is nothing to do if for all time $t_1\in [t_0,\infty)$  if  $\la  F\gamma F\ra|_{t_1}\leq0$.

Now  take a point $t_1\in [t_0, \infty)$ such that $\la  F\gamma F\ra(t_1)>0$,  using (\ref{DHFrF}) on $[t_1, T)$, we have
 \begin{equation}\la F\gamma F\ra\left|_{t=t_1}^{t=T}\right. =\int_{t_1}^T   \frac{4}{t^\alpha} \la \sqrt{FF'}\gamma^2 \sqrt{FF'}\ra   -2\alpha Re \la  F'\frac{\la x\ra}{t^{\alpha+1}} \phi , \gamma F\phi\ra +  O(t^{-3\alpha}+t^{-\beta_0})dt \label{t0-decay} \end{equation}
The second  term is controlled as follows
\begin{align}
& 2\alpha Re \la  F'\frac{\la x\ra}{t^{\alpha+1}} \phi , \gamma F\phi\ra
= 2\alpha Re\la  \sqrt{FF'}\frac{\la x\ra}{t^{\alpha+1}}\phi ,\gamma \sqrt{FF'}\phi \ra
\leq   \frac{1}{t^{\alpha}}\|\gamma \sqrt{FF'}\phi\|^2_2 +\frac{\alpha^2}{ t^{2-\alpha}}\|\phi\|^2_2
\end{align}
Using this estimate in  (\ref{t0-decay}), we have
\begin{align}
 \la F\gamma F\ra|_{t_1} +\int_{t_1}^T    \frac{1}{t^\alpha} \la \sqrt{FF'}\gamma^2 \sqrt{FF'}\ra  dt  \lesssim \la F\gamma F\ra|_{T}   +  \int_{t_1}^T  \frac{\|\phi\|_2^2   }{t^{2-\alpha}} +O(t^{-3\alpha}+t^{-\beta_0})dt
\end{align}
by letting $T\rightarrow +\infty$, we get
\begin{equation}  \la F\gamma F\ra|_{t_1} +\int_{t_1}^\infty    \frac{1}{t^\alpha} \la \sqrt{FF'}\gamma^2 \sqrt{FF'}\ra  dt  = O( t_1^{1-3\alpha} +  t_1^{-1+\alpha} +t_1^{1-\beta_0}), \label{bound-gamma}\end{equation}
Notice that $1-\beta_0\leq max\{1-3\alpha, -1+\alpha\},$ \emph{this is where we  restrict $\beta_0>\frac32,$}
so we conclude that
\begin{equation} \max(\la F\gamma F\ra|_t, 0 ) \leq C( t^{1-3\alpha}+t^{-1+\alpha}), \hspace{1cm} \forall t\geq t_0. \label{bound-FrF}\end{equation}
By  (\ref{x-bound}), we know that $  \la F\jx F\ra$ is bounded at any time $t$. Now we compute the derivative
\begin{align}
&\frac{d}{dt} \la F \jx F\ra = \la  D_H(F \jx F)\ra  \notag\\
=&2 \la F \gamma F\ra +\la     (D_HF)\jx F   + F \jx (D_HF) \ra  \notag \\
=&  2 \la F \gamma F\ra +\la \frac{2}{t^\alpha}\left[ F' \gamma \jx F+ F \jx \gamma F' \right] -2\frac{\alpha}{t}\frac{\la x\ra^2}{t^\alpha} FF'
\ra  \notag \\
= &2 \la F \gamma F\ra
+2 \la \sqrt{FF'\frac{\jx}{t^\alpha}} \gamma \sqrt{FF'\frac{\jx}{t^\alpha}}  \ra +O(t^{\alpha-1})    \label{fxf}
\end{align}
Applying (\ref{bound-FrF}) on the first term and (\ref{integrate-boundary}) on the second term,  we get
\be \la F \jx F\ra_t - \la F \jx F\ra_{t_0} \lesssim t^{2-3\alpha} + t^\alpha   ,\ee
hence we conclude that
\be \la F \jx F\ra_t  \lesssim    t^{2-3\alpha} + t^\alpha , \hspace{1cm}\forall t\geq t_0.\ee
This further implies control of growth for $\jx$ in the whole space. In fact,
\begin{equation}
\la \jx \ra =\la \jx F^2\ra + \la \jx (1-F^2)\ra  \leq C \max(t^{2-3\alpha}, t^{\alpha} ), \hspace{1cm}\forall t\geq t_0. \label{jx-bound}
\end{equation}
Using the fact $\la \jx \ra$ is bounded on $[1,t_0]$, we reach the conclusion (\ref{refined-x-bound}) by refining the constant.
\end{proof}

Proposition~\ref{slowgrowth}
says that    the expectation $\la |x|\ra$ for weakly localized state grows sublinearly in time.   Now we show the converse, if a solution has sublinear growth for $\la |x|\ra$,  it is weakly localized state.

\begin{prop} \label{slowgrowth-converse} 
Given a solution $\phi$ to equation (\ref{Main-eq}) with the property
\be \lim_{t\rightarrow +\infty}\frac{\la |x|\ra}{t}=0,  \label{sublinear}\ee
then for any  $\alpha\in (\frac13, 1)$ and any smooth characteristic function $F(\lambda \geq 1)$, we have
\be\lim_{t\rightarrow+\infty} \la F(\XT)\gamma F(\XT) \ra =0 .\ee\end{prop}
 \begin{proof}
  First we know $ \la F \gamma F\ra$ has a limit as $t\rightarrow +\infty$.  Now  assume   the limit is $\Gamma \not=0$. From the computation (\ref{fxf}), we have
\begin{align}
 & \la F \jx F\ra_T- \la F \jx F\ra_{T_0} \\
=&  \int_{T_0}^T 2 \la F \gamma F\ra
+2 \la \sqrt{FF'\frac{\jx}{t^\alpha}} \gamma \sqrt{FF'\frac{\jx}{t^\alpha}}  \ra +O(t^{\alpha-1})   dt
\end{align}
Using  (\ref{0limit-boundary}) and by taking $T_0$ large enough, we get  for $T$ large
\be  \la F \jx F\ra_T \gtrsim C|\Gamma| T \ee
which contradicts the condition \eqref{sublinear}.
\end{proof}

\begin{thm}~\label{unique-defin} Assume the decay assumption on nonlinearity (\ref{N-decay})  with $\beta_0 \geq \frac32$ holds for any choice of characteristic function $F$ and $\alpha\in(\frac13,1)$, then the definition of weakly localized state is independent of the choice of $F$ and $\alpha\in (\frac13,1)$.

In particular, for the weakly localized state, we have \begin{equation}\la |x|\ra \leq Ct^\frac12, \hspace{1cm} t\geq 1.\label{half-growth}\end{equation}\end{thm}
\begin{proof} 
Suppose that we have a solution that satisfies \be \lim _{t\rightarrow +\infty} \la F(\XT)\gamma F(\XT)\ra =0 \ee
for a given choice of $F$ and $\alpha\in (\frac13, 1)$. From Proposition~\ref{slowgrowth} we have
the slow growth (\ref{refined-x-bound}), notice that $\frac13< \max(2-3\alpha, \alpha)<1$. By Proposition~\ref{slowgrowth-converse}, we conclude that for any other choice of $F', \alpha'\in(\frac13, 1)$, the $\gamma$-limit is also $0$.

In particular, we can take $\alpha=\frac12$ in (\ref{refined-x-bound}) and obtain (\ref{half-growth}).
\end{proof}

\begin{rem}\label{question-remark}
 The $\gamma$ limit is given by its value on the free part of the solution.
 \end{rem}

\subsubsection{Scattering solution}

Next, we will show that the scattering solutions are not weakly localized. i.e. if a solution scatters, then it has positive $\gamma-$limit.
 \begin{thm}~\label{free-limit}
 If a global solution $\phi$ to equation (\ref{Main-eq})  scatters, i.e. there exists a linear solution $\psi_L=e^{it\Delta}\psi_0$ such that
 \be  \lim_{t\rightarrow +\infty}\|\phi(t)-\psi_L(t)\|_{H^1} =0 \label{scattering} \ee
 Then
 \begin{equation}\lim_{t\rightarrow +\infty} \la   F(\XT)\gamma F(\XT)  \ra >0.
 \end{equation}
 \end{thm}

 The proof of Theorem~\ref{free-limit} will be decomposed into the following lemmas.
 \begin{lem}\label{local-decay}
Suppose $A$ is a bounded operator $A=O(\frac{1}{\jx})$, then
\be \lim_{t\rightarrow +\infty} \la \psi_L(t), A\psi_L(t) \ra =0  \label{eq:local-decay} \ee
\end{lem}
\begin{proof}   Now for any $\epsilon$,  we take $N, M$ so that $\|(I-P_{N\leq \cdot \leq M})\psi_0\|_{L^2}\leq \epsilon.$ Then approximate $P_{N\leq \cdot \leq M}\psi_0$ with function $\tilde{\psi}_0$ whose Fourier transform is smooth and  compactly supported in $\{|\xi|\in (N_1,M_1)\}$, with error $\epsilon$. Then  from minimal velocity bound   (\ref{MinVB}), we get
\begin{align}
\int_{\R^3} \frac{|\psi_L(t,x)|^2}{\jx} dx \lesssim & \epsilon + \int_{|x|\geq N_1t} \frac{|e^{it\Delta} \tilde{\psi}_0|^2}{\jx} dx +  \int_{|x|\leq N_1t} \frac{|e^{it\Delta} \tilde{\psi}_0|^2}{\jx} dx\\
\lesssim& \epsilon + (N_1t)^{-1} + t^{-m}
\end{align}
Hence $\limsup_{t\rightarrow +\infty }\int_{\R^3} \frac{|\psi_L(t,x)|^2}{\jx} dx\lesssim \epsilon$. Since $\epsilon$ is arbitrary, we have completed the proof.
\end{proof}

\begin{lem}\label{xt-p}
Denote $A(t)=e^{-i\Delta t} \frac{x}{t}e^{i\Delta t}$ and $D=\{\psi_0\in H^{1}, |x|\psi_0\in L^2 \}$, then
\be \|(A(t)-2\bp )\psi_0\|\rightarrow_{t\rightarrow +\infty}0, \hspace{1cm}\forall \psi_0\in D.\label{pseudo}\ee
Moreover, $A(t)\rightarrow 2\bp$ in the strong resolvent sense.


\end{lem}
\begin{proof} \eqref{pseudo} follows from the
  pseudo-conformal identity
$\| (x-2\bp t)\psi_L(t)\|_{L^2}= \| x\psi_0\|_{L^2}. $
Now given any $\lambda$ with $Im \lambda\not=0$, and any $\psi_0\in \tilde{D} :=\{(2\bp-\lambda I)\tilde{\psi}_0, \tilde{\psi}_0\in D\}$
\begin{align}
& \|(A(t)-\lambda I)^{-1}\psi_0 -  (2\bp -\lambda I)^{-1}\psi_0\|\\
=&  \|(A(t)-\lambda I)^{-1} (2\bp -A(t)) (2\bp -\lambda I)^{-1}\psi_0\|\\
\lesssim& |Im \lambda|^{-1} \|(2\bp -A(t)) (2\bp -\lambda I)^{-1}\psi_0 \|\rightarrow 0.
\end{align}
Since $\tilde{D}$ is dense in $H^1$, we conclude  $A(t)\rightarrow 2\bp$ in the strong resolvent sense.
\end{proof}


\begin{lem}\label{free-limit2} Denote $\gamma(t)=e^{-i\Delta t}\gamma e^{i\Delta t}$.  $D$ defined as in Lemma~\ref{xt-p}.
Then  $\forall \psi_0\in D$
\begin{equation}\lim_{t\rightarrow +\infty}\la \psi_0, \gamma(t) \psi_0\ra  = \la \psi_0, |\bp|\psi_0\ra  \label{gamma-p}\end{equation}
Furthermore, the convergence also holds for all $\psi_0\in H^{1}$.
\end{lem}

\begin{proof}
Since $A(t)=e^{-i\Delta t} \frac{x}{t}e^{i\Delta t}\rightarrow 2\bp$ in the strong resolvent sense, by~\cite{Simon} Theorem VIII.20, $\|(f(A(t)) -f(2\bp))\psi_0\|\rightarrow 0$ for any continuous function $f$ vanishing at $\infty$, and $\psi_0\in D.$

 Take any $\psi_0\in D$, for any $\epsilon$, we find  $\delta$  and $M$, $\|P_{\leq \delta}\psi_0\|_{H^1}+  \|P_{\geq M}\psi_0\|_{H^1}\leq \epsilon.$
 Consider $F(x)=\frac{x}{|x|} f(\delta \leq |x|\leq M)$ where $f$ is a smooth characteristic function. By Weierstrass approximation theorem,  we find polynomials $F_n$ such that
\be |F(x) -F_n(x)|\leq \epsilon, \hspace{1cm}\forall x\in [\delta, M].\ee
To prove (\ref{gamma-p}), we break the physical space
\begin{align}
\la \psi_0, (\gamma(t)-|\bp|)\psi_0\ra  =\int_{|x|\leq 2} \bar{\psi}_L (\gamma -|\bp|) \psi_L +\int_{|x|\geq 2} \bar{\psi}_L (\gamma -|\bp|)\psi_L.
\end{align}
The first term converges to $0$  by Lemma~\ref{local-decay}. The second term is essentially $\int_{|x|\geq 2} \overline{P_{\delta\leq \cdot\leq M}\psi_L} (\gamma -|\bp|)P_{\delta\leq \cdot\leq M}\psi_L$ up to error of size $O(\epsilon). $
Next let us  prove \begin{equation}\limsup_{t\rightarrow +\infty} \int_{|x|\geq 2} \overline{P_{\delta\leq \cdot\leq M}\psi_L} (\gamma -|\bp|)P_{\delta\leq \cdot\leq M}\psi_L \lesssim \epsilon.\label{essential-part}\end{equation}
In fact, when $|x| \geq 2$,  $\gamma =\frac{x}{|x|}\cdot \bp +O(\frac{1}{|x|})$, and
\begin{align}
\gamma -|p| =&   \frac{x/t}{|x/t|}\cdot \bp -\frac{\bp}{|\bp|}\cdot \bp + O(\frac{1}{\jx})\\  =&\chi(\delta\leq |x/t|\leq M)(F(\frac{x}{t}) -F_n(\frac{x}{t}))\cdot \bp  - (F(2\bp)-F_n(2\bp))\cdot \bp  - (F_n(\frac{x}{t})-F_n(2\bp))\cdot \bp \notag\\ &+ \frac{x/t}{|x/t|} \chi(|x|/t\leq \delta)\cdot \bp +  \frac{x/t}{|x/t|} \chi(|x|/t\geq M)\cdot \bp - |p| \chi(|\bp|\leq \delta) -|\bp| \chi(|\bp|\geq M) + O(\frac{1}{\jx})\notag
\end{align}
When plugging into (\ref{essential-part}), the first two are small by Weierstrass approximation, the third one is small since  $F_n(A(t)) -F_n(2\bp) = e^{-it\Delta}(F_n(\frac{x}{t})-F_n(2\bp))e^{it\Delta}$ converges strongly.
  The smallness in the second line comes from the frequency support condition and the minimal/maximal velocity bound (\ref{MinVB})(\ref{MaxVB}), together with Lemma~\ref{local-decay}.

Hence, we conclude the proof of (\ref{essential-part}), which further implies (\ref{gamma-p}) for $\psi_0\in D$.  Finally, since $D$ is dense in $H^{1}$,  by density argument, we see that (\ref{gamma-p}) for $\psi_0\in H^{1}.$

\end{proof}

\begin{proof}[Proof of Theorem~\ref{free-limit}]
Since
\begin{align} ( \phi, F\gamma F \phi ) = (\phi -\psi_L, F\gamma F \phi )+  (\psi_L , F\gamma F (\phi -\psi_L) ) + (\psi_L, F\gamma F \psi_L)
 \end{align}
 The first two terms converge to $0$ because of scattering in $H^\frac12$ \eqref{scattering}.
 We are left to prove that the third term has a positive limit.  Since
 \begin{align}
 (\psi_L, \gamma \psi_L) = (\psi_L, F\gamma F \psi_L) + (\psi_L, (1-F)\gamma F \psi_L) + (\psi_L, (1-F)\gamma (1-F) \psi_L)
 \end{align}
and  the last two terms converge to $0$ by minimal/maximal velocity bound (\ref{MinVB})(\ref{MaxVB}), we conclude
\be  \lim_{t\rightarrow +\infty}   (\psi_L, F\gamma F \psi_L) =(\psi_0, |\bp|\psi_0)>0.\ee
\end{proof}
 \begin{rem} For simplicity, we presented Theorem~\ref{free-limit}  and proof when scattering holds in $H^1$ (\ref{scattering}). However, a close inspection of the proof, reveals that Theorem~\ref{free-limit} holds true if we have scattering in $H^\frac12, $ i.e.
 \be  \lim_{t\rightarrow +\infty}\|\phi(t)-\psi_L(t)\|_{H^\frac12} =0. \label{scattering2} \ee
 \end{rem}

 \section{Propagation Estimates}\label{sec:PE}
 In this section,  we construct more propagation observables. Using the commutator argument as explained in section {\ref{sec:gamma-limit}, we obtain suitable propagation estimates as (\ref{B-PE}) that will
 cover the whole phase-space. Recall that $t_0=100.$

\begin{lem}\label{lem:A2}  Let  $\phi$ be solution to equation (\ref{Main-eq}) satisfying global energy bound (\ref{global-bound}).    $\alpha\in (\frac13, 1)$,  $F_1=F_1(\XT)$,  
 \begin{enumerate}
 \item Let $F_2=F_2(\gamma>\delta)$ for some constant $\delta>0$, then  \begin{align} \int_{t_0}^\infty \frac{1}{t^\alpha}\la  \sqrt{F_1F_1'}\gamma F_2 \sqrt{F_1F_1'} \ra    dt  +  \int_{t_0}^\infty \frac{1}{t} \la \tilde F'_1  F_2\tilde F'_1 \ra dt & <\infty. \label{PE-2}
 \end{align}
Also  the limit \begin{align}  \lim_{t\rightarrow +\infty} & \la F_1(\XT)F_2(\gamma>\delta)  F_1(\XT) \ra
\end{align}  exists.
\item  Let $F_3=F_3(\gamma<-\delta)$ for some constant $\delta>0,$ then
\begin{equation} |\int_{t_0}^\infty \frac{1}{t^\alpha}\la F_1\gamma F_3F_1 \ra    dt|  +  \int_{t_0}^\infty \frac{1}{t^{\alpha}}\la  F_1F_3F_1 \ra dt<\infty. \label{PE-3}\end{equation}
and \begin{equation}\lim_{t\rightarrow +\infty}\la F_1(\XT)F_3(\gamma < - \delta)  F_1(\XT) \ra =0.\label{A30}\end{equation}
\item
We have stronger estimate \begin{align}
\int_{t_0}^\infty \frac{1}{t^\alpha}\la  \sqrt{F_1F_1'}\gamma^3  F_2 \sqrt{F_1F_1'} \ra    dt  +  \int_{t_0}^\infty \frac{1}{t} \la \sqrt{F_1F_1'}\gamma^2 F_2 \sqrt{F_1F_1'}\ra dt & <\infty. \label{PE-2vr2}\\
 |\int_{t_0}^\infty \frac{1}{t^\alpha}\la \sqrt{F_1F_1'}\gamma^3 F_3\sqrt{F_1F_1'} \ra    dt|  +  \int_{t_0}^\infty \frac{1}{t}\la  \sqrt{F_1F_1'}\gamma^2 F_3 \sqrt{F_1F_1'} \ra dt&<\infty. \label{PE-3vr2}\end{align}
 \end{enumerate}
\end{lem}
\begin{proof} 
(1) Denote $A_2=F_1(\XT)F_2(\gamma>\delta)  F_1(\XT)$. Then it verifies assumption (A1)(A2). We only need to compute $D_HA_2$,
 \begin{align}
 D_HA_2=& (D_HF_1) F_2F_1 + F_1F_2(D_HF_1)  + F_1[-i\Delta, F_2] F_1 \notag\\
=&F_1' \frac{1}{t^\alpha}\left[2\gamma   -\alpha \frac{\la x\ra }{t}\right]F_2F_1  +F_1F_2\frac{1}{t^\alpha}\left[2\gamma   -\alpha \frac{\la x\ra }{t}\right] F_1'+  F_1[-i\Delta, F_2]   F_1  \notag\\ &
+ \frac{1}{2t^{2\alpha}} \left[F_1''[[-i\Delta, \jx], \jx] , F_2\right]F_1 +    \frac{1}{2t^{2\alpha}}[F_2,F_1]  F_1''[[-i\Delta, \jx], \jx]  \end{align}
Using (\ref{f1f2}), we see that the last line is of order $O(t^{-3\alpha})  $.

From (\ref{DeltaF2})  we know for $t^\alpha\geq 4$,
 $F_1[-i\Delta, F_2]   F_1  =0 $.

Also notice
\begin{align}
&\la F_1' \frac{1}{t^\alpha}\left[2\gamma   -\alpha \frac{\la x\ra }{t}\right]F_2F_1  + F_1F_2\frac{1}{t^\alpha}\left[2\gamma   -\alpha \frac{\la x\ra }{t}\right] F_1' \ra \notag \\
= & \la \frac{4}{t^\alpha} \sqrt{F_1'F_1} \gamma F_2 \sqrt{F_1'F_1}\ra   -\frac{2\alpha}{t} Re \la F_1'  \frac{\jx }{t^\alpha}\phi, F_2F_1 \phi \ra  + O_1(t^{-3\alpha})
 \end{align}
Here we used (\ref{A2BC2}) and (\ref{f1f2}) to obtain the error bound.   Hence we get
 \begin{equation}\la D_HA_2\ra =\la  \frac{4}{t^\alpha} \sqrt{F_1'F_1} \gamma F_2 \sqrt{F_1'F_1}\ra   -\frac{2\alpha}{t} Re \la F_1'  \frac{\jx }{t^\alpha}\phi, F_2F_1 \phi \ra  + O_1(t^{-3\alpha}), \hspace{1cm} t\in [t_0, +\infty) \label{DAF2}\end{equation}
As in the proof of Theorem~\ref{thm:r-limit}, apriorily $D_HA_2$ doesn't satisfy assumption (A3), so we have to bootstrap.  Consider $\tilde{A}_2=\frac{1}{t^\epsilon}\la \tilde{F}_1(\XT) F_2(\gamma>\delta)\tilde{F}_1(\XT) \ra$ with $\epsilon+\alpha<1$,
$\tilde{F}_1$ is another smooth characteristic function to be chosen later.  Then $\tilde{A}_2$ satisfies assumption (A1)(A2)(A4), in particular
\be  D_H\tilde{A}_2= \frac{4}{t^{\alpha+\epsilon}} \sqrt{\tilde{F}_1'\tilde{F}_1} \gamma F_2 \sqrt{\tilde{F}_1'\tilde{F}_1}  +O_1(t^{-1-\epsilon}) \ee
Hence we get
\begin{equation} \int_{t_0}^\infty\frac{4}{t^{\alpha+\epsilon}}\la \sqrt{\tilde{F}_1'\tilde{F}_1}\gamma F_2\sqrt{\tilde{F}_1'\tilde{F}_1}\ra  dt  <\infty  \end{equation}
By taking   $\sqrt{\tilde{F}_1'(\lambda)\tilde{F}_1(\lambda)}=cF_1'(\lambda)\lambda$,  we get
\begin{align}
&\int_{t_0}^\infty\frac{\alpha}{t} \left| \la F_1'  \frac{\jx }{t^\alpha}\phi, F_2F_1 \phi \ra  \right| dt \notag\\
\leq & \int_{t_0}^\infty \frac{\alpha}{t}  \|\sqrt{ F_2}F_1'   \frac{\jx }{t^\alpha}\phi\|_{L^2}\|\phi\|_{L^2}   dt \notag\\
\lesssim & \frac{1}{\sqrt{\delta}}\left(\int_{t_0}^\infty \frac{1}{t^{\alpha+\epsilon}}  \|\sqrt{ \gamma F_2}F_1'  \frac{\jx }{t^\alpha}\phi\|^2_{L^2}dt\right)^{\frac12}\left( \int_{t_0}^\infty  \frac{1}{t^{2-\alpha-\epsilon}}\|\phi\|_{L^2}^2dt\right)^{\frac12}  <\infty \label{1t-term}
\end{align}
Hence $A_2$ satisfies assumption (A3), which implies the first integral in \eqref{PE-2}  is bounded and the existence of limit.   The same proof as \eqref{1t-term} implies that the second integral in \eqref{PE-2} is bounded.

(2) Denote $A_3=F_1(\XT)F_3(\gamma<-\delta)  F_1(\XT)$,  and
\begin{equation}D_HA_3= \frac{4}{t^\alpha} \sqrt{F_1'F_1} \gamma F_3 \sqrt{F_1'F_1}  -\frac{2\alpha}{t}\sqrt{F_1'F_1\frac{\jx}{t^\alpha}} F_3 \sqrt{F_1'F_1\frac{\jx}{t^\alpha}} + O_1(t^{-3\alpha}).  \label{A3}\end{equation}
We see that  $A_3$ verifies assumption (A1)(A2)(A4). Hence we get the first term of propagation estimate (\ref{PE-3}) and the limit $\lim_{t\rightarrow+\infty}\la A\ra_t$ exists.

To prove (\ref{A30}),  consider another observable $B=F_1\frac{\jx}{t}F_3F_1+h.c.$~\footnote{We use the notation $A+h.c.=A+A^*$}
 \begin{align}
 D_HB =  &   (D_HF_1)\frac{\jx}{t}F_3F_1 +  F_1(D_H\frac{\jx}{t})F_3F_1+ \frac{\jx}{t} F_1 (D_HF_3)F_1 + \frac{\jx}{t} F_1 F_3 (D_H F_1) +h.c.\\
 =&  (D_HF_1)\frac{\jx}{t}F_3F_1 +  F_1 (\frac{2\gamma}{t}- \frac{\jx}{t^2})F_3F_1+ \frac{\jx}{t} F_1 (D_HF_3)F_1 + \frac{\jx}{t} F_1 F_3 (D_H F_1)  +h.c.
 \label{neg-ga}\end{align}
 After symmetrization,
 we can see that the main terms are negative with errors $O_1(t^{-1-\epsilon)}$. $B$ verifies assumption (A1)(A2)(A4), and  we conclude
 \be |\int_{t_0}^\infty \la F_1 \frac{\gamma}{t} F_3F_1 \ra dt| <\infty\ee which further implies
  \be |\int_{t_0}^\infty \frac{1}{t}\la F_1 F_3 F_1\ra dt| <\infty  \ee
  Since   $\lim_{t\rightarrow +\infty} \la F_1F_3F_1\ra$ exists, we will    reach contradiction unless the limit is $0$.

Notice that since the Heisenberg derivative of the nonnegative PROB $A_3$ is negative in leading order, one can use an unbounded positive weight, with negative Heisenberg derivative, to obtain faster decay. In particular, one can get a factor $\frac{x}{t^{\alpha}}$
instead of $\frac{x}{t}$ in equations \ref{neg-ga}.
In general, the estimates on incoming waves imply faster decay.

  (3) The proofs are similar by considering PROB
  \be B_2= F_1(\XT)\gamma^2 F_2(\gamma>\delta)  F_1(\XT), \hspace{0.5cm} B_3= F_1(\XT)\gamma^2 F_3(\gamma<-\delta)  F_1(\XT) \ee
  Since $\phi\in L^\infty_t H^1$, by similar argument as previous cases,  we can check $B_2$ verifies assumption (A1)(A2)(A3), and $B_3$ verifies (A1)(A2)(A4). So we get the propagation estimate \eqref{PE-2vr2} and \eqref{PE-3vr2}, which comes from the leading terms in $D_HB_2, D_HB_3$.
 \end{proof}

\medskip

{\bf Second Microlocalization}

\medskip

 In order to cover the region where $\gamma$ is near $0$, we have to perform a second microlocalization in the phase space.
 \begin{lem}~\label{PA4} Let  $\phi$ be solution to equation (\ref{Main-eq}) satisfying global energy bound (\ref{global-bound}).   $F_1=F_1(\XT), F_4=F_4(\gamma t^\beta >c_0)$ . Then   the following estimate
 \begin{align}
\int_{t_0}^\infty \frac{1}{t^\alpha} \la \sqrt{F_1'F_1}  \gamma F_4 \sqrt{F_1'F_1} \ra +\frac{1}{t} \la F_1 F_4'F_1  \ra dt<\infty,  \label{PE-4V2}
\end{align}
  holds if the parameters $\alpha, \beta, c_0$ are in one of the following cases.
  \begin{enumerate}
  \item $\alpha\in (\frac12, 1)$,  $ \alpha+\beta<1$, and  $c_0>0$ is any constant.
  \item $\alpha\in (\frac13, 1), \beta\in (0,\alpha), $ and $c_0>0$ is any constant.
  \item $\alpha> \frac12, \alpha+\beta=1$, $c_0>\frac12\alpha$.
  \end{enumerate}
  \end{lem}
 \begin{proof}
 Denote  $A_4= F_1(\XT)F_4(\gamma t^\beta >c_0)F_1(\XT)$,  we will mostly ignore $c_0$ unless it's essential.     \begin{align}
  D_HA_4=& (D_HF_1)F_4F_1 + F_1F_4(D_HF_1)  + F_1(D_HF_4)F_1 \notag\\
=&F_1' \frac{1}{t^\alpha}\left[2\gamma   -\alpha \frac{\la x\ra }{t}\right]F_4F_1  +F_1F_4\frac{1}{t^\alpha}\left[2\gamma   -\alpha \frac{\la x\ra }{t}\right] F_1'+   F_1([-i\Delta, F_4] +\frac{d F_4}{dt})F_1  \notag\\ &
+ \frac{1}{2t^{2\alpha}} \left[F_1''[[-i\Delta, \jx], \jx] , F_4\right]F_1 +    \frac{1}{2t^{2\alpha}}[F_4,F_1]  F_1''[[-i\Delta, \jx], \jx]   \label{DA4}
  \end{align}
   We organize the terms into three groups $D_HA_4=I_1+I_2+I_3$, such that
 \begin{align}
 I_1= &F_1' \frac{1}{t^\alpha}\left[2\gamma   -\alpha \frac{\la x\ra }{t}\right]F_4F_1  + F_1F_4\frac{1}{t^\alpha}\left[2\gamma   -\alpha \frac{\la x\ra }{t}\right] F_1' \\
=& \frac{4}{t^\alpha} \sqrt{F_1'F_1} \gamma F_4 \sqrt{F_1'F_1}   -\frac{\alpha}{t}  [F_1'  \frac{\jx }{t^\alpha}F_4F_1  + F_1F_4 \frac{\jx}{t^{\alpha}}F_1'] + R \\
I_2=& F_1(D_HF_4)F_1   =  \frac{\beta}{c_0 t} F_1 \gamma t^{\beta} F_4'F_1, \hspace{1cm} \text{for } t^\alpha\geq 4,\\
I_3=& \frac{1}{2t^{2\alpha}} \left[F_1''[[-i\Delta, \jx], \jx] , F_4\right]F_1 +    \frac{1}{2t^{2\alpha}}[F_4,F_1]  F_1''[[-i\Delta, \jx], \jx] =O(t^{-3\alpha+\beta}).
 \end{align}
Here $R$ is the remainder terms coming from symmetrization, i.e.
\begin{align} & \frac{1}{t^\alpha}\sqrt{G_1}[[\sqrt{G_2}, \gamma F_4], \sqrt{G_3}]\sqrt{G_4} \\
=&   \frac{1}{t^\alpha}\left\{\sqrt{G_1}[\sqrt{G_2}, \gamma ][F_4, \sqrt{G_3}]\sqrt{G_4}  +   \sqrt{G_1}[\gamma , \sqrt{G_3}] [\sqrt{G_2}, F_4]\sqrt{G_4}   +  \sqrt{G_1}\gamma [[\sqrt{G_2}, F_4], \sqrt{G_3}]\sqrt{G_4} \right\} \notag\end{align}
with $G_1, G_2, G_3, G_4\in \{F_1, F_1'\}$, hence $R=O_1(t^{-3\alpha +\beta})$.    So we have
\begin{equation}D_H A =  \frac{4}{t^\alpha} \sqrt{F_1'F_1} \gamma F_4 \sqrt{F_1'F_1}  +\frac{\beta}{c_0t} F_1 \gamma t^{\beta}F_4'F_1   -\frac{\alpha}{t}  [F_1'  \frac{\jx }{t^\alpha}F_4F_1  + F_1F_4 \frac{\jx}{t^{\alpha}}F_1']  + I_3+R.
\label{A44}\end{equation}
The first two terms are positive, so it is left to show $I_3, R\in L^1(dt)$ and also control the third term.

Now let us prove (\ref{PE-4V2}) in different scenarios.

\textbf{Case I: }  If $\alpha>\frac12$ and $\alpha+\beta<1$, $c_0>0$, then $I_3, R \in L^1(dt)$.
  We get for $t\geq t_0$,
Then we perform the usual bootstrap argument, taking $\tilde{A}_4 = \frac{1}{t^\epsilon}\tilde{F}_1 \tilde{F}_4 \tilde{F_1}$, $\alpha+\beta+\epsilon<1. $ Then $\tilde{A}_4$ satisfies assumption (A1)(A2)(A4), hence we get
\begin{equation}\int_{t_0}^\infty\frac{1}{t^{\alpha+\epsilon}}\la \sqrt{\tilde{F}_1'\tilde{F}_1}\gamma \tilde{F}_4\sqrt{\tilde{F}_1'\tilde{F}_1}\ra   + \frac{1}{t^{1+\epsilon}} \la \tilde{F}_1 \gamma t^{\beta}\tilde{F}_4'\tilde{F}_1\ra
<\infty \label{e-improve}\end{equation}which further implies
$ \int_{t_0}^\infty\frac{1}{t^{\alpha+\beta+\epsilon}}\la \sqrt{\tilde{F}_1'\tilde{F}_1}  \tilde{F}_4\sqrt{\tilde{F}_1'\tilde{F}_1}\ra
<\infty.$
 By taking $\tilde{F}_4=F_4,  \sqrt{\tilde{F}_1\tilde{F}_1'}(\lambda) =F_1'(\lambda)\lambda$, we can control the third term in (\ref{A44}),  so $A_4$ satisfies assumption (A1)(A2)(A3).  We get a propagation estimate (\ref{PE-4V2}).
Notice that the restriction of $\alpha+\beta<1$ comes from the bootstrap argument.

\textbf{Case II:}  If $\alpha>\frac13$, $\alpha>\beta>0$, $c_0>0.$
Let us recall from Remark~\ref{rem:boundary}, the commutator $[F_1, F_2]$ is essentially supported in the region of $F_1, F_2$, up to error terms that are integrable.  More precisely,
 \begin{align}
&[F_1(\XT), F_4(\gbb)] \notag\\=&\sum_{k=1}^{n-1}\frac{1}{k!}F_4'(\gamma t^\beta)ad^{(k)}_{\gamma t^\beta}(F_1(\XT))+R_n =\sum_{k=1}^{n-1}(-1)^k\frac{1}{k!}ad^{(k)}_{\gamma t^\beta}(F_1(\XT)) F_4'(\gamma t^\beta)+R^*_n \label{FXFG}
\end{align}
 since $ad^{(k)}_{\gamma t^\beta}(F_1(\XT))= t^{k(\beta-\alpha)} O(\jx^0)F_1^{(k)}(\frac{\jx}{t^\alpha})$, and $R_n , R_n^*$  are of size $O(t^{n(\beta-\alpha))})$, by taking $\beta<\alpha$, and $n$ large, $R_n\in L^1(dt).$

 Now we  construct more observables to control the error term $I_3, R$ which are of the type
\be  \frac{1}{t^{2\alpha +k(\alpha-\beta)}} K(\frac{\jx}{t^{\alpha}}) L(\gamma t^\beta)K(\frac{\jx}{t^{\alpha}}),  \hspace{0.5cm}  \frac{1}{t^{2\alpha +k(\alpha-\beta)}}   L(\gamma t^\beta)K(\frac{\jx}{t^{\alpha}}) L(\gamma t^\beta) \ee
where $K, L$ are  functions with compact support coming from derivatives of $F_1, F_4$, $1\leq k\leq  n$, with the choice of $n$ such that $n+1$ is the smallest integer  which guarantees $2\alpha +(n+1)(\alpha-\beta)> 1,$ so that if we perform the commutator expansion  for $I_3, R$, we have  integrable error.

We perform the bootstrap argument in the following step:

\textit{0th-step}, we start with
$\tilde{A}_4 = \tilde{F}_1 \tilde{F}_4 \tilde{F_1}$, then  $\frac{1}{t^{n(\alpha-\beta)}}\la \tilde{A}_4\ra$ satisfies assumption (A1)(A2)(A4),  so we obtain estimate
\begin{equation}\int_{t_0}^\infty\frac{1}{t^{\alpha+\beta+n(\alpha-\beta)}}\la \sqrt{\tilde{F}_1'\tilde{F}_1}  \tilde{F}_4\sqrt{\tilde{F}_1'\tilde{F}_1}\ra    +\frac{1}{t^{1+n(\alpha-\beta)}}\la \tilde{F}_1 \tilde{F}_4'  \tilde{F}_1 \ra
<\infty. \label{0-iteration}\end{equation}


\textit{k-th step}, suppose in previous step we use $\frac{d}{dt}\frac{1}{ t^{k(\alpha-\beta)}}\la \tilde{A}_4\ra$ and obtain estimate
\begin{equation}\int_{t_0}^\infty\frac{1}{t^{\alpha+\beta+k(\alpha-\beta)}}\la \sqrt{\tilde{F}_1'\tilde{F}_1}  \tilde{F}_4\sqrt{\tilde{F}_1'\tilde{F}_1}\ra    +\frac{1}{t^{1+k(\alpha-\beta)}}\la \tilde{F}_1 \tilde{F}_4'  \tilde{F}_1 \ra dt
<\infty. \label{k-iteration}\end{equation}
We continue with $\frac{d}{dt}\frac{1}{ t^{(k-1)(\alpha-\beta)}}\la \tilde{A}_4\ra$, we only need to control the term coming from the error term $I_3, R$
\be  \frac{1}{t^{2\alpha +k(\alpha-\beta)}} \la K(\frac{\jx}{t^{\alpha}}) L(\gamma t^\beta) K(\frac{\jx}{t^{\alpha}})\ra  \ee
using  the first term of (\ref{k-iteration}). Hence, the error term is in $L^1(dt)$  and we get
\begin{equation}\int_{t_0}^\infty\frac{1}{t^{\alpha+\beta+(k-1)(\alpha-\beta)}}\la \sqrt{\tilde{F}_1'\tilde{F}_1}  \tilde{F}_4\sqrt{\tilde{F}_1'\tilde{F}_1}\ra    +\frac{1}{t^{1+(k-1)(\alpha-\beta)}}\la \tilde{F}_1 \tilde{F}_4'  \tilde{F}_1 \ra dt.
<\infty. \label{k1-iteration}\end{equation}
So up to $n$-th step, we succeeded in controlling the error term $I_3, R\in L^1(dt)$, so $A_4$ satisfy assumption (A1)(A2)(A3) and we obtain the existence of limit, together with the propagation estimate \eqref{PE-4V2}.
Of course, in each step we are not taking the same operator $\tilde{F}_1, \tilde{F}_4$, but modify it so as to cover the operator $K,L$ in the next iteration step.

\textbf{Case III:}  If  $\alpha>\frac12$ and $\alpha+\beta=1$, $c_0>\frac12 \alpha$. In this case, the choice of $c_0$ becomes essential.  Since $I_3, I =O(t^{-3\alpha+\beta})\in L^1(dt)$, but we can not use bootstrap to control the third term in \eqref{A44}. Instead, let us symmetrize the third term with a symmetrization error of size $O(t^{-3\alpha+\beta})$,  and we get
\be D_H A_4 =  \frac{4}{t^\alpha} \sqrt{F_1'F_1} \gamma F_4 \sqrt{F_1'F_1}  +\frac{\beta}{c_0t} F_1 \gamma t^{\beta}F_4'F_1   -\frac{2\alpha}{t}\sqrt{F_1'F_1\frac{\jx}{t^\alpha}}F_4 \sqrt{F_1'F_1\frac{\jx}{t^\alpha}}   +O(t^{-3\alpha+\beta}). \label{DA4-V3}\ee
Notice our choice of characteristic function, $supp F' \subset [\frac12, 1]$, hence we get that~\footnote{The error term comes from writing  $\sqrt{F_1'F_1\frac{\jx}{t^\alpha}}F_4 \sqrt{F_1'F_1\frac{\jx}{t^\alpha}}$ as $\sqrt{F_4}F_1'F_1\frac{\jx}{t^\alpha}\sqrt{F_4} $ so that we can use the support condition of $F_1'$, and then transform it back to the form $\sqrt{F_1'F_1}F_4 \sqrt{F_1'F_1}$ using \eqref{AB2A}.}
\begin{align}
\la \frac{2\alpha}{t}\sqrt{F_1'F_1\frac{\jx}{t^\alpha}}F_4 \sqrt{F_1'F_1\frac{\jx}{t^\alpha}} \ra \leq & \frac{2\alpha}{t} \la \sqrt{F_1'F_1} F_4 \sqrt{F_1'F_1} \ra +O(t^{-1-\epsilon})\\
\leq & \frac{2\alpha}{c_0 } \frac{1}{t^\alpha} \la \sqrt{F_1'F_1}\gamma F_4 \sqrt{F_1'F_1} \ra +O(t^{-1-\epsilon})
\end{align}
Since $c_0>\frac12\alpha$, this term can be controlled by first term in \eqref{DA4-V3}, and we get propagation estimate
\begin{align}
\int_{t_0}^\infty \frac{1}{t^\alpha} \la \sqrt{F_1'F_1}  \gamma F_4 \sqrt{F_1'F_1} \ra +\frac{1}{t} \la F_1 \gamma t^{\beta}F_4'F_1  \ra dt<\infty.
\end{align}
  \end{proof}

\begin{lem}~\label{PA5}  Let  $\phi$ be solution to equation (\ref{Main-eq}) satisfying global energy bound (\ref{global-bound}).    $\alpha> \frac12, \alpha+\beta=1$.  $F_1=F_1(\XT)$, $F_5=F_5(\gamma t^{\beta} \leq c_1)$, with $c_1<\frac14\alpha$. Then we have the propagation estimate
 \begin{align}
  & \int_{t_0}^\infty \frac{1}{t^\alpha} \left| \la \sqrt{F_1'F_1}\gamma F_{5,1}(\gamma t^\beta <-c_1) \sqrt{F_1'F_1}\ra \right| dt  \\+& \int_{t_0}^\infty \frac{1}{t} \la\sqrt{F_1'F_1 }F_5(\gamma t^\beta\leq c_1)\sqrt{F_1'F_1 } \ra dt +  \int_{t_0}^\infty\frac{1}{t} \left| \la F_1 F_5' \gamma t^{\beta}F_1 \ra \right| dt  <\infty. \label{PE-5}
 \end{align}
\end{lem}
\begin{proof} Denote $A_5= F_1(\XT)F_5(\gamma t^{\beta} \leq c_1)F_1(\XT)$, we have
\begin{align}
D_HA_5 =& (D_H F_1) F_5 F_1 +  F_1F_5 (D_HF_1) + F_1 (D_H F_5) F_1\notag\\
=& F_1' \frac{2\gamma}{t^\alpha}F_5F_1  +F_1F_5\frac{2\gamma}{t^\alpha}F_1'
-\frac{\alpha}{t}[F_1' \frac{\jx}{t^\alpha}F_5F_1  +F_1F_5\frac{\jx}{t^\alpha}F_1'] +  \frac{\beta}{t}F_1 F_5' \gamma t^{\beta}F_1 + O(t^{-3\alpha+\beta}) \notag \\
=& \frac{4}{t^\alpha} \sqrt{F_1'F_1}\gamma F_{5} \sqrt{F_1'F_1}   -\frac{2\alpha}{t} \sqrt{F_1'F_1\frac{\jx}{t^\alpha}}F_5\sqrt{F_1'F_1\frac{\jx}{t^\alpha}} + \frac{\beta}{t}F_1 F_5' \gamma t^{\beta}F_1 +O(t^{-3\alpha+\beta}).
\end{align}
 Denote $F_{5,1}=F(\gamma t^\beta <-c_1),  F_{5,2}=\tilde{F}(-c_1\leq \gamma t^\beta \leq c_1)$ such that $F_5 =F_{5,1} +F_{5,2}$.
\begin{align}
 & \frac{4}{t^\alpha} \la \sqrt{F_1'F_1}\gamma F_{5,2} \sqrt{F_1'F_1}\ra - \frac{2\alpha}{t} \la \sqrt{F_1'F_1\frac{\jx}{t^\alpha}}F_5\sqrt{F_1'F_1\frac{\jx}{t^\alpha}}\ra \\ \leq &  \frac{4c_1}{t}\la \sqrt{F_1'F_1}  F_{5,2} \sqrt{F_1'F_1}\ra  -\frac{\alpha}{t} \la \sqrt{F_1'F_1}F_5\sqrt{F_1'F_1}\ra  + O(t^{-3\alpha+\beta})\\
 \lesssim  &  -\frac{1}{t} \la \sqrt{F_1'F_1}F_5\sqrt{F_1'F_1}\ra +O(t^{-3\alpha+\beta})
 \end{align}
 This implies that if $c_1< \frac14 \alpha$, we get the propagation estimate (\ref{PE-5}).
\end{proof}

Lemma~\ref{PA4}(2) and Lemma~\ref{PA5} leave us with the boundary region $F_1 F(\frac14\alpha\leq \gamma t^\beta\leq \frac12\alpha)F_1$.  To handle it, we use the scaling argument.
 \begin{lem}\label{lem:boundary} Let  $\phi$ be solution to equation (\ref{Main-eq}) satisfying global energy bound (\ref{global-bound}). $\alpha>\frac12, \alpha+\beta=1$. For characteristic function $F_1, F_2$  of $[1,+\infty)$, we can find characteristic function $\tilde{F}_1, \tilde{F}_2$ of $[\frac14\alpha, 1]$ and $t_1\gg1 $,  such that the following propagation estimate holds
 \begin{align}
\int_{t_1}^\infty \frac{1}{t}\la \tilde{F}_1(\frac{\jx}{t^\alpha})  F_2(\gamma t^\beta) \tilde{F}_1(\frac{\jx}{t^\alpha})  \ra +\frac{1}{t}\la F_1(\frac{\jx}{t^\alpha}) \tilde{F}_2(\gamma t^\beta) F_1(\frac{\jx}{t^\alpha})\ra<\infty. \label{PE-boundary}
\end{align}
 \end{lem}

 \begin{proof}

  Let us first specify the choice of functions. Let $F_1, F_2$ be two smooth increasing characteristic functions of $[1,+\infty)$, i.e.

  $F_1(\lambda)=0, \lambda \leq \frac12; F_1(\lambda)=1, \lambda \geq 1.$ So $supp F_1'\in [\frac12, 1]$.

 $F_2(\lambda)=1, \lambda\geq 1;  F_2(\lambda)=0, \lambda \leq \frac34$. (The choice of $\frac34$ is not essential.)

  Now we take $B_\lambda=F_1(\lambda \frac{\jx}{t^\alpha})F_2(\lambda \gamma t^\beta)F_1(\frac{\lambda\jx}{t^\alpha})$, with $\alpha>\frac12,  \alpha+\beta=1,$ $\lambda\geq 1$ being a scaling parameter.  Then by repeating the calculation   similar to the proof in Lemma~\ref{PA4}(2), we get
 \begin{align}
 D_HB_\lambda =    \frac{4}{t^\alpha} \sqrt{F_1'F_1} \lambda \gamma F_2 \sqrt{F_1'F_1}  +\lambda\frac{\beta}{t} F_1 \gamma t^{\beta}F_2'F_1   -\lambda\frac{2\alpha}{t}\sqrt{F_1'F_1\frac{\jx}{t^\alpha}}F_2 \sqrt{F_1'F_1\frac{\jx}{t^\alpha}}   +O(\lambda^3t^{-3\alpha+\beta}).
 \end{align}
 Notice that
 \be \frac{4}{t^\alpha} \sqrt{F_1'F_1} \lambda \gamma F_2 \sqrt{F_1'F_1}      -\lambda\frac{2\alpha}{t}\sqrt{F_1'F_1\frac{\jx}{t^\alpha}}F_4 \sqrt{F_1'F_1\frac{\jx}{t^\alpha}}  \geq \frac{3-2\alpha}{t}   \sqrt{F_1'F_1} F_2 \sqrt{F_1'F_1},\ee
so we have
 \begin{align}
\int_{t_\lambda}^\infty \frac{1}{t} \la \sqrt{F_1'F_1} F_2 \sqrt{F_1'F_1}  \ra +\frac{1}{t} \la F_1   F_2'F_1 \ra dt<\infty.  \label{PE-9}
\end{align}
Here $t_\lambda$ is chosen so that $\{\lambda\frac{\jx}{t^\alpha} \geq 1 \}\subset\{\jx\geq 4\}$ holds true for all $t\geq t_\lambda$, to insure that $F_1[-i\Delta, F_2]F_1$ vanishes.

 The main observation is that (\ref{PE-9}) indicates that there is  no propagation in the region  $\mathcal{A}_\lambda\cup \mathcal{B}_\lambda$
\begin{align}
\mathcal{A}_\lambda=&\left\{ \lambda\frac{\jx}{t^\alpha}\in [\frac12, 1]\right\} \cap \left\{\lambda\gamma  t^\beta \geq \frac34  \right\}\\
\mathcal{B}_\lambda=& \left\{ \lambda\frac{\jx}{t^\alpha} \geq \frac12\right\} \cap \left\{\lambda\gamma t^\beta \in [\frac34, 1] \right\}
\end{align}
which means $\lim_{t\rightarrow +\infty }\la F_1' F_2F_1'\ra=0, \lim_{t\rightarrow +\infty }\la F_1F_2'F_1\ra=0. $
In fact, the proof is similar to Proposition~\ref{0limit-frf}. We repeat the calculation of derivative to show that they have limit,  the only difference is that $F_1'', F_2''$ may not be a positive function, but we can write it as a sum of positive functions, and use (\ref{PE-9}).  Then argue by contradiction with (\ref{PE-9}) to show that the limit is $0$.

Next, we perform the scaling argument by taking $\lambda_j =(\frac{4}{3})^j$.  Notice that the phase-space region $\cup_{j=0}^J\mathcal{B}_{\lambda_j}$ will cover the region $\{\frac{\jx}{t^\alpha} \geq \frac12\} \cap \{\gamma t^\beta \in [(\frac34)^J, 1]\}$.  In particular, when $J\geq J_0$ is large enough, $(\frac34)^J<\frac18\alpha$, we can cover the boundary region.
And we have the following estimate
\begin{align}
\int_{t_1}^\infty \frac{1}{t}\la \tilde{F}_1(\frac{\jx}{t^\alpha})  F_2(\gamma t^\beta) \tilde{F}_1(\frac{\jx}{t^\alpha})  \ra +\frac{1}{t}\la F_1(\frac{\jx}{t^\alpha}) \tilde{F}_2(\gamma t^\beta) F_1(\frac{\jx}{t^\alpha})\ra<\infty \label{PE-8}
\end{align}
Here $t_1=\max_{0\leq j\leq J_0} t_{\lambda_j}$, and $\tilde{F}_1$ is the characteristic function of $[(\frac{3}{4})^k\frac12,1]$, which can be viewed as $\sum_{j=0}^{J_0} \sqrt{F_1F_1' ((\frac43)^j  \lambda)} $ up to some renormalization constant. And $\tilde{F}_2$ is the characteristic function of $[(\frac{3}{4})^{J_0}, 1]$, which can be viewed as $\sum_{j=0}^{J_0}F_2'((\frac43)^j\lambda)$.
 \end{proof}

   \begin{rem}  We can also try to sharpen the choice of parameters.   In particular, we can consider PROB \be B_{a, b,c} =F_1(\frac{\jx}{t^\frac12 (\log t)^a}\geq 1)F_2(\gamma t^\frac12 (\log t)^b\geq c) F_1(\frac{\jx}{t^\frac12 (\log t)^a} \geq 1) , \hspace{1cm} a, b\in \R, c>0.\label{Bab}\ee
   Let us denote $G=\sqrt{F_1F_1'}(\frac{\jx}{t^\frac12 (\log t)^a})$, $\tilde{G}=\sqrt{\frac{\jx}{t^\frac12 (\log t)^a} F_1F_1'}$, $\tilde{F}_2=\frac{\gamma t^\frac12 (\log t)^b}{c}F_2'$. Then, by computing the Heisenberg derivative, we get \begin{align}
   D_H B_{a,b,c} = &  \frac{4}{t^\frac12 (\log t )^a} G\gamma F_2 G- (\frac{1}{t} +\frac{2a}{t\log t })  \tilde{G} F_2 \tilde{G} +(\frac{1}{2t}+\frac{b}{t\log t })F_1\tilde{F}_2F_1\label{DHBabc}\\
   &+ O(\frac{1}{ t (\log t)^{3a-b}} + \frac{1}{ t (\log t)^{2a-2b}} ).
   \end{align}
  Here, we ignore the constant dependence on $c$.  The first error term comes from the symmetrization of $[-i\Delta, F_1]F_2F_1$, the second error term comes from the symmetrization of $(\partial_t F_1)F_2F_1$.  So by making the following assumptions
  \begin{align}
  3a-b>&1,  \quad 2a-2b>1\label{error-control}\\
  a+b <& 0, \text{ any } c>0, \hspace{1cm} \text{ or } \quad
  a+b =0, c>\frac14. \label{1t-control}
  \end{align}
 \eqref{error-control} ensures the error terms are integrable in time, while \eqref{1t-control} ensures that the second term in \eqref{DHBabc} is controlled by the first term for $t$ large enough.  Hence $B_{a,b,c}$ satisfies assumption (A1)(A2)(A3), and we get the propagation estimate for large time $t_1$,
 \be  \int_{t_1}^\infty  \frac{1}{t^\frac12 (\log t )^a}   \la G\gamma F_2 G  \ra  +\frac{1}{t}\la F_1\tilde{F}_2F_1 \ra dt<\infty.
 \ee
 We notice some of the special cases of \eqref{error-control}\eqref{1t-control}
 \begin{enumerate}
 \item $a=0$, $b<-1, c>0.$
 \item $a=-b>\frac14, c>\frac14$.
 \item $a>\frac14, a+b>0, c>0.$
 \end{enumerate}
 \end{rem}

At the end, we list one more propagation estimate that will be used in the next section.
\begin{lem}~\label{PAvr2} Let  $\phi$ be solution to equation (\ref{Main-eq}) satisfying global energy bound (\ref{global-bound}).   $\alpha>\frac12,  \beta \in(0,1-\alpha] $ and $ M\geq 1$ is any constant.  Denote $F_1=F_1(\XT), F_5=F_5(|\gamma| t^\beta>1), \tilde{F}_5=\tilde{F}_5(|\gamma| <M)$.
Then we have   the following estimate
 \begin{align}
\int_{t_0}^\infty \frac{1}{t^\alpha} \la \sqrt{F_1'F_1}  |\gamma|^3 F_5\tilde{F}_5\sqrt{F_1'F_1} \ra + \frac{1}{t}   \la \sqrt{F_1'F_1}  \gamma^2 F_5\tilde{F}_5\sqrt{F_1'F_1} \ra dt<\infty,  \label{PE-r2}
\end{align}
\end{lem}
\begin{proof}  We split the proof into the case $\gamma>0$ and $\gamma <0$.  When $\gamma>0$
\eqref{PE-r2} simply follows from the first integral in  \eqref{PE-4V2} using the bound $|\gamma|\leq M$.


When $\gamma<0$, we denote $B=F_1\gamma^2 F_6(\gamma t^\beta<-1) \tilde{F}_6(\gamma>-M) F_1$.   By calculating the Heisenberg derivative, we get
\begin{align}
D_HB= \frac{1}{t^\alpha}\sqrt{F_1F_1'}  \gamma^3 F_6\tilde{F}_6\sqrt{F_1F_1'}   - \frac{1}{t} \sqrt{\frac{\jx}{t^\alpha}F_1F_1'}  \gamma^2 F_6\tilde{F}_6 \sqrt{\frac{\jx}{t^\alpha}F_1F_1'}   +\frac{\beta}{t} F_1\gamma^3 t^\beta F_6'\tilde{F}_6F_1 +O_1(t^{-1-\epsilon})
\end{align}
The first two terms are negative, while the third term is of size $O(t^{-1-2\beta})$ hence is integrable in time. So $B$ satisfies Assumption (A1) (A2) (A4), and we get the propagation estimate \eqref{PE-r2}.
\end{proof}
\begin{rem} All propagation estimates shown in this section also work for free solutions with $H^1$ data.
\end{rem}

\section{Wave operator}
In this section, we will verify the existence of wave operators by Cook's method.   In fact, we have the following theorem
for the free channel wave operator. 
\begin{thm}\label{thm:WP}Let   $\alpha_0\in (\frac12, 1)$ and $F$ be a smooth characteristic function of $[1,\infty)$.  Let  $\phi(t)$ be the  solution to equation (\ref{Main-eq})  with initial data $\phi_0$,  satisfying global $H^1$ bound (\ref{global-bound})

Then the {\bf  free channel wave operator $\Omega^*_F$} defined by
\be
\Omega^*_F\phi_0: = \lim_{t\rightarrow +\infty}e^{-i\Delta t} F(\frac{\jx}{t^{\alpha_0}}\geq 1)\phi(t)\label{channelWP}
\ee
 exists in $H^1$. 
\end{thm}

\begin{proof}  Denote
\be \Omega_F^*(t)= e^{-i\Delta t} F(\frac{\jx}{t^{\alpha_0}}\geq 1)\phi(t) \ee
We first   show $ \Omega_F^*(t)$ converges  in $L^2$,  and then upgrade the convergence to $H^1.$

\textbf{Step 1: $L^2$ convergence}. Denote $\sqrt{F}=F_1$, and choose the partition of unity $I=F_2+F_3+F_4$, such that $F_k, k=2,3,4$ is a smooth characteristic function of $ [\alpha_0, +\infty),  [\frac14 a_0, a_0]$ and $(-\infty, \frac14 a_0]$. Take $\beta =1-\alpha_0$ and denote \begin{align}
A_k= & F_1(\frac{\jx}{t^{\alpha_0}}\geq 1) F_k(\gamma t^\beta) F_1(\frac{\jx}{t^{\alpha_0}}\geq 1) ,\\ \omega_k(t)=& e^{-i\Delta t} A_k\phi(t).\end{align}
So $\Omega_F^*(t)=\sum_{k=2}^4 \omega_k(t)$ and we only need to show that $\omega_k(t)$ converges in $L^2$. By Cauchy criterion, this reduces to showing
\begin{align}
 \|\omega_k(t) -\omega_k(s)\|^2_{L^2}   \rightarrow 0, \hspace{1cm} t,s\rightarrow +\infty.
\end{align}
In fact, we denote $\psi_{t,s} = \omega_k(t) -\omega_k(s)$, and  we have the uniform estimate for any $t,s\geq 1.$
\be \|\psi_{t,,s}\|_{H^1}\lesssim \sup_{t\in[0,1)}\|\phi(t)\|_{H^1}\ee
Since
\begin{align}  \frac{d}{dt}\omega_k(t)  = &  e^{-i\Delta t} [(D_HA_k)\phi(t) -i    A_k \mathcal{N}(\phi)],\\
D_HA_k=&  \frac{4}{t^{\alpha_0}} \sqrt{F_1'F_1} \gamma F_k \sqrt{F_1'F_1}  +\frac{\beta}{t} F_1 \gamma t^{\beta}F_k'F_1   -\frac{2\alpha_0}{t}\sqrt{F_1'F_1\frac{\jx}{t^{\alpha_0}}}F_k \sqrt{F_1'F_1\frac{\jx}{t^{\alpha_0}}}   +O(t^{-3{\alpha_0}+\beta}).
\end{align}
We first look at $k=2$, and denote $D_HA_2 =B_{2,1}^*B_{2,1} + B_{2,2}^*B_{2,2} - B_{2,3}^*B_{2,3} + O(t^{-3{\alpha_0}+\beta})$, with
\begin{align}
B_{2,1}(t) =&2t^{-\frac{{\alpha_0}}{2}}   \sqrt{\gamma F_2} \sqrt{F_1'F_1}\label{B1} \\
B_{2,2}(t)=& t^{-\frac12}  \sqrt{\gamma t^{\beta}F_2'} F_1\label{B2} \\
B_{2,3}(t)=&t^{-\frac12}\sqrt{F_2}  \sqrt{F_1'F_1\frac{\jx}{t^{\alpha_0}}}  \label{B3}
\end{align}
Using the decay assumption on the nonlinearity (\ref{N-decay}), we get
\begin{align}
&  \| \omega_2(t) -\omega_2(s)\|_{L^2}^2 \\= &|  \lp \psi_{t,s}, \int_s^t \frac{d}{d\tau}\omega_2(\tau)d\tau \rp |\\
  \lesssim & \sum_{j=1}^3  |\int_s^t \lp B_{2,j} e^{i\Delta \tau}\psi_{t,s}, B_{2,j}\phi(\tau) \rp d\tau|  + \int_{s}^{t} O(\tau^{-1-\epsilon})   d\tau\\
  \lesssim & \sum_{j=1}^3 \lp \int_{s}^t \|B_{2,j}(\tau) e^{i\Delta \tau}\psi_{t,s}\|_{L^2}^2d\tau\rp^{\frac12}   \lp \int_{s}^t \|B_{2,j}(\tau)\phi(\tau)\|_{L^2}^2d\tau\rp^{\frac12}  + O(t^{-\epsilon}+ s^{-\epsilon}) \label{w2-estimate} \end{align}
Now using propagation estimate (\ref{PE-4V2}) for $\phi $, we see that $ \int_{s}^t \|B_{2,j}(\tau)\phi(\tau)\|_{L^2}^2d\tau\rightarrow 0$  as $t, s\rightarrow +\infty$. On the other hand $e^{i\tau \Delta}\psi_{t,s}$ is a free wave with $H^1$ data, so (\ref{PE-4V2})  also holds. In particular we have
\be \int_{s}^t \|B_{2,j}(\tau) e^{i\Delta \tau}\psi_{t,s}\|_{L^2}^2 d\tau  \leq \int_{t_0}^\infty \|B_{2,j}(\tau) e^{i\Delta \tau}\psi_{t,s}\|_{L^2}^2 d\tau <\infty\ee
The bound holds uniformly for any $t,s\geq t_0$.  Hence, we proved $\omega_2(t)$ has a limit in $L^2$.

The proof for $k=3$ is essentially the same. The only
  difference is  that $\frac{\beta}{t} F_1 \gamma t^{\beta}F_3'F_1$  now decomposes into two terms depending on $\lambda F_3'(\lambda)$ being positive and negative. Then we can apply (\ref{PE-boundary}).

To deal with $k=4$,
 we   write
\be F_4(\gamma t^\beta \leq \frac14{\alpha_0})=F_{4,1}(\gamma t^\beta \leq -\frac14 {\alpha_0})+F_{4,2}(|\gamma t^\beta|\leq \frac14{\alpha_0})\ee
and decompose
\be D_HA_4 =B_{4,1}^*B_{4,1} +  \frac{4}{t^{\alpha_0}} \sqrt{F_1'F_1} \gamma F_{4,2} \sqrt{F_1'F_1} - B_{4,3}^*B_{4,3} - B_{4,4}^*B_{4,4} + O(t^{-3{\alpha_0}+\beta})\ee with
\begin{align}
B_{4,1}(t) =&2t^{-\frac{{\alpha_0}}{2}}   \sqrt{-\gamma F_{4,1}} \sqrt{F_1'F_1} \\
B_{4,2}(t) =&t^{-\frac12}   \sqrt{ F_{4,2}} \sqrt{F_1'F_1} \\
B_{4,3}(t)=& t^{-\frac12}  \sqrt{-\gamma t^{\beta}F_4'} F_1 \\
B_{4,4}(t)=&t^{-\frac12}\sqrt{F_4}  \sqrt{F_1'F_1\frac{\jx}{t^{\alpha_0}}}
\end{align}
notice that
\begin{align} |(e^{i\tau \Delta}\psi_{t,s},  \frac{4}{\tau^{\alpha_0}} \sqrt{F_1'F_1} \gamma F_{4,2} \sqrt{F_1'F_1}  \phi(\tau)) |=& |\frac{4}{\tau^{{\alpha_0}}}( \sqrt{ F_{4,2}} \sqrt{F_1'F_1}  e^{i\tau \Delta}\psi_{t,s},      \gamma  \sqrt{ F_{4,2}} \sqrt{F_1'F_1}   \phi(\tau))|\\
\lesssim & \|B_{4,2}(\tau) e^{i\Delta \tau}\psi_{t,s}\|_{L^2} \|B_{4,2}(\tau) \phi(\tau)\|_{L^2}
\end{align}
Similarly as (\ref{w2-estimate}), we use  propagation estimate (\ref{PE-5}) to show that $\omega_4(t)$ has a limit in $L^2.$

\textbf{Step 2: $H^1$ convergence.} Now we perform a different decomposition for $\Omega_F^*(t)$. For any chosen $0<\varepsilon \ll 1- a_0$ arbitrarily small  and $M\geq 1$,
\begin{align}
\tilde{A}_2= & F_1(\frac{\jx}{t^{\alpha_0}}\geq 1)    G_2(|\gamma|\leq t^{-\varepsilon}) F_1(\frac{\jx}{t^{\alpha_0}}\geq1)\\
\tilde{A}_3= & F_1(\frac{\jx}{t^{\alpha_0}}\geq 1)   G_3(|\gamma|\geq   t^{-\varepsilon}) \tilde{G}_3(|\gamma| \leq M) F_1(\frac{\jx}{t^{\alpha_0}}\geq1)\\
\tilde{A}_4= & F_1(\frac{\jx}{t^{\alpha_0}}\geq 1) G_4(|\gamma|\geq M)F_1(\frac{\jx}{t^{\alpha_0}}\geq1)\\
\rho_k(t)= &  e^{-i\Delta t} \tilde{A}_k \phi(t) \end{align}
$\Omega_F^*(t)=\sum\rho_k(t)$.  In the following,  we will show that $\frac{\bp^2}{\sqrt{\bp^2+1}}\rho_k(t)$ converges in $L^2,$  when combined with step 1 implies the $H^1$ convergence of $\Omega_F^*(t)$.

The key fact we use is that $\gamma^2=\bp^2$ on the support of $F_1$.  So for $k=2$
\begin{align}
\|\frac{\bp^2}{\sqrt{\bp^2+1}}\rho_2(t)\|_{L^2}\lesssim   \|   F_1 \gamma^2 G_2 F_1 \phi(t) \|_{L^2} +\|[\gamma^2, F_1]G_2F_1\phi\|_{L^2} \lesssim t^{-\varepsilon}\|\phi(t) \|_{H^1}
\rightarrow   0.
\end{align}
For $k=3,4$, we denote $\theta_{t,s} =  \frac{\bp^2}{\bp^2+1} (\rho_k(t)-\rho_k(s))$, then $\|\theta_{t,s}\|_{H^1}$ is uniformly bounded. We have
\begin{align}
&\|\frac{\bp^2}{\sqrt{\bp^2+1}}(\rho_k(t) -\rho_k(s))\|^2_{L^2} \\
=& \lp  \frac{\bp^4}{\bp^2+1} (\rho_k(t)-\rho_k(s)),   \int_s^t   \frac{d}{d\tau}  \rho(\tau) d\tau\rp  \\
=& \int_s^t \lp  e^{i\tau \Delta} \theta_{t,s}  
, \gamma^2  D_H\tilde{A}_k\phi(\tau) \rp  +O(\tau^{-1-\epsilon}) d\tau
\end{align}
Here the error term comes from estimating the nonlinearity.   Now we calculate
\begin{align}
\gamma^2  D_H\tilde{A}_3=&\frac{1}{t^{\alpha_0}}\sqrt{F_1F_1'}\gamma^3G_3\tilde{G}_3 \sqrt{F_1F_1'}  -\frac{1}{t} \sqrt{F_1F_1'\frac{\jx}{t^{\alpha_0}}}\gamma^2 G_3\tilde{G}_3\sqrt{F_1F_1'\frac{\jx}{t^{\alpha_0}}}\\ & + \frac{\varepsilon}{t} F_1\gamma^3 t^\varepsilon G_3'\tilde{G}_3F_1+ O_1(t^{-2\alpha_0})\label{good-dG} \\
\gamma^2  D_H\tilde{A}_4=&\frac{1}{t^{\alpha_0}}\sqrt{F_1F_1'}\gamma^3G_4 \sqrt{F_1F_1'} -\frac{1}{t} \sqrt{F_1F_1'\frac{\jx}{t^{\alpha_0}}}\gamma^2 G_4 \sqrt{F_1F_1'\frac{\jx}{t^{\alpha_0}}} +O_1(t^{-2\alpha_0}) &
\end{align}
The error terms come from symmetrization.  Also notice the first  term in \eqref{good-dG} is of size $O_1(t^{-1-2\varepsilon})$. 
This is due to the fact that $\gamma^3 G'_3~t^{-3\epsilon}.$
So as in step 1, the propagation estimate  \eqref{PE-2vr2}\eqref{PE-3vr2} and\eqref{PE-r2} are sufficient to show
\be\|\frac{\bp^2}{\sqrt{\bp^2+1}}(\rho_k(t) -\rho_k(s))\|^2_{L^2} \rightarrow 0, \text{ as } t, s\rightarrow 0,\ee
which finishes the proof.

\end{proof}

\begin{rem}A-priory, the definition of  $\Omega^*_F$ depends on the choice of $F, {\alpha_0}$. Here let us show that for ${\alpha_0}$ fixed, the definition is actually independent of $F$.
In fact,  propagation estimates  (\ref{PE-4V2}) (\ref{PE-5})(\ref{PE-boundary}) together, imply that for any characteristic function $G$ of $[1,+\infty)$
\be
\int_{t_0}^\infty  \la G'\ra \frac{dt}{t} <\infty,
\ee
hence we find sequence of time $t_n\rightarrow +\infty$, $\la G'\ra_{t_n}\rightarrow 0.$ Now given $F, \tilde{F}$ two characteristic function of $[1,+\infty)$, we find $G$ such that $|F-\tilde{F}|^2 \lesssim G'$, then
\[\|\Omega_{F}^*(t) -\Omega_{\tilde{F}}^*(t)\|^2_{L^2}\leq  \la G'\ra_t.\]
 LHS has a limit as $t\rightarrow +\infty$,  while the RHS  converge to $0$ on a sequence of time. Hence $\lim_{t\rightarrow +\infty}\|\Omega_{F}^*(t) -\Omega_{\tilde{F}}^*(t)\|_{L^2}=0$.

 It follows from our analysis that($\alpha_1>1/2)$    \be \int_{t^{\alpha_1}\lesssim |x|\lesssim t^{\alpha_2}} |\phi(t)|^2dt \rightarrow 0, \hspace{1cm}t\rightarrow +\infty. \label{Fa-independence}
\ee

\end{rem}
\begin{rem}
  The domain of $\Omega_F^*$   is the set of all data that leads to a global solution satisfying the global bound of $H^1$.

With the help of various propagation estimates, we proved in Section~\ref{sec:PE}, and by identical proof, we actually have the limit defining various channels. We list them here:
\begin{align}
  \lim_{t\rightarrow +\infty}&e^{-i\Delta t} F(\XT)F(\gamma>\delta)\phi(t), \hspace{1cm}  \alpha>\frac13,\label{channelWP0}\\
  \lim_{t\rightarrow +\infty}&e^{-i\Delta t} F(\XT)F(\gamma< -\delta)\phi(t), \hspace{1cm}  \alpha >\frac13,\label{channelWP1}\\
    \lim_{t\rightarrow +\infty}& e^{-i\Delta t} F(\XT)F(\gamma t^\beta>1)\phi(t), \hspace{1cm}  \alpha>\frac13, \beta\in (0,\alpha),\label{channelWP3}\\
        \lim_{t\rightarrow +\infty}& e^{-i\Delta t} F(\frac{\jx}{t^\frac12}\geq 1)F(\gamma t^\frac12 (\log t)^{b} >1)\phi(t), \hspace{1cm}   b<-1.\label{channelWP4}
\end{align}
\end{rem}




\subsection{Asymptotic decomposition}

 Now denote $\phi_L(t)=e^{i\Delta t}\Omega^*_F\phi_0$ and  $\pwb=\phi(t)-\phi_L(t)$, which we call \textit{weakly bounded part }of the solution.  Let us discuss various options for the asymptotic decomposition of the solution.
\begin{itemize}
\item \textbf{Decomposition 1:} $ \phi(t) =\pwb+  e^{i\Delta t}\Omega^*_F\phi_0$, this is the exact decomposition that holds at any time.
\begin{enumerate}
\item Since $\Omega^*_F\phi_0\in H^1$, we know $\|\pwb\|_{L^\infty_t H^1}<\infty.$
\item $\pwb$ verifies the equation
\be (i\partial_t+\Delta )\pwb -\mathcal{N}(\pwb)\pwb={\bf N}(\pwb+\phi_L)-{\bf N}(\pwb) \equiv \phi_L \tilde{V}(t,|x|,\pwb, \phi_L) \label{eq:pwb}\ee
and $\tilde{V}$ has sufficient decay in $x$. Hence we have
\be \|{\bf N}(\pwb+\phi_L)-{\bf N}(\pwb)\|_{L^2_tL^{\frac65}_x} \lesssim \|\phi_L\|_{L^2_tL^6_x} \|\tilde{V}(t,|x|,\pwb, \phi_L)\|_{L^\infty_tL^{\frac32}_x} <\infty.\ee
And we conclude that $\pwb$ verifies the equation in the following sense:

First, asymptotically $\pwb$ verifies the equation in the weak sense, i.e. for $\forall \eta \in L^2$
\be \lp \eta,  (i\partial_t+\Delta )\pwb -{\bf N}(\pwb)\rp \rightarrow 0, \hspace{1cm} t\rightarrow +\infty.\ee
This follows from the decay on $\tilde{V}$ and the local decay estimate for free solution Lemma~\ref{local-decay}.

Second, $\pwb$ verifies the integral equation asymptotically
\be \left\|\pwb - e^{it\Delta}\pwb(0) -\int_0^t e^{i(t-s)\Delta} {\bf N}(\pwb)(s)ds\right\|_{L^2_tL^6_x  ([T,\infty)\times \R^3)}\rightarrow 0.\ee
\item
By
decomposing $\pwb = F(\frac{\jx}{t^{\alpha_0}}\geq 1)\phi - e^{i\Delta t}\Omega^*_F\phi_0  + F(\frac{\jx}{t^{\alpha_0}}\leq 1)\phi$,  we have
\begin{align} \|\tilde{F}(\XT)\pwb\|_{L^2}\leq  & \|F(\frac{\jx}{t^{\alpha_0}}\geq 1)\phi - e^{i\Delta t}\Omega^*_F\phi_0\|_{L^2} + \|\tilde{F}(\XT)F(\frac{\jx}{t^{\alpha_0}}\leq 1)\phi\|_{L^2}\\
  \rightarrow&_{t\rightarrow +\infty} \, 0, \hspace{1cm}\forall \alpha>\alpha_0.\end{align}
The same decomposition also implies
 \be
\lp  \pwb,  \tilde{F}(\frac{\jx}{t^{\alpha}}\geq 1)\gamma \tilde{F}(\frac{\jx}{t^{\alpha}}\geq 1) \pwb \rp\rightarrow_{t\rightarrow+\infty} 0 , \hspace{0.5cm}\alpha\geq \alpha_0. \label{pwb-gamma-limit}
\ee
Here we allow $\alpha=\alpha_0$ since on the boundary, $\gamma$-limit is still $0$, see \eqref{0limit-boundary}.

We also notice that $\pwb$   is asymptotically orthogonal to any free wave.
$\lp \pwb, e^{it\Delta}f\rp\rightarrow 0$, for any $f\in L^2$. This follows from minimal/maximal velocity bound on free solution.


\item Observe that the argument in  Section~\ref{sec:gamma-limit} is essentially about the linear part of the equation,  treating the nonlinearity perturbatively (since it has enough decay in exterior region).  So we can repeat the arguments there for $\pwb$, which verifies equation $(i\partial_t +\Delta )\pwb=\mathcal{N}(\phi)\phi$, and conclude that the results also hold for $\pwb.$

In particular, from Theorem~\ref{thm:r-limit},   for any $F=F(\XT), \alpha>\frac13$, $(F\gamma F\pwb, \pwb)$ has a limit.
Now \eqref{pwb-gamma-limit} implies that for $\alpha\geq\alpha_0$, the $\gamma$-limit is $0$.  So from Lemma~\ref{slowgrowth},~\ref{slowgrowth-converse} and~\ref{unique-defin}, we see that
\begin{align}(F\gamma F\pwb, \pwb)&\rightarrow 0, \hspace{1cm} \forall \alpha\in (\frac13, 1)\\
 (\pwb, |x|\pwb)&\lesssim t^\frac12. \label{pwb-slowgrowth}
\end{align}
\end{enumerate}

\item \textbf{Decomposition 2:}  $\tilde{\phi}_{wb}:= F(\frac{\jx}{t^{\alpha_0}}\leq 1)\phi(t)$, in this way, we have the asymptotic decomposition
\be
\|\phi(t) - \tpwb -e^{i\Delta t}\Omega^*_F\phi_0\|_{L^2}\rightarrow 0, \text{ as }t\rightarrow +\infty. \label{pwb-decomp}
\ee
 with  the following properties
\begin{enumerate}
\item $\tpwb$ satisfies the equation
\begin{align}
(i\partial_t +\Delta)\tpwb  = \widetilde{\mathcal{N}}\phi(t):= -i (D_HF)\phi(t) +F(\XTT)\mathcal{N}(\phi)\phi \label{eq:pwb-}
\end{align}
with global $H^1$ bound $\sup_{t\in [1,+\infty)} \|\tpwb\|_{H^1}<\infty$. 
  $ \widetilde{\mathcal{N}}$ verifies the decay assumption (\ref{N-decay}) using $\tilde{F}(\XT)$ with $\alpha>\alpha_0$.
\item
\be
\lp \tpwb, \tilde{F}(\XT)\gamma \tilde{F}(\XT) \tpwb \rp\rightarrow_{t\rightarrow +\infty} 0,\hspace{1cm} \alpha\geq \alpha_0\label{FrFwb}
\ee
This is simply because of definition of $\tpwb$ when $\alpha>\alpha_0$, and (\ref{0limit-boundary}) for $\alpha=\alpha_0$. Arguing as in  Section~\ref{sec:gamma-limit} but only for $\alpha>\alpha_0$, we conclude that \be (\tpwb, |x|\tpwb)\lesssim t^{\alpha_0+}.\ee

\item $\pwb$ is asymptotically orthogonal to any free wave.
$\lp \pwb, e^{it\Delta}f\rp\rightarrow 0$, for any  $f\in L^2$. This follows from minimal/maximal velocity bound on free solution.
\end{enumerate}

\item  \textbf{Decomposition 3:}  Another option is to define $\phi_f$ a solution to \eqref{Main-eq},  so that \be \|\phi_f -e^{it\Delta}\Omega^*_F\phi_0 \|_{L^2}\rightarrow 0.\ee
The existence of such function, relies on the fact \emph{all solutions are global for RSS Nonlinearities with global $H^1$ bound,  so it doesn't work in power type with potential, or time dependent potentials.}
Now we define $\pwl=\phi(t)-\phi_f(t)$.
\be \|\phi(t) - \pwl -e^{it\Delta}\Omega^*_F\phi_0 \|_{L^2}\rightarrow_{t\rightarrow +\infty} 0.\ee
  And this decomposition is similar to the first one.
\begin{enumerate}
\item  $\pwl$ has globally bounded $H^1$ norm, provided $\phi_f$ has. 
\item $\pwl$ verifies the equation
\be (i\partial_t+\Delta )\pwl ={\bf N}(\pwl+\phi_f)-{\bf N}(\phi_f) \ee
Such $\phi_f$ satisfies the global strichartz bound as a free wave. Then we can argue similarly that asymptotically $\pwl$ verifies the equation in the weak sense, i.e. for $\forall \eta \in L^2$
and it verifies the equation asymptotically.
\be \lp \eta,  (i\partial_t+\Delta )\pwb -{\bf N}(\pwl)\rp \rightarrow 0, \hspace{1cm} t\rightarrow +\infty.\ee
\be \left\|\pwl - e^{it\Delta}\pwl(0) -\int_0^t e^{i(t-s)\Delta}{\bf N}(\pwl)(s)ds\right\|_{L^2_tL^6_x  ([T,\infty)\times \R^3)}\rightarrow 0.\ee
If we further peal off the scattering part in the source term ${\bf N}(\pwl+\phi_f)-{\bf N}(\pwl) -{\bf N}(\phi_f)$,  by defining
\be \psi_{+}  = \int_0^\infty  e^{-i\Delta s}{\bf N}(\pwl+\phi_f)-{\bf N}(\pwl) -{\bf N}(\phi_f)](s)ds \ee
we  have
\be \left\|\pwl - e^{it\Delta}(\pwl(0)+\psi_{+}) -\int_0^t e^{i(t-s)\Delta}{\bf N}(\pwl)(s)ds\right\|_{L^2_x}\rightarrow_{t\rightarrow +\infty} 0.\ee
\item As in the decomposition 1, we have  \begin{align}(F\gamma F\pwl, \pwl)&\rightarrow 0, \hspace{1cm} \forall \alpha\in (\frac13, 1)\\
 (\pwl, |x|\pwl)&\lesssim t^\frac12.\end{align}
 \end{enumerate}
{\bf So in fact, we see that there is not much difference with decomposition 1, however, this one requires extra information about the equation: global existence and global boundedness, while decomposition one just attacks one single solution.}

\item \textbf{Decomposition 4:} Further refinement of decomposition is to define
\be  \pwl = F(\XTTZ)\pwb-F(\frac{\jx}{ t^{\frac12}\ln t}\geq 1)F_2(|\gamma| t^\frac12\geq 1)\pwb \ee
and use scattering for the second term.

With $\beta=1-\alpha_0$, then we have
\be \phi(t) = \pwl + e^{i\Delta t}\Omega^*_F\phi_0 + R(t) \ee
 $\|R(t)\|_{L^2}\rightarrow 0$, just like for $\pwl$.
\end{itemize}
Finally we   notice that if $\phi(t)$ is  WLS,   then  $\Omega_F\phi_0 =0.$ This follows from slow growth bound (\ref{half-growth})
\be
\|F(\XTZ)\phi(t)\|^2_{L^2} =\lp\phi(t),  \frac{\jx}{\jx}F^2(\XTZ)\phi(t)\rp \leq t^{\frac12-\alpha_0} \rightarrow 0.
\ee

\section{On weakly bounded states}
First let us recall the notation: by $\phi_{wls}$ we mean a solution to \eqref{Main-eq} with $\gamma$-limit being $0$.  By $\pwb$, we mean the weakly bounded part of the solution.  By $\pwl$, we mean the weakly localized part of the solution, which usually comes from pealing off extra parts from $\pwb$ that converges to $0$.

In this section, we will focus on $\phi_{wls}$ and $\pwb$, and prove additional properties.

\begin{prop}[Zero frequency channel] Let $F_p =F(|\bp|t^\beta\leq 1)$, then for any solution $\phi(t)$ to equation (\ref{Main-eq}),
$\lim_{t\rightarrow +\infty}\la F_p \ra$
exists for $\beta > \frac23$.

For WLS,  the limit exists for $\beta>\frac58$.
\end{prop}
\begin{proof} Since $[-i\Delta, F_p]=0$,  we have $D_HF_p =\beta F'_p |\bp|t^{\beta-1}$, which is a negative operator.  So, $F_p$ satisfies Assumption (A1)(A3), we are left to control the nonlinear term $I(t)=(F_p\phi, \mathcal{N}(\phi))$.

First, observe that if   we denote
 $G_p(|\bp|t^\beta \leq 1)= |\bp|^mt^{m\beta}F_p,$ then
 for any Schwarz function $f$,  $\widehat{G_pf(x)} =\chi (t^\beta |\xi|\leq 1)\hat{f}(\xi) $. Then \be  G_pf(x) = t^{-3\beta}\check{\chi}(\frac{x}{t^\beta})  * f,  \ee  using Young's inequality, this implies that $G_p:L^r\rightarrow L^r$ bounded for any $1\leq r\leq \infty.$

Now we write ${\bf N}(\phi)=\mathcal{N}(|x|,t,|\phi|)\phi$. For
any $ a\in (0,\frac32)$, using Hardy's inequality \eqref{Hardy},  we get
\begin{align}
|I(t)| =& |\lp \phi,  F_p \mathcal{N}\phi\rp | = |\lp   \phi,     F_p \bp^a \bp^{-a}\jx^{-a}\jx^a \mathcal{N}\phi\rp|\\
\leq &\|\phi\|_2 \| F_p \bp^a\|_{L^2\rightarrow L^2} \|\bp^{-a}\jx^{-a}\|_{L^2\rightarrow L^2} \|\jx^a \mathcal{N}\phi\|_{L^2}\\
\leq &  t^{-a\beta}\|\phi\|_2\|\jx^a \mathcal{N}\phi\|_2
\end{align}
Hence for $\beta>\frac23, $ we can find $a\in (0,\frac32)$ such that $a\beta>1$.  $F_p$ also satisfies Assumption (A2), and we conclude that $\lim_{t\rightarrow +\infty}\la F_p \ra$   exists.

Now for WLS, since $\||x|^\frac12 \phi\|_{L^2}^2= \la |x|\ra \lesssim t^{\frac12}$, we have
\begin{align}
|I(t)| =& |\lp \phi,  F_p \mathcal{N}\phi\rp | = |\lp  |x|^\frac12  \phi,   |x|^{-\frac12}|\bp|^{-\frac12} |\bp|^{a+\frac12}  F_p \bp^{-a}\jx^{-a}\jx^a \mathcal{N}\phi\rp|\\
\leq &  t^{-(a+\frac12)\beta}\||x|^\frac12\phi\|_2\|\jx^a \mathcal{N}\phi\|_2  \lesssim t^{-(a+\frac12)\beta+\frac14}
\end{align}
Hence for $\beta>\frac58,  I(t)\in L^1(dt)$, which implies $F_p$   satisfies assumption (A2) and the existence of limit $\lim_{t\rightarrow +\infty}\la F_p \ra$.
\end{proof}
\begin{rem}
The proof also shows that the free channel wave operator with support $F_p$ exists. Since, however, a free wave can not concentrate at zero frequency, it follows that this operator is zero.
That is, the solution goes to zero, in the strong $L^2$ sense, in the region supported by $F_p.$
\end{rem}

In the following proposition, we show that strictly localized solution is in fact smooth. 
\begin{prop}[{\bf Improvement of regularity}]
Let $\phi(t)$ be a global solution to (\ref{Main-eq}) satisfying the global energy bound (\ref{global-bound}),  and it is strictly localized, i.e.
 $supp \phi(t) \in B_K(0)$ for all $t\geq 0$. Assume that $A\phi(t)\in L^2$ uniformly bounded in time $t.$
 Assume moreover that either 
 \begin{align}[N-cond]
 &\quad \quad \textbf {(a)}\quad \|\nabla {\bf N}(\phi)\|_{L^{\infty}_tL^2_x}\lesssim 1\\
 & \textbf {or  (b)}  \int_0^{\delta} \|F(|x|\leq 1)\nabla {\bf N}(\phi)\|_{L^2_x}^{L^{1+\eta}} dt <\infty,
 &\delta,\eta \, \textbf{some small positive numbers}.
 \end{align}
 Then $\phi_0\in C^\infty$.
\end{prop}

  \begin{proof}
 Let us differentiate the equation (\ref{Main-eq}), and get the Duhamel formula for $\nabla \phi,$
  \begin{equation}
  \nabla \phi(t) = e^{it\Delta} \nabla \phi(0) -\int_0^te^{i(t-s)\Delta}\nabla{\bf N}(\phi)(s)ds.\label{Duhamel-dphi}
  \end{equation}
  Let $M\gg K>1$,  and $P_M$ be the Littlewood-Paley operator that projects onto the frequency $M$. Take $\chi(x)$ to be the characteristic function of $\{|x|\geq K+M^\frac12\}$.

  Notice that if we denote $f= \phi(t)\in L^2$ and $supp  f\in B_K(0)$, then
  \be \|\chi(x)P_M\nabla f\|_{L^2}=O(M^{-m}). \label{almost-local}\ee This is because $P_M\nabla f = F*f$ where $F(x)=M^4 \tilde{F}(M|x|)$ with $\tilde{F}$ being a  Schwartz function.  
  Take $F_1(x)=\tilde{F}_1(M^{-\frac12}|x|)$ with $\tilde{F}_1$ being a smooth characteristic function of $[-1,1]$ . Then $(FF_1)*f$ has support inside $\{|x|\leq K+M^{\frac12}\}$. Hence we get
 \begin{align}
 \|\chi(x)P_Mf\|^2_{L^2}= &\|\chi(x) (F(1-F_1))*f\|^2_{L^2}\leq \|F(1-F_1)\|^2_{L^1}\|f\|_{L^2}^2\\
  \lesssim& \|f\|_{L^2}^2 \int_{|x|\geq M^{\frac12}} M^4\tilde{F}(M|x|)dx =O(M^{-m})
 \end{align}
  Now by taking $t=M^{-\frac12}$   in (\ref{Duhamel-dphi}),  and we get
  %
  \begin{align}
 & \|\chi(x)P_M  e^{iM^{-\frac12}\Delta} \nabla \phi(0) \|_{L^2} \leq  M^{-m}+\\
  &  \|\chi(x)e^{iM^{-\frac12}\Delta}P_M\int_0^{\frac{1}{\sqrt{M}}}e^{-is\Delta} (\chi_r(|x|>1)+\bar\chi_r(|x|\leq1))\nabla{\bf N}(\phi)(s)ds\|_{L^2}\\
  \lesssim M^{-m}+ \\
  &\frac{1}{\sqrt{M}}  \|F(|x|\geq 1) \nabla{\bf N}(\phi)\|_{L^\infty_tL^2_x([0,M^{-\frac12}]\times \R^3)}\\
  &+\sup_{0\leq s \leq \frac{\epsilon}{\sqrt M}} \|\chi(x)P_M  e^{iM^{-\frac12}\Delta}F(|x|\leq 1) \nabla {\bf N}(\phi(s)) \|_{L^2}\epsilon/\sqrt M +\\
  &\frac{1}{\sqrt M}(\epsilon\sqrt M+K)^{-m}\|P_M{\bf N'}(\phi)\bar\chi_r(|x|\leq 1)\|_{l^{\infty}}\sup_s\|\nabla \phi\|_2\\
 & \lesssim \frac{1}{\sqrt{M}}  + (\epsilon /\sqrt{M})\sup_{0\leq s \leq \frac{\epsilon}{\sqrt M}}\|F(|x|\leq 1) \nabla {\bf N}\|_{L^2_x}.
  \end{align}
Then we get  \begin{align}
\|P_M   \nabla \phi(0) \|_{L^2} =& \|P_M  e^{iM^{-\frac12}\Delta} \nabla \phi(0) \|_{L^2}\\ \leq& \|\chi(x)P_M  e^{iM^{-\frac12}\Delta} \nabla \phi(0) \|_{L^2}  + \|(1-\chi(x))P_M  e^{iM^{-\frac12}\Delta} \nabla \phi(0) \|_{L^2}\\ \lesssim &\frac{1}{\sqrt{M}}
 \end{align}
 Here we used (support of $(1-\chi)$  is  $\in |x|\leq K+\sqrt M$),
  \be \|(1-\chi(x))P_M  e^{iM^{-\frac12}\Delta} \nabla \psi(0) \|_{L^2}  \lesssim M^{-\frac{m}{2}}\ee
  which is basically  minimal velocity bound (\ref{MinVB}), or can be proved directly by method of non-stationary phase.


Next, we show that the integrated condition (b) above is sufficient. In the radially symmetric case, the solution can only be unbounded at the origin. Therefore, we conclude that the above argument applies if we get a control of the solution in a ball around the origin.
To do that, we now apply a similar argument as above, to control the solution near the origin.

By Duhamel identity we have:
\be 
F(|x|\leq 1)P_M\nabla \phi(t)= F(|x|\leq 1) e^{i\Delta t} P_M\nabla \phi(0)-i\int_0^t F(|x|\leq 1) e^{i\Delta (t-s)}P_M\nabla {\bf N}(\phi(s))ds.
\ee
 For $Mt>>1$ we have that $F(|x|\leq 1) e^{i\Delta t} P_M\nabla \phi(0)= \mathcal O(M^{-m}).$
\be
\|F(|x|\leq 1) e^{i\Delta (t-s)}P_M\nabla \phi(s)\|\leq \mathcal O(M^{-m}) \,\textbf{for} |t-s|>>1/M,
\ee
 since by assumption ${\bf N}(\phi)$ is localized in space.
So, the main contribution from the Duhamel term is the integral over an interval of time of size $c/M.$ 

Therefore we get,
\begin{align}
&\|F(|x|\leq 1)P_M\nabla \phi(t)\|_2 \leq \mathcal O(M^{-m})+ \|\int_0^{c/M}  F(|x|\leq 1) e^{i\Delta u}P_M\nabla {\bf N}(\phi(t-u))du\|_2\\
&\leq  \mathcal O (M^{-m})+\left(\int_0^{c/M} du\right)^{1/2}\left(\int_0^{c/M} M^2 \|P_M {\bf N}(\phi)\|^2 du\right)^{1/2}\\
&\leq \sqrt {CM} \left(\int_0^{c/M} \|P_M{\bf N}(\phi)\|^2 du\right)^{1/2}.
\end{align}

  Hence by taking  $M_k=2^k M_0,  M_0\gg K^2>1$, we have  for $\alpha\in(0,\frac12)$
  \begin{align}
\||D|^{1+\alpha}\phi_0\|^2_{L^2} \lesssim M_0^{2\alpha} \|P_{\leq M_0}\nabla \phi(0)\|^2_{L^2}+ \sum_{k=1}^\infty  M_k^{2\alpha} \|P_{M_k}\nabla \phi(0)\|^2_{L^2}\lesssim M_0^{2\alpha}
\end{align}

Therefore, if $\phi, \nabla \phi$ are supported in $B_K(0)$ for $t\in [0,+\infty)$ and $\| \nabla|^{1/2}{\mathcal{N}}(\phi)\|_{L^2_tL^2_x([0,1]\times \R^3)}\lesssim 1$, we get $\alpha-$improvement of the regularity.  Moreover by time translation, we actually have
\be \sup_{t\in[0,\infty) }\||D|^{1+\alpha}\phi_t\|^2_{L^2}  \lesssim M_0^{2\alpha} .\ee

Now by taking $|\nabla|^{1+k\alpha}$ on equation (\ref{Main-eq})  with $k=1, 2,\ldots$, or equivalently multiplying by $M^{k\alpha}\nabla,$  we  can iterate the argument.  In every step, we need to check two things

(1) $ \|\chi(x)P_M |\nabla|^{1+k\alpha}\phi(t) \|_{L^2_x}=O(M^{-m})$. The proof is the same as for (\ref{almost-local}).

(2)$\| |\nabla|^{1+k\alpha}{\bf N}(\phi)\|_{L^\infty_tL^2_x([0,1]\times \R^3)}\lesssim 1.$  
\end{proof}

We now repeat this estimate for a general domain, located around $\sqrt M+nK$ in space.
It follows that:

\begin{prop}[Improvement of regularity]
Let $\phi(t)$ be a global solution to (\ref{Main-eq}) satisfying the global energy bound (\ref{global-bound}),  and it is  localized in the sense that
 $A \phi(t) \in L^2$ for all $t\geq 0$. Assume that $A\phi(t)\in L^2$ is uniformly bounded in time $t.$
 Then $\phi_0\in C^\infty$.
\end{prop}

\begin{proof}
Let $ \chi_n(x)$ be the smooth projection of $x$ at  the  $|x|>\sqrt M+nK.$ Then, we get from the above Duhamel formula:
\begin{align}
& \|\chi_n(x)P_M  e^{iM^{-\frac12}\Delta} \nabla \phi(0) \|_{L^2} \leq  M^{-m}(nK)^{-\sigma}+\\  &\|\int_0^{\epsilon/\sqrt m}\chi_n(x)e^{iM^{-\frac12}\Delta}P_M\bar{\chi_n} \nabla{{\bf N}}(\phi)(s)ds\|_{L^2}.\\
&\bar{\chi_n}\equiv F(|x|/( nK)=1), \quad\quad  F_1(|x|/(\sqrt M +nK) =1).
\end{align}
$\sigma$ stands for the decay rate in $L_x^2$ of the gradient of the localized solution.
The reason why the only interaction term supported in $ |x|\sim nK$ contributes is due to the minimal and maximal velocity bounds for free flow, at frequency (velocity) $M.$

It remains to estimate the size in $L^2$ of the Duhamel term.
A typical term in the interaction part (when $\phi$ is small in $L^{\infty}$ ) is of the form
$$
\phi \phi \nabla\phi,\, \text{or}\,\, W(x,t)\phi \phi \nabla\phi.
$$
 Therefore,
 $$
 \|\bar{\chi_n} \nabla{{\bf N}}(\phi)(s)\|_{L^2} \leq c( +nK)^{-2}  \|\bar{\chi_n} \nabla(\phi)(s)\|_{L^2}\|\phi\|^2_{H^1}.
 $$

 Furthermore, by the  \emph{assumed} bound on $A\phi$ we have that
 $$
  \|\bar{\chi_n} \nabla(\phi)(s)\|_{L^2}\leq  c( +nK)^{-1}.
 $$
 We do not add contributions from dyadic domains that are far out, around $\sqrt M +nK$, since they are incoming waves.
 Incoming waves cannot cancel out the outgoing waves coming from the linear part acting on the initial data.
 The above bounds show that the sum over $n$ can be taken, and we get the same bound as before, for each fixed $M.$

 For the domain containing the origin, we cannot use the $L^{\infty}$ bounds, unless the interaction is a bounded function.
  In this case, we need to trade some of the $t$ integration of the Duhamel term to gain regularity. This can be done as follows:

   \begin{align}
   &\|F_1(\frac{|x|}{K}\leq 2) P_M\int_0^{\frac{1}{\sqrt{M}}}e^{-is\Delta}\bar{\chi_0}|\phi|^q\nabla \phi(s) ds\|_2\leq\\
   &K^2 \int_0^{\frac{1}{\sqrt{M}}}\la \frac{1}{\sqrt{M}} -s\ra^{-3/2}\||\phi|^q\|_2\|\nabla \phi\|_2 ds\leq\\
   &K^2\frac{1}{\sqrt{M}}\||\phi|^q\|_2\|\nabla \phi\|_2\leq K^2\frac{1}{\sqrt{M}}\|\phi\|^{3}_{H^1}.
   \end{align}
   We assumed here that $2q=6.$
   The estimate can be adjusted to include other values of $q.$
   If $q=1$ a simple use of the Cauchy-Schwarz inequality gives a bound with a loss of $M^{-1/4}.$
A similar argument can be used for $q< 5.$

 For the part of the Duhamel source term which lives near the origin, we need a different kind of propagation estimate, namely \emph{local smoothing estimate}.
 We can then break the interaction term into two parts; one part is where $|x|\geq L, L\geqq 1,$ and the other part is with $|x|\leq L.$

 To complete the proof of the regularity proposition,
we control the Duhamel term with the localized source near the origin by the non-homogeneous Strichartz estimate and the local smoothing estimates, proved in the next subsection.
We have that
\begin{align}
&\|\chi(x)\int_0^{1/\sqrt M}P_M  e^{iM^{-\frac12+s}\Delta}F(|x|\leq 1)\nabla{\mathcal{N}}(\phi)(s)ds \|_{L^{\infty}_tL^2_x}\\
&\lesssim \|P_M F(|x|\leq 1)\nabla{\mathcal{N}}(\phi)(s)\|_{L^{2}_sL^{6/5}_x}\|\chi(x)\|_{L^3}\\
& \lesssim \|P_M |D|^{-b}|D|^{b}F(|x|\leq 1)|\phi|^{4-a}r^{2-a'}r^{-1/2+b'}\left< D\right>^{3/2-b}\phi\|_{L^{2}_tL^{6/5}_x}\\
&\lesssim M^{-b} \|F(|x|\leq 1)|\phi|^{4-a}r^{2-a'}\|_{L^3}\|F(|x|\leq 1)r^{-1/2+b'}\left< D\right>^{3/2}\phi\|_{L^{2}_tL^{2}_x}\\
&\lesssim c(\|\phi|_{H^1})M^{-b}.
\end{align}
Here, $b,a',b'$ are chosen positive and sufficiently small, depending on $a>0.$

When the solution is not compactly supported, the above analysis is essentially the same. As before, the most difficult domain is a neighborhood of the origin, so we focus on the is part.
We use the Duhamel representation \ref{Duhamel-dphi} as before, and we have:

 \begin{equation}
   \chi_0(|x|\leq K) P_M\nabla \phi(0)=-  \chi_0(|x|\leq K) P_Me^{-it\Delta} \nabla \phi(t)- \chi_0(|x|\leq K) P_M\int_0^te^{i(-s)\Delta}\nabla{\mathcal{N}}(\phi)(s)ds\label{Duhamel-dphi1}
  \end{equation}
  We take $t=1/ \sqrt M$, and estimate the $L^2$ norm of the LHS.
  The first term on the RHS is the part of $ P_M \nabla \phi(t)$ that propagates under the free flow, by time $1/\sqrt M$ into the support of $\chi_0(|x|\leq K).$
  By minimal and maximal propagation estimates this comes from $\chi(\sqrt M \leq |y|\leq \sqrt M +K) [ P_M \nabla \phi(t)(y)$.
  Therefore the contribution of this term to the $L^2$ of the RHS is:
  $$
 \| \chi(\sqrt M \leq |y|\leq \sqrt M +K)\la y\ra^{-1}y [ P_M \nabla \phi(t)(y)\|_2\leq (c/\sqrt M)\|A\phi(1/\sqrt M)\|_2 \lesssim  1/\sqrt M.
 $$
 Next, we estimate in a similar way the contribution of the Duhamel term.
 The main contribution comes from the source of the Duhamel term localized in a ball of size $K<<\sqrt M.$
 The contribution from this term is restricted to times $s\leq 2K/M,$ again by minimal and maximal velocity bounds.
 So the bound on the Duhamel term is the same as above, using local smoothing, but with a better bound, since $ 2K/M<<1/\sqrt M.$
 The contribution from any annular domain at distance $nK, n\geq 1,$ is smaller, since we get an extra decay of the source term:
 $$
 \| \chi(nK\leq |x|\leq (n+1)K) |\phi|^m\nabla \phi\|\lesssim (nK)^{-m}.
 $$
 For $m>1$ the sum over $n$ is finite.
 \end{proof}

 \subsection{ LOCAL SMOOTHING ESTIMATES}

 On the localized part of the solution, we gain an extra derivative from the integration over time.
 \begin{thm}[Local Smoothing]
 The solution of the Schr\"odinger equation satisfies the following local smoothing estimate:

\begin{align}
 & \int_{T_0}^{T} \|\jx^{-1/2-0}F_2(|p|/K_0\geq 1)|D|^{3/2} \phi\|^2 dt+\\
 &\int_{T_0}^{T} \|\jx^{-3/2-0}F_2(|p|/K_0\geq 1) \phi\|^2 \\
 &\lesssim ( I\|F_2(|p|/K_0\geq 1)\nabla\phi\|^2+\int_{T_0}^{T} \|\jx^{-1/2-0}F_2(|p|/K_0\geq 1)\nabla \phi\|^2 dt+\\
 &\lp \phi(T),F_K(|p|/K_0\geq 1)p^2 F_K(|p|/K_0\geq 1)\phi(T)\rp+\lp \phi(0),F_K(|p|/K_0\geq 1)p^2 F_K(|p|/K_0\geq 1)\phi(0)\rp.
\end{align}

 $I\equiv |T-T_0|.$

 \end{thm}

 \begin{proof}
 Let $\gamma_g \equiv (-i/2)[\jx^{-1}x\cdot \nabla+  \nabla \cdot \jx^{-1}x]$
 where the vector-field is chosen so that
 $$
 i[-\Delta, \gamma_g]=-\partial_l\, g_{lm}\,\partial_m+ a\jx^{-3-0}, a>0.
 $$

 $g_{lm}$ is a positive definite matrix function with $g_{lm}\geq \jx^{-1-0}.$

 The proof now follows by using the following PROB
 \be
 F_K(|p|/K\geq 1)\gamma_g F_K(|p|/K\geq 1).
 \ee
 The commutator with the Laplacian is a standard calculation, leading to positive terms.

 The commutator of $x$ with $F_K$ is of lower order:
 $$
 i[x,F_K] =K^{-1}\tilde F_K'.
 $$
From this we get the leading positive terms:
\begin{align}
& \int_{T_0}^{T} \|\jx^{-1/2-0}F_K\nabla \phi\|^2 dt+\\
&\int_{T_0}^{T} \|\jx^{-3/2-0}F_K \phi\|^2
\end{align}

 The lower order terms come from the commutator with the interaction term.

The commutator with the interaction term gives
\begin{equation}
\lp \phi, [F_2\gamma_g F_2,\mathcal {N}(\phi)]\phi\rp
\end{equation}

Then,

\begin{align}
&\lp \phi, F_K [\gamma_g, \mathcal{N}(\phi)]F_K \phi\rp=\\
&\lp \left< p\right>^{1-a}F_K\phi,\left< p\right>^{-1+a}  \mathcal {N'}r\jx^{-1}(\phi)(\partial_r) F_K\phi\rp,\\
&\gamma_g \mathcal {N}(\phi)\phi \equiv \mathcal{N'}(\phi) \gamma_g\phi.
\end{align}
Then, we use:
\begin{align}
&|r^{2-a}\mathcal {N'}F_K\phi| \leq c(\|\phi\|_{H^1})\\
&r^{1/2}F_Kr^{-1/2}= \mathcal{O}(1)\\
&\|\left< p\right>^{1-a}F_K\phi\|\leq K^{-a}\|\phi\|_{H^1}
\end{align}
 The control of the other parts of the commutator with the interaction term is similar.

 In this case, we note that $[F_K, \mathcal{N}(\phi)]\gamma_g F_K$ can be treated similarly, as one can commute the derivative part of $\gamma$
 through $F_K$ and the same for $r.$

The typical term in the interaction part is of the general form (for $\phi$ large, energy subcritical nonlinearity)
$$
|\phi|^{4-a}\,\, \text{or}\,\, W(x)|\phi|^{4-a}.
$$
We use the factor $x/\jx$ of  $\gamma_g$ to bound $|x||\phi|^2\leq c\|\phi\|^2_{H^1}.$
The leftover factors can now be estimated by Cauchy-Schwarz in the time variable.

 Extra factor $\jx^{-1}$  can come from the decay of $\phi$ for $x$ large.
Summing over a (dyadic) choice of $K$ one sees a gain of $a$ derivatives.
Then we iterate. The resulting smoothing estimate is order $a$ better (in number of derivatives) than the standard Morawetz.
We then use that estimate to control the interaction term, where now we redo the estimate with a=0.
In this way, we gain another factor of a. This can be repeated many times, as long as the LHS is controlled by the $H^1$ norm of the solution.
\end{proof}
\begin{rem}
The above form of error estimates is sufficient to control non-linear terms of order 4-a:
$$
\frac{|\phi|^{q+4-a}}{1+|\phi|^q}
$$
\end{rem}

 The above local smoothing estimate is NOT optimal.
 Since the vector-field defining the PROB, $\gamma_g$, vanishes linearly at the origin, the scaling dimension of the PROB is in fact zero, namely it is at the level of $L^2.$

 We lift it up by one half order through iteration and summing over $K_n,$ after multiplying by $K_n=2^nK_0.$
 This brings us to order $1/2.$
 But our assumption is that the solution is uniformly bounded in $H^1$, which is order 1.

 To get an optimal estimate we then need to have another half order. It is not possible to get another half derivative, even for the linear free equation.
 However, it is possible to get the extra 1/2 dimension in terms of (fractional) powers of $1/r.$

 The classical Morawetz estimate is an example, but in three dimensions we get no derivatives, just a delta function.

 So, we introduce a vector-field which is not as singular as the Morawetz one.

 \begin{thm}[Optimal Local Smoothing]
 The solution of the Schr\"odinger equation satisfies the following local smoothing estimate:

\begin{align}
 & \int_{T_0}^{T} \|\jx^{-a}|x|^{-1/2+0}F_2(|p|/K_0\geq 1)|D|^{3/2} \phi\|^2 dt+\\
 &\int_{T_0}^{T} \|\jx^{-3/2-0}F_2(|p|/K_0\geq 1) \phi\|^2 dt\\
 &\lesssim ( I\|F_2(|p|/K_0\geq 1)\nabla\phi\|^2+\int_{T_0}^{T} \|\jx^{-1/2-0}F_2(|p|/K_0\geq 1)\nabla \phi\|^2 dt+\\
 &\lp \phi(T),F_K(|p|/K_0\geq 1)p^2 F_K(|p|/K_0\geq 1)\phi(T)\rp+\lp \phi(0),F_K(|p|/K_0\geq 1)p^2 F_K(|p|/K_0\geq 1)\phi(0)\rp.
\end{align}

 $I\equiv |T-T_0|.$

 \end{thm}
\begin{proof}
 The proof follows the same steps as the smoothing theorem above, except that we use a different vector-field $\nabla g.$
 Let
 \begin{align}\label{gamma-g}
 &g(r)\equiv \int_0^r \frac{s ds}{\sqrt{s^2+s^{\theta}}}.\\
 & \nabla g= \frac{\overrightarrow{x}}{\sqrt{r^2+r^{\theta}}},\\
 &r=|\overrightarrow{x}|.\\
 &\theta= 2-\epsilon
 \end{align}

 Then, using as before the formula for the commutator of the laplacian with $\gamma_g$:
 \begin{equation}
 i[-\Delta,\gamma_g]= -\nabla_i g_{ij}\nabla_j -\Delta^2 g,
 \end{equation}

(summation over $i,j$ is implied),

the above local smoothing estimate follows, if we control the contributions of the Interaction term.
The estimate of the interaction term is similar to the previous case. Instead of pulling a factor of $r$ near zero from the vector-field $\nabla g$, we can pull a factor of $r^{(2-\theta)/4}$ and $r^{1-\theta/2}$ from the vector-field. Since $\theta$ can be arbitrarily close to 2, we can control $|\phi|^{4-a}$ near zero, for all $0<a<4.$
\end{proof}




 \subsection{Exterior Morawetz Estimate}

 \medskip

 Let $M\geq 100$  be a constant, we consider the PROB
   \be B=   F_1(\frac{\jx}{M}\geq 1)\gamma F_1(\frac{\jx}{M}\geq 1)\ee for $M\gg 1.$

 By assuming  $|F_1(\frac{\jx}{M}\geq 1)\mathcal{N}(\phi)|\leq M^{-k}$, we have
 $(B\phi , {\bf N}(\phi))=O(M^{-k})$, $k$ large (\emph{$k>3$ is sufficient}). This holds true if we have saturated nonlinearity with high power $p$, or a time-dependent potential that decays fast.

By direct computation we get
\begin{align}
D_H  F_1(\frac{\jx}{M}\geq 1)\gamma F_1(\frac{\jx}{M}\geq 1) =  \frac{4}{M}    \sqrt{F_1F_1'}  \gamma^2  \sqrt{F_1F_1'}   + F_1[-i\Delta ,\gamma] F_1+ \frac{1}{M^3}\tilde{F}_1(\frac{|x|}{M}\sim 1).
 \end{align}
Noice  that the $[-i\Delta ,\gamma]$ vanishes on the support of $F_1$, so by integration in time we get  the Morawetz type identity  in exterior domain
  \begin{align}
 \la F_1 \gamma F_1 \ra_{t_2} -  \la F_1 \gamma F_1 \ra_{t_1} =\int_{t_1}^{t_2} \frac{4}{M }  \la   \sqrt{F_1F_1'}  \gamma^2  \sqrt{F_1F_1'} \ra  +\frac{1}{M^3} \la \tilde{F}_1(\frac{|x|}{M}\sim 1)\ra +O(M^{-k}) ds. \label{EME0}
 \end{align}
 Here $\tilde{F}_1$ comes from symmetrization terms, and it is    supported in $supp F_1'$.

 Another version of this exterior estimate would follow by using the (smoothed) Morawetz Multiplier (\ref{EME0}) for $\gamma: $
 $$
 \gamma_g\equiv g(x)\cdot \nabla+ \nabla \cdot g(x),
 $$
 with $g(x)=\nabla G(x)$ is a vector-field such that the term $F_1[-i\Delta ,\gamma] F_1$ has a positive sign (only possible in 3 or more dimensions). In this case, in fact this can dominate the interaction term at least at large distances for decay of the interaction with $k>3.$

 \begin{prop}\label{A0-sequential-bound} For $\pwb$ defined by $\pwb=\phi(t)-e^{it\Delta}\Omega^*_F\phi_0$,
 there exists a sequence of time $t_n\rightarrow+\infty$, such that
\be \|A\pwb(t_n)\|_{L^2(|x|\leq \sqrt{t_n})}\lesssim1. \label{A0-sequential}\ee
\end{prop}
\begin{proof}
Take $F_1$ smooth characteristic function of $[1,+\infty)$ and $supp F_1\subset[\frac12, +\infty)$.  $F=F_1^2,\tilde{F}(\lambda) =\frac{1}{\lambda}\int_{-\infty}^\lambda F(s)ds$.  Notice that $\tilde{F}$ satisfies
\be \lambda \tilde{F}'(\lambda)+\tilde{F}(\lambda)= (\lambda \tilde{F}(\lambda))'= F(\lambda)\ee
and  $supp \tilde{F}\subset[\frac12, +\infty)$. $\tilde{F}$  is smooth, nonnegative and bounded by $1$.

Now for any $M>100$, we consider PROB acting on $\pwb$
\be  B_0=\jx \tilde{F}(\frac{\jx}{M}) , \quad\quad B_M=F_1(\frac{\jx}{M}\geq 1)\gamma F_1(\frac{\jx}{M}\geq 1).\ee
By direct computation we have
\be  D_HB_0= \gamma \left[\tilde{F}(\frac{\jx}{M})+\tilde{F}' (\frac{\jx}{M}) \frac{\jx}{M}\right] +h.c.= \gamma F+F\gamma =2\sqrt{F}\gamma\sqrt{F}\ee
Since $\sqrt{F}=F_1$, we get
\be \la B_0\ra_{t_2}-\la B_0\ra_{t_1} =2\int_{t_1}^{t_2}\la F_1\gamma F_1\ra_s ds \label{HB0} \ee

Now take constant $C_0= \lp \phi_0, \jx\phi_0\rp$. We claim that for each $M\geq 100,$ there exists a time $t_0=t_0(M)\in [0, M]$, such that
 such that $\la B_M\ra_{t_0}\geq -\frac{C_0}{M}$.    We prove the claim by contradiction.  Suppose this is not true,  then for any $M$,   $\la B_M\ra_{t} <  -\frac{C_0}{M}$ for $t\in [0,M]$. From \eqref{HB0} we get
\be
  \la B_0\ra_{M} =\la B_0\ra_{0}+ 2\int_{0}^{M}\la F_1\gamma F_1\ra_t dt\geq 0
\ee
which implies that
\be  2C_0 =2\int_0^M\frac{C_0}{M}ds \leq -2\int_{0}^{M}\la F_1\gamma F_1\ra_t dt \leq \la B_0\ra_{0} \leq \lp \phi_0, \jx\phi_0\rp=C_0,  \ee
thus we reach a contradiction.

Now for each $M\geq 100$, we integrate the exterior Morawetz \eqref{EME0} and get
\begin{align}
& \int_{t_0}^T \la F_1\gamma F_1\ra_t  - \la F_1\gamma F_1\ra_{t_0} dt\\=& \int_{t_0}^T\int_{t_0}^t    \frac{4}{M }  \la   \sqrt{F_1F_1'}  \gamma^2  \sqrt{F_1F_1'} \ra_{s}  +\frac{1}{M^3} \la \tilde{F}_1(\frac{\jx}{M}\sim 1)\ra_{s} +O(M^{-k})\,\,ds\, dt
\end{align}
Together with \eqref{HB0}, we get
\begin{align}
&  \int_{t_0}^T\int_{t_0}^t    \frac{4}{M }  \la   \sqrt{F_1F_1'}  \gamma^2  \sqrt{F_1F_1'} \ra_{s} \,\,ds\, dt  +(T-t_0) \la F_1\gamma F_1\ra_{t_0} +\frac12 \la B_0\ra_{t_0} \\= & \frac12 \la B_0\ra_{T}  - \int_{t_0}^T\int_{t_0}^t   \frac{1}{M^3} \la \tilde{F}_1(\frac{\jx}{M}\sim 1)\ra_{s} +O(M^{-k})\,\,ds\, dt
\end{align}
Now from \eqref{pwb-slowgrowth}, $\la B_0\ra_{T}\lesssim T^\frac12$.  Also our choice of $t_0$, implies either $\la F_1\gamma 
 F_1\ra_{t_0} \geq 0$, or $$-\frac{C_0}{M}\leq \la F_1\gamma F_1\ra_{t_0} <0.$$
Hence we obtain the estimate (when $\la F_1\gamma F_1\ra_{t_0}<0$)
\begin{align}
& \int_{t_0}^T\int_{t_0}^t    \frac{4}{M }  \la   \sqrt{F_1F_1'}  \gamma^2  \sqrt{F_1F_1'} \ra_{s} \,\,ds\, dt\\  \lesssim & T^\frac12 +\frac{T-t_0}{M} +\frac{(T-t_0)^2}{M^k}+ \int_{t_0}^T\int_{t_0}^t   \frac{1}{M^3} \la \tilde{F}_1(\frac{\jx}{M}\sim 1)\ra_{s} \,\,ds\, dt \end{align}
\emph{Notice that, the place that uses the $\la |x|\ra \lesssim t^{\frac12}$, is the first term, otherwise for $\la |x|\ra \lesssim t^{\alpha}$, we will have $\frac{t^\alpha M^3}{t^2}$, and there is problem with summation which means we can only sum up to $t^{\frac{2-\alpha}{3}}$, which is strictly less than $t^\frac12$. Or we have to put weight like $\frac{1}{\jx^{2\alpha-1}}$}.

Let us  multiply $\frac{M^3}{T^2} $ on both sides. Notice  that   $\gamma^2=\bp^2$ in exterior region and from \eqref{AB2A},  we  have $ \la   \sqrt{F_1F_1'}  A_0^2  \sqrt{F_1F_1'} \ra =  \la   A_0 F_1F_1'  A_0 \ra + \la \tilde{F}_1(\frac{\jx}{M}\sim 1)\ra$. Hence we get
\begin{align}
& \frac{1}{T^2}  \int_{t_0}^T\int_{t_0}^t      \la     A_0  F_1F_1' A_0 \ra_{s} \,\,ds\,dt\\
\lesssim &\frac{1}{T^2}  \int_{t_0}^T\int_{t_0}^t      \la   \sqrt{F_1F_1'}  A_0^2  \sqrt{F_1F_1'} \ra_{s}  + \la \tilde{F}_1(\frac{\jx}{M}\sim 1)\ra_s\,\,ds\,dt \\
\lesssim  & \frac{1}{T^2}  \int_{t_0}^T\int_{t_0}^t     M^2 \la   \sqrt{F_1F_1'}  \gamma^2  \sqrt{F_1F_1'} \ra_{s} \,\,ds\,dt  +\frac{1}{T^2}  \int_{t_0}^T\int_{t_0}^t     \la \tilde{F}_1(\frac{\jx}{M}\sim 1)\ra_s\,\,ds\,dt\\
\lesssim &
 \frac{1}{T^2}  \int_{t_0}^T\int_{t_0}^t    \la \tilde{F}_1(\frac{\jx}{M}\sim 1)\ra_{s} \,\,ds\, dt  +    \frac{M^3}{T^\frac32}+\frac{(T-t_0)M^2}{T^2}+\frac{(T-t_0)^2}{T^2}M^{3-k}  \end{align}

Now for $T$   large, and $M\leq \sqrt{T}$, $t_0\leq M\leq   \frac12 T$. On the left side,  we shrink the integration region to $\frac12 T\leq s\leq t\leq T $, on the right side, we enlarge the integration region to $0\leq s\leq t\leq T$, and we get

\begin{align} & \frac{1}{T^2}  \int_{\frac12 T}^T\int_{\frac12 T}^t     \|      \sqrt{F_1F_1'} A_0 \pwb(s)\|_{L^2_x}^2 \,\,ds\,dt  \\
\lesssim & \frac{1}{T^2}  \int_{0}^T\int_{0}^t   \|      \tilde{F}_1(\frac{\jx}{M}\sim 1)  \pwb(s)\|_{L^2_x}^2  ds\, dt+  \frac{M^3}{T^\frac32}+\frac{M^2}{T}+ M^{3-k} \label{A0-estimate}\end{align}

Now for each large time $T$ fixed, we take $M_0=100$, and
$M_j=2^jM_0$. Denote $j_0$ such that $M_{j_0-1}< \sqrt{T} \leq M_{j_0}$.  We perform estimate   (\ref{A0-estimate}) for each $M_j$, $0\leq j\leq j_0$, and sum up the estimates.
Notice that  $\sum_{j} \|      \tilde{F}_1(\frac{\jx}{M}\sim 1)  \pwb(s)\|_{L^2_x}^2  \lesssim \|\pwb\|_{L^2(|x|\leq \sqrt{T})}^2$ which is uniformly bounded,  so we get
\begin{align}     \frac{1}{T^2}  \int_{\frac12 T}^T\int_{\frac12 T}^t     \|   A_0 \pwb(s)\|_{L^2(|x|\leq \sqrt{T})}^2 \,\,ds\,dt  \lesssim 1 \label{intA0-bound}
\end{align}
Now by mean value theorem, we find $s_0\in [\frac12 T, T]$ such that
\be \|A_0\pwb(s_0)\|^2_{L^2(|x|\leq \sqrt{s_0})} \lesssim 1  \label{A-bound}\ee
Since $T$ is arbitrary, we find a sequence of  time $t_n\rightarrow +\infty$, such that the above bound holds. Hence we proved \eqref{A0-sequential}.
  \end{proof}
\begin{rem}The easier case is  when
there exists $M\geq 100$,  and $t_1\geq 0$, such that  \be \la F_1(\frac{\jx}{M}\geq 1)\gamma F_1(\frac{\jx}{M}\geq 1)\ra_t \leq 0,\quad \forall t\geq t_1.\ee
From \eqref{HB0}, we immediately have the solution is localized in the sense that
\be \lp  \pwb, |x|\pwb\rp\lesssim 1, \quad\quad  \forall t\geq 0.\ee
\end{rem}

\begin{rem}
Now we consider further propagation observables: \be  B_\alpha  =\jx^\alpha F_1\gamma +\gamma F_1\jx^\alpha= g_\alpha(x)\cdot \nabla  +\nabla\cdot g_\alpha(x) \ee  with $\alpha\in [0,1]$. Notice that  when $F_1\equiv 1$, for $\alpha=0,  $ $B_\alpha=\gamma.$  for $\alpha=1$, $B_\alpha=A$. For $\alpha\in (0,1)$ $B_\alpha$ is just interpolation of the two cases, with growing $g_\alpha$.

In this case, the commutator with $-i\Delta$ is positive; this is useful in cases where the solution is localized, or even weakly localized. The most general case is when $F_1=F_1(\frac{|x|}{M}\geq 1)$ and $\alpha \in (0,1)$.

\begin{lem}
Suppose that the solution of NLS is weakly localized, but not localized, in the sense that $\la F_1(|x|/t^{\alpha}\geq 1)\gamma  F_1(|x|/t^{\alpha}\geq 1)\ra$ converges to zero at $t\to \infty.$
Then, there exists a sequence of times $t_n$ going to infinity, such that $\la F_1(|x|/t^{\alpha}\geq 1)\gamma  F_1(|x|/t^{\alpha}\geq 1)\ra(t_n)>0$ for all $n.$
\end{lem}
\begin{proof}
By contradiction, if not, then for all $t$ large enough, the above expectation is non-positive. Then the integral is decreasing. It must be integrable in time, since the integral over time gives, up to a constant, the positive quantity  $\la F_1(|x|/t^{\alpha}\geq 1)|x|  F_1(|x|/t^{\alpha}\geq 1)\ra.$ 
Therefore the the solution is localized in the region $|x|\leq t^{\alpha}.$ But $\alpha$ is arbitrary.
\end{proof}
\begin{prop}
Let a weakly localized solution of the NLS as above, has\newline  $B(t_0)\equiv\la F_1(|x|/t^{\alpha}\geq 1)\gamma  F_1(|x|/t^{\alpha}\geq 1)\ra(t_0)>0.$ 
Then, for all $T$ large enough,
\be
B(T)=B(t_0)+\int_{t_0}^T \|F_1' \A\psi(t)\|^2 t^{-3\alpha} dt +\mathcal{O} (T^{1-3\alpha}c(\|\psi\|_{H^1})). \quad 
\alpha\neq 1/3.\ee
\end{prop}
\begin{proof}
 The proof follows directly from applying the standard argument for deriving PRES from the PROB: $F_1\gamma F_1,$
 and then using that the LHS of the Heisenberg equation, given by $\la F_1 \gamma F_1\ra_T-\la F_1 \gamma F_1\ra_{t_0}$ is a non-positive quantity: For $T$ large enough, the first term is arbitrary small. By choosing $t_0$ such that the second term is negative (possible due to the above Lemma), we conclude the non-positivity of the LHS.
 Therefore, the positive integral term on the RHS, is bounded by the symmetrization term, which gives the factor $T^{1-3\alpha}$
 and a higher order correction coming from the interaction terms, provided the decay is faster than $<x>^{-3}$ at infinity.
 For $\alpha=1/3$ we get a $\ln T$ bound.
 Finally, we note that $t^{-\alpha}\gamma F_1'\gamma=t^{-3\alpha}\A F_1'\A+ \mathcal{O}(t^{-4\alpha}).$
\end{proof}
\begin{cor}
On the support of $F_{1,\alpha}$ the quantity $\|F_{1,\alpha}\A\psi(t)\|$ is uniformly bounded on large time intervals in $[t_0,T].$
By integrating over $\alpha \in [\alpha_1,\alpha_2]$ similar estimate holds with a $\ln T$ factor:
\be\label{A bound}
\|F_(\frac{|x|}{t^{\alpha_1}}\geq 1)F_(\frac{|x|}{t^{\alpha_2}}\leq 1)\A\psi(t)\| \lesssim \ln{T_n},
\ee
for $t=T_n$ and $T_n$ can be arbitrary large.
This means in particular that $(\ln{1+|x|})^{-1} \A\psi(T_n)$ is uniformly bounded in $L^2.  $
\end{cor}
\subsubsection[AP]{\bf Analytic Projections \cite{Soffer-monotonic} }

Another class of multipliers is based on \emph{Analytic Projections.}

 Let  the outgoing projection
\be  P_M^+(A)\equiv \frac12(1+ tanh \frac{A-M}{R}),\hspace{1cm} \text{ for } M\gg 1, 2\leq R\lesssim M^{1-0} .\ee
Similarly we define the incoming projection
\be  P_M^-(A)\equiv \frac12(1 - tanh \frac{A+M}{R}),\hspace{1cm} \text{ for } M\gg 1, 2\leq R\lesssim M^{1-0} \ee
These operators project on $A$ positive or negative, provided that $|A|>M$, up to exponentially small corrections.

The crucial properties of such projections are
\be  [-i\Delta, P_M^{\pm}(A)] =\pm G_M^2(A) \ee
with $G_M(A)$ is an explicit and localized around $A\approx M $. Furthermore, we notice that
\begin{align}
P_M^+(A)\chi(|x|\leq K) = P_M^+(A) \la A\ra^{-2m} (1+A^2)^m\chi= O(M^{-2m})P_M^+(A)O(K^{2m})O(P^{2m})
\end{align}
Therefore, the commutator of $P^{\pm}_M(A)$ with a localized smooth interaction term,  and a smooth function, is higher order in powers of $M$. \end{rem}


\subsection{New propagation estimate}

  In this section, we deal with the dilation operator $A$

\subsubsection{Properties of $f(A)$}
For dilation operator  $A$, let $\tanh (A/R)$ be defined by the spectral theorem.  Recall the results of~\cite{Soffer-monotonic}.
\begin{lem}For $R>\frac{2}{\pi}$
\be \tanh (\frac{A}{R}): D(-\Delta)\rightarrow D(-\Delta)\ee
\end{lem}
\begin{lem} For $R>\frac{2}{\pi}$
\be [-i\Delta, \tanh(\frac{A}{R})] = 2\bp g^2(\frac{A}{R}) \bp \geq 0.\ee
Furthermore $g$ is explicit: (we take $R>>1$)
\be  [-i\Delta, \tanh(\frac{A}{R})]=\bp\frac{2\sin(2/R)}{\cosh(2A/R)+\cos(2/R)}\bp\gtrapprox \bp\frac{4}{R\cosh^2(A/R)}\bp .\ee
Moreover, we have for analytic function $F$ in a sufficiently wide strip around the real axis, that
\begin{align}
i[p, F(A) ]=&  i[ F(A-i)p-F(A)p] = i p[F^*(A+i)-F^*(A)]\\
i[x, F(A) ]=&  i [F(A+i)x-F(A)x] = i x[F^*(A-i)-F^*(A)]\end{align}
\end{lem}

 \subsubsection{Boundedness of $A$}

 Therefore,  for $P_{M,R}^{+}(A) =\frac12 (1+\tanh(\frac{A - M}{R}))$ (similarly the incoming projection  $P_{M,R}^{-}(A) =\frac12 (1-\tanh(\frac{A + M}{R}))$
\be  [-i\Delta,P_{M,R}^{+}(A)]\gtrapprox\bp\frac{1}{R \cosh^2(\frac{A-M}{R})}\bp \ee
This implies that for propagation observable $B=AP_{M,R}^{+}(A)$, we have
\begin{align}
\partial_t \la B\ra \gtrapprox  4\la \bp P^+\bp\ra +\la \bp\frac{A}{R \cosh^2(\frac{A-M}{R})}\bp \ra + \la [i\mathcal{N}(x,t), AP_{M,R}^{+}(A) ]\ra
\end{align}
This follows from the identity
\be
R^{-1}(A+i) \tanh((A+i)/R)-R^{-1}(A-i)\tanh((A-i)/R)=\frac{2i}{R}\frac{A\sin(2/R)+\sinh(2A/R)}{\cosh(2A/R)+\cos(2/R)}.
\ee

For general time-dependent potential, or nonlinearity $\mathcal{N},$ the first two terms are positive modulo corrections that are exponentially small. If $M\gg 1, R=\sqrt{M}, $ then the second term is negative for $A\leq 0 $ or $A-M\leq -M$. But then
\be  \left|\frac{A}{R \cosh^2(\frac{A-M}{R})}\right|\lesssim |\frac{A - M}{R}|e^{-2(\frac{A-M}{R})} \ee
for $|A-M|\geq M$, so the largest value is when $|A_M|=M$, which implies that  when taking $R=\sqrt{M}$,
\be  \left|\frac{A}{R \cosh^2(\frac{A-M}{R})}\right|\lesssim |\frac{M}{R}|e^{-2(\frac{M}{R})}\lesssim \sqrt{M} e^{-2\sqrt{M}}  \ee
So we ignore for the moment the exponentially small corrections,

It remains to estimate $ \la [i\mathcal{N}, AP_{M,R}^{+}(A) ]\ra$
\begin{align}
i[\mathcal{N}, AP^+ ] = i[\mathcal{N}, A]P^+ +iA [\mathcal{N}, P^+]= -x\cdot \nabla \mathcal{N} P^++iA [\mathcal{N}, P^+]
\end{align}
Now we use that $\mathcal{N}$ is localized and smooth ( and $\psi$ is localized and smooth Case I), we also use that $P^+$ is projection on $|A|\sim |x\cdot p| \geq M-\sqrt{M}$; for this we use
\begin{align}
P^+(A)= &P^+(A)\la A\ra^{-2m} (1+A^2)^{m}\lesssim  P^+(A)\la A\ra^{-2m} (1+ c_1 |x|^{2m} +|p|^{2m})\\
\lesssim & O(M^{-2m})(1+ c_1 |x|^{2m} +|p|^{2m})
\end{align}
Then we can cancel the powers of $|x|$ by $V$ or $\psi; $ similar for derivatives, by the $H^1$ property.

For example
\begin{align}
&\la  \psi,  \frac{|\psi|^p}{1+|\psi|^q}P^+(A)A\psi\ra =\\
& \la \psi, F(|\psi|) (1+ c_1 |x|^{2} +|p|^{2}) \la A\ra^{-2} P^+(A)A\psi\ra\\
& \lesssim   \la \psi, (1+ |x|^{2} ) F(|\psi|)  O(M^{-2}) A\psi\ra + \la \Delta \psi,  F(|\psi|) O(M^{-2}) A\psi \ra +\text{similar}.
\end{align}

In case of only $H^1$ regularity, we use $(A+i)^{-1}$ instead of $\la A\ra^{-m} \la A\ra^m.$
So the propagation estimate that follows is 
\begin{align}
& \la AP_{M,R}^+(A)\ra(T) - \la AP_{M,R}^+(A)\ra(0)= \\
 &\frac{1}{R}\int_0^T \la \bp\psi, \frac{A}{\cosh^2(\frac{A-M}{R})}\bp\psi\ra ds +\int_0^T \la \bp \psi(s), P_{M,R}^+(A)\bp\psi \ra ds +\\
 &O(M^{-2m})\int_0^T\sum_j\la D^j\psi, \tilde{\mathcal{N}}_j P_{M,R}^+(A) A\psi\ra dt ,
\end{align}
where $\tilde{\mathcal{N}}_j$ involves $\jx^\sigma \mathcal{N}, D^\alpha \mathcal{N}$.

In case we have $j=1$, the bound is in terms of powers of $\|\psi\|_{H^1}$, where we used $\||x|\psi\|_{L^\infty}\leq \|\psi\|_{H^1}$ for $|x|\geq 1$, and
 $||x|^\frac12\psi \chi(|x|\leq 1)|\leq C\|\psi\|_{H^1}$.

Finally we conclude that : if
$|\la \psi(t_n), AP^+ \psi(t_n) \ra|\leq C$, then
\begin{align}
&\frac{1}{T}\int_0^T \la \psi(s), P_{M,R}^+(A)\psi(s)\ra ds\lesssim  O(M^{-2m}) +\frac{C}{T}+\\
&\frac{1}{R}\int_0^T \la \bp\psi(s), \frac{|A|}{\cosh^2(\frac{(A-M)}{R})}\bp\psi(s)\ra ds \leq O(e^{-\sqrt{M}} +O(M^{-2k}))T+const.
\end{align}
for all T.

It also follows that
\begin{thm}
\be  \left| \la \psi(s),A P_{M,R}^+(A)\psi(s)\ra(T) \right| \leq \left| \la \psi(s),A P_{M,R}^+(A)\psi(s)\ra(0) \right| +O(TM^{-2m})  \ee
\end{thm}
Therefore, for $T\lesssim \la M\ra^{2m},$
\be  \left|\la AF(A\geq M)\ra(T)\right| \lesssim 1. \ee

This means that for $ A\geq t^{1/2m}, \left|\la AF(A\geq M)\ra(T)\right| \lesssim 1.$

For incoming waves the estimates are similar
\be  \left| \la \psi(s), P_{M,R}^-(A)\psi(s)\ra(T) \right| \leq C_0+O(TM^{-2m})  \ee
in this case, we do not need to know $ \left| \la \psi(s), P_{M,R}^+(A)\psi(s)\ra(t_n) \right| $ is bounded, since it is negative up to small exponential corrections.
\begin{cor}
If at some time $t_0$ we have $\left| \la \psi(s),A^LP_{M,R}^+(A)\psi(s)\ra(0) \right|\lesssim 1,$ then
$$
\left| \la \psi(s),A^L P_{M,R}^+(A)\psi(s)\ra(T) \right|\lesssim O(TM^{-2m'}),
$$
for all $T$ and $L$ large, depending only on $m.$
\end{cor}
\begin{rem}
In fact, one can iterate to estimate the nonlinearity. Take $m=1$ in previous bounds, then the nonlinear term is bounded by
\begin{align}
& \int_0^Tdt \la  [\mathcal{N}, A]P^+\ra +\la A \mathcal{N} P^+ \ra +  \la A  P^+ \mathcal{N} \ra\\
\lesssim &  \int_0^Tdt \la \bp\psi, \la \bp\ra^{-1} A \jx^{-1}[\mathcal{N}, A]\jx^{-1} \la \bp\ra^{-1} \la A\ra^{-1}P^+  \bp\psi\ra\\
\lesssim &M^{-1}  \int_0^T dt  \|p\psi\|_{L^2} \|P^+  \bp\psi \|_{L^2} \|\la p\ra^{-1} A [\mathcal{N}, A]\jx^{-2} \la \bp\ra^{-1} \la A\ra^{-1}\|_{L^2\rightarrow L^2}\\
\lesssim & \sqrt{T}M^{-1}\left( \int_0^T \|P^+  \bp\psi \|^2_{L^2}dt\right)^{\frac12}\\
\lesssim & \sqrt{T}M^{-1}(\frac{T}{M})^\frac12 = CTM^{-\frac32}
\end{align}
Hence we gain $M^{-\frac12}$, and we can iterate to gain more powers.  After iteration, take $M\sim t^\epsilon, $
 and we get that
 \be  \la \psi, |A|\psi\ra(t)\lesssim t^\epsilon \ee
 this is because $|A|\lesssim AP^+_M(A)-AP_M^{-}(A)+|A|P_M^0(A) $, where $P_M^0(A) =P(|A|\leq M)$.
 The details are postponed to the next sections, where we introduce the necessary high-low decompositions.
\end{rem}

 So we can conclude that
 \begin{thm} For solution $\psi$ to equation (\ref{Main-eq}), satisfying the global bound (\ref{global-bound}), we have
 \begin{equation}
 \la |A|\ra \lesssim t^0.
 \end{equation}
 \end{thm}

 We extend this result by a different argument to prove that:

 \begin{thm} For solution $\psi$ to equation (\ref{Main-eq}), satisfying the global bound (\ref{global-bound}),  the asymptotic localized part (with frequencies away from zero)  satisfies
 \begin{equation}
 \la |A|^2\ra \lesssim 1.
 \end{equation}
 \end{thm}
 \begin{proof}
 We begin with some preliminary comments.
 We know that the localized part of the solution in supported in $|x|\lesssim t^{\alpha}, \alpha\leq 1/2.$ We are interested in the part of the solution with large frequency, in order to prove smoothness of the localized part.
 So we will restrict our attention to the region of phase-space where the solution is localized in space and has a frequency larger than $K>1.$. This will be done by estimating the quantity $J_K\phi(t),$ with $J_K$ microlocalizing the solution in the desired region.

 We know that there exists a sequence of times along which $A^2$ has a bounded expectation in the region where $|x|\lesssim t^{1/2}.$
 More generally, we know, using \eqref{A bound} that the bound on the expectation of $|A|$ holds for all $|x| \lesssim t$ with a log loss in time.
 Therefore, on the support of $J_K,$  the expectation of $r^2$ will be bounded on a sequence, since $\A^2=r^2D^2$ plus lower order terms, and $K>1.$
 The idea of the proof to get the bound at all times is to use the above boundedness on a sequence with a PROB that is of the same order as $\A^2$ but which is monotonic to  leading order.
 
 This is done by proving a \emph{weighted ($r^2$) outgoing estimate}  . A similar estimate for the incoming waves is simpler.
 \begin{rem}
The propagation estimate we derive is not done in the usual way; rather, we will get directly a pointwise estimate, based on the knowledge and properties of the R.H.S. of the estimate. For this, we use a separate argument based on the sign of the L.H.S.
In this estimate the PROB(the multiplier) is $r^p$ with $p=2.$ There are related bounds for the Wave Equation due to the method introduced by Dafermos and Rodnianski \cite{}.
There are important differences, since we are dealing with the Schr|"odinger equation. The main estimate is derived on outgoing waves, and the incoming waves estimate is elementary. Here it is used to control the weakly localized solutions, not free waves.
The $r^p$ weights are used directly on microlocal projections on the outgoing/incoming waves, since there is no simple relation between the energy functional and the projections in the case of wave equations.
\end{rem}

  Consider then the following propagation estimate: For each $M>M_0$ with $M_0$ large, we use the PROB
 \be
 B_{KM}= J_K (r^2P^+_M +P^+_M r^2) J_K.
 \ee
 \be
J_K=F_1(\frac{|x|}{t^{\alpha}}\leq 1)F_2(|p|\geq K)F_1(\frac{|x|}{t^{\alpha}}\leq 1).
 \ee
 Now, observe the behavior of the operator $(r^2P^+_M +P^+_M r^2)$ under the full flow. The commutator with the Laplacian is positive, of order $\A P^+_M.$ The commutator with the non linearity is bounded and decays in $M$ depending only on the decay rate in $x$ and regularity of the nonlinear terms. The factor $r^2$ can be bounded by a factor $|\psi|^2$ in the 3 dimensional radial case.

 Therefore, in leading order, we get the desired estimate that the integral of the leading term is uniformly bounded (by the LHS) and is monotonic in $T.$

 For this argument to work, we must first keep the frequency away from zero (so that $r^2$ can be bounded by $\A^2$).
 Furthermore, we need to localize $|x|\leq t^{\alpha}$ by either $\alpha=1/2$  or $ 1.$
 The commutators with the frequency localization are easy to control: the commutator with the nonlinear term is bounded and of some extra decay in $M$, if the interaction term is localized better than $x^{-2}$ at infinity.
 So, the key complication is the commutator of the Laplacian with the $x$ localization, the function $F_1,$ which gives a negative term.
 
There is a sequence of times going to infinity along which the expectation of this operator is uniformly bounded.
If, by contradiction, there is a sequence of times along which this expectation goes to infinity, there must be an infinite sequence of maxima going to infinity.
So, by continuity, at each time $t_0$ the aforementioned expectation will be either a maximum point, or has negative derivative w.r.t. time, or positive derivative.

Next, we compute this derivative.
Denote the expectation of $B_{KM}$ by
\be
\lp J_K\phi(t),(r^2P^+_M +P^+_M r^2) J_K \phi(t)\rp=B_{KM}(t).
\ee

Then we have
\begin{align}
& \partial_t B_{KM}(t)=\lp J_K\phi(t),i[-\Delta,(r^2P^+_M +P^+_M r^2)] J_K \phi(t)\rp+\\
& \lp i[\Delta,J_K]\phi(t),(r^2P^+_M +P^+_M r^2) J_K \phi(t)\rp+\lp J_K\phi(t),(r^2P^+_M +P^+_M r^2)i[-\Delta, J_K ]\phi(t)\rp+\\
& \lp \phi(t),i[\mathcal{N}(\phi)(t),J_K(r^2P^+_M +P^+_M r^2) J_K] \phi(t)\rp +D_t J_K.\\
&D_t J_K\equiv \lp (\partial_t J_K)\phi(t),(r^2P^+_M +P^+_M r^2) J_K \phi(t)\rp +\lp J_K\phi(t),(r^2P^+_M +P^+_M r^2)(\partial_t  J_K )\phi(t)\rp.
\end{align}
The first term on the RHS is :
\begin{align}
C_1(t)&= \lp J_K\phi(t),(r^2i[-\Delta,P^+_M] +i[-\Delta,P^+_M] r^2+8AP^+_M) J_K \phi(t)\rp=\\
& \lp J_K\phi(t),((2A^2+cA+c')\tilde P^+_M/M +8AP^+_M) J_K \phi(t)\rp.
\end{align}
Here $\tilde P^+_M=F(\frac{A-M}{R}\thicksim 0)$ is a positive exponentially localized bump function around M of the operator $A.$
It is asymptotically proportional to $1/\cosh(\frac{A-M}{R}).$
The second term on the RHS is coming from the commutator of the Laplacian with the cutoff function $F_1(\frac{|x|}{t^{\alpha}}\leq 1).$
It is therefore dominated by
\begin{align}
C_2(t)\lesssim t^{\alpha}\|F^{'}_{1,\alpha}|\gamma|^{1/2}P^+_M F_2(|p|\geq K)\phi(t)\|_2^2+ L.O.T.
\end{align}
The factor $t^{\alpha}$ comes from $r^2F^{'}_{1,\alpha}/t^{\alpha}$; the factor $F^{'}_{1,\alpha}/t^{\alpha}$comes from the commutator of J with the Laplacian.

Here, $ F^{'}_{1,\alpha}=F^{'}_{1,\alpha}(|x|\thicksim t^{\alpha}).$

The contribution of the $D_t J_K$ term is similar but with a good sign and decays faster in time.

We will show that these terms, $C_2(t),D_tJ_K$ vanish pointwise in time as time goes to infinity.

The third term is the contribution of the Interaction part. We will consider the more difficult type of term, a nonlinear monomial in $\phi.$
These terms are the expectation of expressions of the form ( coming from commuting the interaction term with a factor $J_{K,M}$):
\begin{align}
&\lp \phi(t), |\phi(t)|^{m} J_{K} r^2 P^+_M(|p|+i)(|p|+i)^{-1} J_{K}\phi(t)\rp=2\Re\lp (|p|\phi), \phi^*|\phi(t)|^{m-1} |p|^{-1}J_{K} r^2P^+_M J_{K}\phi(t)\rp+L.O.T.\\
&=2\Re\lp (|p|\phi),(r^2(\phi^{\sharp})^{m})J_{K}(P^+_M)\lp A\rp^{-1/2}\lp A\rp^{+1/2}(J_{K}|p|^{-1})\phi(t)\rp +L.O.T.
\end{align}
$L.O.T.$ stands for lower order terms, that we get by commuting the various factors above through each other.

We do NOT commute through non-linearity. All other commutators have higher order power, in one or more of $1/t^{\alpha},1/M, 1/K.$
$\phi^{\sharp}$ stands for $\phi$ or its complex conjugate.
This is due to the fact that the commutator $[\partial_x,x]=1,$ so commuting with $J_K$ gives a gain of $1/K$ or $1/r.$ Similarly, commuting $x,\partial_x$ with $A$, gives back $x,\partial_x.$ So we gain a factor of $A$ or equivalently $1/M.$

So, the expectation w.r.t. $\phi(t)$ gives:

\be
C_3(t)\lesssim M^{-1/2}K^{-1} \lp (D\phi), \mathcal{O}_{\phi}(1)J_{K}|A|^{1/2}P^+_M J_{K}\phi \rp +L.O.T.
\ee
\be
 \mathcal{O}_{\phi}\lesssim  f(\|\phi\|_{H^1}).
 \ee
In the above last estimate we used that by Sobolev radial embedding, we can control the factor $r^2$ by $\phi^2$ at infinity.
We can also control the local singularities of each power of $\phi^2$ at the origin by a factor of $r$ near the origin, up to the critical energy power $|\phi^4$ in three dimensions.
By Cauchy-Schwarz Inequality, we bound the above expression by:
\be
C_3(t)\lesssim M^{-1}K^{-2}\|(D\phi)\|^2 +\delta \||A|^{1/2}P^+_M J_{K}\phi\|^2
\ee
\be
\delta\leqq 1/2.\ee

Since the $\delta$ term is bounded by the leading positive term, $C_1$, it follows that
\be
C_1(t)\lesssim M^{-1}K^{-2}+ C_2(t) +L.O.T.+ LHS(t)\ee
\be
LHS(t)\equiv \frac{d}{dt} B_{K,M}(t).
\ee
To get a {\bf pointwise in time} estimate, we will need to control $ C_2(t) +L.O.T.+ LHS(t).$

The control of the first two parts will follow from propagation estimates, with improved time decay.

To deal with the LHS, we use a different argument: recall that the LHS is the derivative of the expectation of the PROB w.r.t. time.

The left-hand side is a bounded function at each t. It is also differentiable, since its derivative, the RHS, is bounded for each $t$.

The boundedness of the RHS follows since we project on $|x|\leq t^{\alpha},$ and the assumption that the solution is in $H^1.$
Therefore, the LHS (t) is a function that oscillates between some finite numbers, but can be arbitrarily large, at sometimes between the minima.

So, for each time $t_0$, its derivative, being bounded, is either positive, negative, or zero.

Consider first the case where the derivative is zero at $t_0.$ (Note that there is an infinite sequence going to infinity with this property.)
We conclude that at these times the RHS is zero, and therefore we get a bound on
$$
C_1(t)\lesssim M^{-1}K^{-2}+ C_2(t) +L.O.T.
$$
Since, as explained above, the L.O.T. come from extra commuting, with a gain of powers of $K$ and $M,$
the above inequality holds with $L.O.T.$ dropped.

We will see that the PRES on the support of $F^{'}_{1,\alpha}(|x|/t^{\alpha}\thicksim 1)P^+_M F_{2}(|p|\geq K)$ decays fast, in time and in $K,M.$
We finally have then;
\be
C_1(t)\lesssim M^{-1}K^{-2}
\ee

Clearly, then, if $t$ is such that the LHS is negative, the same bound holds (in fact, better).
Note that these estimates are on the \emph{derivative} of the (expectation) of the PROB, and we may derive a PRES without the usual integration over time. As such, this estimate is pointwise in time.

 It remains to estimate the case where the LHS is positive. These are time intervals, which begin at time $t_m$ when the PROB is at minimum, and end at the time of the next maximum, at time $t_M.$ At both such times the derivative w.r.t. time is zero.

In this case, the above estimates do not give a bound on $C_1(t)$, except for a short time interval before the time of the critical point, where the derivative is zero.

To proceed, we now compute for such times the derivative of the expression $C_1(t).$

If we denote by $t_M$ the time when the LHS is zero (at maximum), then we are looking at times less than $t_M.$
Clearly then, if the derivative $\frac{d C_1(t)}{dt}\geq 0,$ then the upper bound on $C_1(t)$ above also hold, since it holds at time $t_M>t.$

Therefore, it is left to estimate the case when the derivative of $C_1(t)$ is negative or zero.
We then get that for such $t$
\begin{align}
&\frac{d C_1(t)}{dt}=D_t \lp J_K\phi(t),((2A^2+cA+c')\tilde P^+_M/M +8AP^+_M) J_K \phi(t)\rp=\\
&\lp J_K\phi(t),i[-\Delta, A\hat P^+_M] J_K \phi(t)\rp+\lp \phi(t),i[\mathcal{N},J_K A\hat P^+_M J_K] \phi(t)\rp+D_H J=\\
&\lp \nabla J_K\phi(t),\check{ P}^+_M J_K \nabla \phi(t)\rp  +\lp D_r\phi,r^2\mathcal{N}J_K (A+i)^{-1}\hat P^+_M J_K D_r\phi(t)\rp+D_H J +l.o.t.=\\
&\lp \nabla J_K\phi(t),\check{ P}^+_M J_K \nabla \phi(t)\rp+\lp J_K \mathcal{O}(1)_{\phi}D_r\phi,(A+i)^{-1}\hat P^+_M J_K D_r\phi(t)\rp+D_H J +l.o.t.\\
&|\lp J_K \mathcal{O}(1)_{\phi}D_r\phi,(A+i)^{-1}\hat P^+_M J_K D_r \phi(t)\rp |\lesssim\\
&\mathcal{O}(1/M) \|(A+i)^{-1}M\hat P^+_{M/2}J_K \mathcal{O}(1)_{\phi}D_r\phi\|_{L^2}\|\hat P^+_M J_K D_r\phi(t)\|_{L^2} +l.o.t.\\
&D_H J\equiv \lp (D_H J_K)\phi(t),((2A^2+cA+c')\tilde P^+_M/M +8AP^+_M) J_K \phi(t)\rp+\\
& \lp  J_K \phi(t),((2A^2+cA+c')\tilde P^+_M/M +8AP^+_M)(D_H  J_K) \phi(t)\rp.
\end{align}
We also used that $|r^2\mathcal{N}|\lesssim 1$, which holds if
\begin{align}
& \mathcal{N}\lesssim |\phi|^4 +cr^{-2}, \quad \quad r\leq 1,\\
&\mathcal{N}\lesssim |\phi|^2 +c<r>^{-2}, \quad \quad r\ge r_0>>1.
\end{align}
Here we used that:
\begin{align}\label{commutators1}
&i[-\Delta,(r^2P^+_M +P^+_M r^2)]=8AP^+_M+r^2\bp G^2(A)\bp +\bp G^2(A)/R\bp r^2=\\
&8AP^+_M+2A^2G^2(A)/R+(c_1A/R+c_2)G^2(A)+(d_1+d_2A/R^2)G'/R=\\
&8AP^+_M+2A^2G^2(A)/R+ (cA/R+d)\tilde G(A).\\
&G^2\geqq \cosh ^{-1}(\frac{A-M}{R}).\\
&(d_1/R+d_2A/R^2)G'\thicksim (d_1/R+d_2A/R^2) \cosh ^{-1}(\frac{A-M}{R}).
\end{align}
For large $M$, the first two positive terms on the RHS of  \ref{commutators1} dominates. so that this expression is positive, modulo exponentially small terms in $M.$
When taking the derivative of $C_1$ we commute again this expression with $-\Delta.$
The result is
\begin{align}
&i[-\Delta,8AP^+_M+2A^2G^2(A)/R+(c_1A/R+c_2)G^2(A)+(d_1+d_2A/R^2)G'/R]=\\
&8\bp (P^+_M +AP_M'/R)\bp +2\bp(2AG^2/R+G'A^2/R^2)\bp +(d_1+d_2A/R^2)\tilde G'/R^2.\\
&R\equiv M^{1-a}; \tilde G'\thicksim  \cosh ^{-1}(\frac{A-M}{R}); a<1.
\end{align}

We also used that
\be
\hat P^+_{M/2}P^+_M=P^+_M+\mathcal{O}(e^{-M/R}).
\ee
By applying Cauchy-Schwarz inequality as before, we control the leading terms by
\be
\lp \nabla J_K\phi(t),\check{ P^+_M }J_K \nabla \phi(t)\rp\lesssim   \mathcal{O}(1/M^2)  \|(A+i)^{-1}M\hat P^+_{M/2}J_K \mathcal{O}(1)_{\phi}D_r\phi\|_{L^2}^2.\label{PMM}
\ee

We have that:
\begin{align}
&\lp J_K\phi(t_m-L),P_M^+ J_K  \phi(t_m-L)\rp+\int_{t_m-L}^{t_m} \lp \nabla J_K\phi(t),G^2(A)/R J_K \nabla \phi(t)\rp dt\lesssim \\
&\lp J_K\phi(t_m),P_M^+ J_K  \phi(t_m)\rp+2\Re \int_{t_m-L}^{t_m} \lp(D_H J_K)\phi(t)P_M^+ J_K  \phi(t)\rp dt+\\
&2\Re \int_{t_m-L}^{t_m} \lp ( J_K)\phi(t),i[\mathcal{N},P_M^+]  J_K \phi(t)\rp dt.
\end{align}
The first term on the RHS of this last equation is bounded by $M^{-2   }K^{-2}$, since at this critical point the Virial is small, as we saw before.

The second term on the RHS is bounded by $Lt_m^{-1/2-\alpha}K^{-1}$ using the bound of Theorem \ref{D_H JK}.
Since on the support of $D_H J_K,$  $|x|\sim t^{\alpha},$ we see that the second term is bounded by $Lt_m^{-1/2} M^{-1}.$

The third term on the RHS is bounded by $LM^{-2}$, by dividing $P_M$ by $A^2$ and controlling the $A^2$ by the nonlinearity and the $H^1$ norm of the solution.

Further iteration (using the fact that the operator $B_3$ is bounded by $K^{-1}$ on a sequence of times) allows us to improve the bound on the second term by a factor of $1/K.$
This implies that all three terms are bounded by $M^{-2}.$ This is only possible if the expectation of $r^2$ is bounded, since $K>1.$
But $t_m$ is a maximum point of the expectation of $r^2P_M^+$ on the localized state.

We therefore conclude that in fact (the expectation of ) $r^2$ is uniformly bounded on the weakly localized state, when projected on frequencies greater than 1.
\medskip

Next, we improve the above bounds.
There are residual terms, living on the support of $G'$ with factor $A^2/R^2 \sim M^{2(1-a)}$, on the support of $G'.$
The main negative term is coming from
$$
-(A^2/R^2 )\sinh {\frac{A-M}{R}}/\cosh^2{\frac{A-M}{R}}.
$$
Therefore, we conclude that the main negative part is when $A-M\geq 0; A-M\leq kR, k>>1.$
Therefore if we add the estimate from $$
-(A^2/R^2 )\sinh {\frac{A-M-kR}{kR}}/\cosh^2{\frac{A-M-kR}{kR}}
$$
{\bf times 2},  we cover the negative part from the previous step, and get a new negative part supported in $k^2 R>A-M-kR\geq 0,$ and with size 2.
So we add the next one to cover that.
We get a sum of the form ($M_n= M+k^nR-R$)
$$
\sum_{n=0}^N (1+n)M_n^2(-A^2/R^2) \sinh {\frac{A-M_n}{k^nR}}/\cosh^2{\frac{A-M_n}{k^nR}}\geq cA^2 F(A\geq M).
$$
Applying this to \ref{PMM}, we get:

\be
\lp \nabla J_K\phi(t),A^2 P^+_{M_0} J_K \nabla \phi(t)\rp\lesssim \|\ln A \hat P^+_{M/2}J_K \mathcal{O}(1)_{\phi}D_r\phi\|_{L^2}^2 +C.
\ee
One can remove the $\ln A$ term from the RHS at the expense of a constant $c\ln M.$
The estimate for the domain $A\leq -M_0$ is similar.
Similar arguments work for the other times, when the derivative of the LHS is less than or equal to zero.
So, in fact, we get the stronger bound that ($K>1$)
\be
\lp J_K \nabla \phi(t),A^2 J_K \nabla \phi(t)\rp \lesssim 1.
\ee
\medskip

It remains to control the contribution of the commutator of the Laplacian with $J_K.$
\begin{thm}\label{D_H JK}
If $\phi(t)$ is a weakly localized solution (in $H^1$), in the sense that $\lp \phi(t),|x|\phi(t)\rp \lesssim \sqrt t,$
then, for $1/2\geq \alpha >1/3 $:
\be
t^{1/2} \|F_1(\frac{|x|}{t^{\alpha}}\geq 1)\gamma F_2(|p|\geq K)\phi(t)\|^2\lesssim 1.
\ee
as $t\rightarrow \infty.$
\end{thm}
\begin{proof}
The proof follows by the consecutive use of the following PROBs:
Here, the vector field $g$ is chosen as the smoothed Morawetz multiplier, with optimal slow decay rate at infinity.
\begin{align}
&2 B_1= F_1(\frac{|x|}{t^{\alpha}}\geq 1)\gamma_g +\gamma_g F_1(\frac{|x|}{t^{\alpha}}\geq 1)\\
&2 B_2=\frac{|x|}{t^{1/2}}F_1(\frac{|x|}{t^{\alpha}}\geq 1)F_2(\gamma) +F_2(\gamma) \frac{|x|}{t^{\alpha}} F_1(\frac{|x|}{t^{\alpha}}\geq 1)\\
&2B_3= t^{1/2} F_2(\gamma)(F_1\gamma_g +\gamma_g F_1)F_2(\gamma)\\
&F_2(\gamma)=F_2(\gamma\geq K).\\
&(\gamma-\gamma_g)F_1(\frac{|x|}{t^{\alpha}}\geq 1) \thicksim \la x\ra^{-1}F_1\gamma.
&0<\alpha\leq 1/2.
\end{align}
The first PRES we get is by computing the derivative w.r.t. time of the expectation of $B_1,$ and estimating the integral over time of the positive terms by $B_1$
and by the integral of the symmetrization term, the negative term (coming from the derivative w.r.t. of $B_1$ ) and the interaction term:
\begin{align}
&\lp \phi(T), B_1(T)\phi(T)\rp-\lp \phi(1), B_1(1)\phi(1)\rp \geqq \int_1^T  ds \lp \gamma \phi(s),[s^{-\alpha}F_1'+\la x\ra^{-1-\epsilon}F_1] \gamma \phi(s)\rp+\\
&\int_1^T\lp \phi(s), \tilde F_1'\phi(s)\rp s^{-3\alpha} ds- c\Re\int_1^T\lp \gamma_g \phi(s),  F_1'\phi(s)\rp s^{-1} ds+\\
&2\Re\int_1^T\lp \gamma_g \phi(s), \mathcal{O}(\la x\ra^{-3-\epsilon})( F_1\phi(s)\rp  ds\equiv\\
&C_{1,1}+S_1+R_1+INT_1.
\end{align}

The negative term $R_1$ is controlled by bootstrap. This is done by estimating it in terms of a fraction of the $C_{1,1}$ term:
\be
|R_1|\leq c\int_1^T s^{-2+\alpha}\|F_1'\phi(s)\|_2^2 ds +\int_1^T  s^{-\alpha}  \|F_1'\gamma\phi(s)\|_2^2ds\leq c+cC_{1,1},
\ee
where we used Cauchy-Schwarz Inequality in time, in the last step.

Since we chose $\alpha>1/3$ it follows that $S_1\lesssim \int_1^Tt^{-3\alpha} dt$ and $ INT_1\lesssim \int_1^T \la x\ra^{-3-\epsilon}F_1 dt$
are both bounded by a constant that depends only on the $H^1$ norm of the solution. Therefore, we get the propagation estimate

$$
C_{1,1} \leq c(H^1).
$$

The PROB resulting from $B_2$ is performed in a similar way.
In this case, the LHS is uniformly bounded in time for any solution that is weakly localized.
We then get:
\begin{align}
 &2\Re\int_1^T\lp F_2\phi(s),  \gamma(F_1+F_1')F_2\phi(s)\rp \frac{ds}{\sqrt s}\lesssim C + S_2+R_2+INT_2\lesssim\\
 &c(H^1)+\int_1^T[ s^{-2\alpha-1/2}]ds\leq c(H^1).
 \end{align}
 The bound above on the symmetrization term $S_2$ and the interaction terms $INT_2$, follow since now $s^{-\alpha}$ is replaced by $s^{-1/2}.$
 We therefore have the PROB
 \be
 \Re \int_1^T\lp F_2\phi(s),  \gamma(F_1+F_1')F_2\phi(s)\rp \frac{ds}{\sqrt s}\lesssim C,
 \ee
 which implies the existence of a sequence of times $t_n$ going to infinity, such that
 the expectation of $B_3$ on this sequence of times is uniformly bounded.

 So we can use it for another PROB, $B_3.$
 
 This estimate implies:
 \begin{align}
&\lp \phi(T), B_3(T)\phi(T)\rp-\lp \phi(1), B_3(1)\phi(1)\rp \gtrless 2\Re \int_1^T  ds \lp \gamma_g F_2\phi(s),s^{-1/2}F_1F_2\phi(s)\rp+\\
&\int_1^T  \lp \gamma_g F_2\phi(s),F_1' \gamma_g F_2 \phi(s)\rp s^{1/2-\alpha} ds +\\
&c\int_1^T\lp \nabla F_2\phi(s),  \la x\ra^{-1-\epsilon}F_1\nabla F_2\phi(s)\rp s^{1/2} ds +\\
& \int_1^T\lp F_2\phi(s), \tilde F_1' F_2 \phi(s)\rp s^{1/2-3\alpha} ds+\\
&2\Re\int_1^T\lp \gamma_g F_2 \phi(s), F_1F_2\mathcal{O}(\la x\ra^{-3-\epsilon})(\phi(s)\rp s^{+1/2} ds\equiv\\
&C_{3,1}+C_{3,2}+C_{3,3}+S_3+INT_3.
\end{align}

 Since $1/2-3\alpha <-1/2$, the term $S_3$ is bounded by $C_{1,1}.$
 Finally we have:
 \begin{align}
 &INT_3\leq   \delta  \int_1^T \|\tilde F_1 \la x\ra^{-1/2-\epsilon/2}\gamma_g F_2\phi(s)\|_2^2 \, s^{1/2} ds+\\
 &\delta^{-1}\int_1^T \|\tilde F_1 \la x\ra^{-5/2-\epsilon/2}\phi(s)\|_2^2 \,  s^{1/2} ds \leq (1/2) C_{3,3}+c\int_1^T t^{1/2} t^{-5\alpha-\epsilon} dt <(1/2) C_{3,3}+c.
 \end{align}
 It follows that the expectation of $B_3$ is uniformly bounded, which implies the result but with $F_2(\gamma).$
 To show that it follows for $F_2(|p|)$, we need to use the spherical symmetry and Localization Lemma \cite{SSAnnals}.
 It shows that 
 $$
 F_1(\frac{|x|}{t^{\alpha}}\geq 1)F_2(|p|\geq K)F_2(|\gamma|\leq K/10)F_1(\frac{|x|}{t^{\alpha}}\geq 1)F_2(|p|\geq K)=\mathcal{O}(1/<x>).
  $$
 It also follows that on the subspace where $|x|\leq t^{1/2}$ and frequency larger than 1, the expectation of $|A|$ is bounded.
 \end{proof}

\begin{rem}
The proof above can now be iterated to get a bound on $A^n$ for all $n.$
To see why, we note that the estimates of the correction terms used $\phi \in H^1$ only.
But now we know that (uniformly in time) also $A F_2(|p|\geq 1 )\phi \in L^2, $
So we gain an extra factor of $r$ or $M$ for each $\phi.$
But in the proofs of regularity, we did not use those extra estimates.
\end{rem}

\end{proof}

\section{ Microlocalization  of Weakly Bound States}

We have shown that any state of the system under consideration splits into two channels: a free wave and a weakly localized part.
 The difference between the exact solution and these two parts goes to zero in $H^1$ as time goes to infinity.
 One can ask what the asymptotic behavior is of any part of the solution when microlocalized. In particular, for each domain in phase-space, with a smooth approximate projection $J$, we can consider the existence of the corresponding {\bf channel wave operator}, given by
 \begin{equation}\label{Channel-J}
 \Omega^*_J\phi(0)\equiv \lim_{t \to \infty} e^{-i\Delta t}J\phi(t).
 \end{equation}
 Typically, the way we microlocalize is by using operators (products and sums of) of the form
 $$
 F_{1,a}F_2(t^{b}\gamma),\quad F_{1,a'}F_A(t^{-b'}A).
 $$

 If the above wave operator exists, then we expect $\pwls(t)$ and $(I-J)\pwls(t)$ have the same asymptotic form.
 In fact, this is not hard to prove.
 \begin{prop}
 Let $\pwls(t)$ be a localized part of the solution of NLS, \ref{Main-eq}.
 Then, for $J$ as above we have,
 $$
 \|J\pwls(t)\|_{L^2}\rightarrow 0.
 $$
 \end{prop}
\begin{proof}
 Let $1/2<\alpha<1$, and $F_{1,\alpha}(x,t)\equiv F_{1,\alpha}(\frac{|x|}{t^{\alpha}}\geq 1)$
 be the usual smooth characteristic function of the domain $\frac{|x|}{t^{\alpha}}\geq 10.$
 Then, Since we know that $\pwls$ in localized in the region $\frac{|x|}{t^{1/2+0}}$,
 we have that
\begin{align}
& F_{1,\alpha}(\frac{|x|}{t^{\alpha}}\geq 1)J\pwls=JF_{1,\alpha}\pwls+ [F_{1,\alpha},J]\pwls,\\
&[F_{1,\alpha},J]=[F_{1,\alpha},F_{1,a}F_2(t^{b}\gamma)+F_{1,a'}F_A(t^{-b'}A)]=\mathbb{O}(t^{b-\alpha})+\mathbb{O}(t^{-b'})\\
&JF_{1,\alpha}\pwls=\mathbb{O}(t^{-\alpha+1/2+0}).
\end{align}
Therefore, for $ \alpha >b$ we have that
$$
\|F_{1,\alpha}J\pwls(t)\|_{L^2}\rightarrow 0.
$$
Finally,in the sense of $L^2$ convergence:
\begin{align}
&(I-F_{1,\alpha})e^{i\Delta t}e^{-i\Delta t}J\pwls(t)\rightarrow (I-F_{1,\alpha})e^{i\Delta t}\Omega^*_J\phi(0)\\
&(I-F_{1,\alpha})e^{i\Delta t}\Omega^*_J\phi(0)\rightarrow 0,
\end{align}
where the last limit follows from minimal velocity bounds for the free Schr\"odinger equation.
\end{proof}

The main result we prove now is that in a large part of the phase space, namely the part where
$|x|\geq t^{\alpha} \bigcap t^{\beta}|\gamma|\geq 1,\alpha>\beta>0,$ the solution to equation (\ref{Main-eq}) satisfying the global energy bound \ref{global-bound},
propagates to a free wave.
That is, the free channel wave operator $\Omega^*_J$ exists for $J$ of the form
$$
 F_{1,a}F_2(t^{b}\gamma),\quad F_{1,a'}F_A(t^{-b'}A), \quad, a>b>0, \quad a'>0, b'>0.
$$

 \begin{thm}~\label{PA-J} Let  $\phi$ be solution to equation (\ref{Main-eq}) satisfying global energy bound (\ref{global-bound}).   Let $F_1=F_1(\XT), F_4=F_4(\gamma t^\beta >c_0)$ .
 Assume moreover, that every factor $V(x,t)$ in the interaction satisfies $(r\partial_r)^nV(x,t)\lesssim 1$ for all $n.$
  Then   the following estimates hold:

A)  For all $\alpha>\beta>0:$
 \begin{align}
\int_{t_0}^\infty \frac{1}{t^\alpha} \la \sqrt{F_1'F_1}  \gamma F_4 \sqrt{F_1'F_1} \ra +\frac{1}{t} \la F_1 F_4'F_1  \ra dt<\infty,  \label{PE-JVG}.
\end{align}

B) For all $\alpha>0\, \,  \text{and}\, \epsilon>0,$ we have:
\begin{align}
\int_{t_0}^\infty \frac{1}{t^\alpha} \la \sqrt{F_1'F_1}  \gamma F_A \sqrt{F_1'F_1} \ra +\frac{1}{t} \la F_1 F_A'F_1  \ra dt<\infty,  \label{PE-JVA}.
\end{align}

  \end{thm}
 \begin{proof}
 Notice that this was proved before, in section 5, for $\alpha>1/3.$ We focus here on the cases where $\alpha\leq 1/3$ and $\beta $ close to $\alpha.$ $\delta\equiv \alpha-\beta.$
 In case $\alpha>1/3,$ the interaction term is integrable in time, due to the assumption that it decays faster than $|x|^{-3}$ at infinity.
 So, the proof will be the same, if we can find a way to bootstrap the symmetrization term and the interaction terms as well.
 We begin by showing an iteration scheme to control the symmetrization term that will get us the desired estimate.

 Let the PROB be    $t^{-\eta}A_4= t^{-\eta}F_1(\XT)F_4(\gamma t^\beta >c_0)F_1(\XT).$  We will choose  $c_0>0$  large enough for the iteration procedure to apply.
 We will choose $\eta_n$ successively smaller, until it is zero.
 Furthermore, after each derivation of a PRES from the above PROB, we will derive a companion PRES using:
 $$
  t^{-\eta}A_{4,l}= t^{-\eta}\ln^{-1}{\la x\ra}F_1(\XT)F_4(\gamma t^\beta >c_0)F_1(\XT)\ln^{-1}{\la x\ra}.\label{A4-log}
 $$
 The resulting positive term, which is controlled by the PRES is then:
 $$
  \int_1^T t^{-\eta-\alpha}\lp \sqrt {F_1'(\XT)F_1(\XT)}\gamma F_4(\gamma t^\beta >c_0)\sqrt {F_1'(\XT)F_1(\XT)}\rp dt \lesssim C+Sym+ INT,
 $$
 where $Sym,INT$ stand for the symmetrization and Interaction terms  and the constant comes from the LHS (initial and final times of the PROB).

 The Symmetrization term is of the form:
 \begin{align}
 &\lp G_1(\XT)\phi(t), [G_1(\XT),[G_1(\XT), \gamma F_4(\gamma t^\beta >c_0)]] G_1(\XT)\phi(t)\rp t^{-\alpha-\eta}=\\
 &\lp G_1(\XT)\phi(t), [t^{-2\alpha+\beta}\tilde F_4 +\mathcal{O}(1)t^{-N\delta+ \alpha}]G_1(\XT)\phi(t)\rp t^{-\alpha-\eta}=\\
 & t^{-\eta-3\alpha+\beta}\lp G_1(\XT)\phi(t),  \tilde  F_4(\gamma t^\beta >c_0) G_1(\XT)\phi(t)\rp + \mathcal{O}(L^1(dt)).
 \end{align}
 Here $G_1(\XT)$ stands for a generic bump function around $1$  of $\frac{|x|}{t^{\alpha}}.$  It is different at each place, not just from line to line!
 The second equality above follows from direct application of the commutator expansion lemma up to order $N$ to the expression
 $$
 [G_1(\XT), \gamma F_4(\gamma t^\beta >c_0)],
 $$
and noting that in each order of the expansion we gain a factor of  $t^{-\delta},$  and that all terms in the expansion are of the same form: $G_1 \tilde F_4$ .
Here $\tilde F_4$ is a generic bump function of $\gamma t^\beta $ supported around $c_0.$

Since the positive leading term of the PRES satisfies:
\begin{align}
 & \mathcal{O}(L^1(dt)) \geq t^{-\eta-\alpha}\lp \sqrt {F_1'(\XT)F_1(\XT)}\gamma F_4(\gamma t^\beta >c_0)\sqrt {F_1'(\XT)F_1(\XT)}\rp\geq \\
&t^{-\eta-\alpha-\beta}\lp G_1(\XT)\tilde F_4 G_1(\XT)\rp \label{PRES-eta}
\end{align}
provided  $Sym+ INT$ are also in $L^1,$ we conclude that the resulting estimate controls a generic Sym. term.
So, to ensure that $Sym+ INT$ are also in $L^1,$, we choose
$$
\eta\equiv \eta_1=1-  3\alpha+\beta+0=1-2\alpha-\delta+0.
$$
$0$ stands for a fixed small number, smaller than all parameters.
With this choice of $\eta$ the PRES \ref{PRES-eta} has a weight $t^{-1+2\delta-0}.$
Therefore, since $\delta>0,$ the resulting estimate can be used to control the  $Sym$  term with
$$
\eta_2+3\alpha -\beta=1-2\delta+0.
$$
Therefore the new PROB will need a weight $\eta_2$ that is less than $\eta_1$ by an amount $2\delta-0.$
Therefore, after finitely many iterations, depending only on the size of $\delta$ we get the desired PRES with $\eta=0.$

Next, we need to show that we can also iterate the interaction term.
We would like to proceed in a way similar to the above, but the problem is that commuting $F_4$ through the interaction term is not possible. Even if the interaction term is smooth, we still need to gain a factor of $t^{-\alpha}$ (or equivalently a factor of $x^{-1}$). This means that the interaction term should be smooth and such that
$(r\partial_r)^n \mathcal{N}$ is bounded for all $n.$

We will instead introduce a different approach, which is a new \emph{iterative high-low  phase-space estimates} .
By this we mean that the high-low decomposition will be w.r.t.  the spectrum of pseudo-differential operators, rather than the derivative operator.

   Another complication is that the interaction term has $F_1$ in it, and not $G_1.$

   We therefore need to upgrade the basic PRES. This is done using the companion estimate with the extra $log$ factor.
   The extra $log$ factor gives another estimate, in which the leading PRES is a bound on
   $$
   \int_1^T \| \ln^{-3/2}{\la x\ra}F_1(\XT)\la x\ra^{-1/2}\sqrt{\gamma F_4} \phi(t)\|^2 t^{-\eta} dt \lesssim C+Sym+INT.\label{PRES-ln}
   $$

   With this bound, we can control the interaction term, since it is decaying faster than this term.
   However, it is necessary to bring a factor of $F_4$ to the right place.

  The nonlinear term is :
 $$
  \lp \phi,i[\mathcal{N}(\phi),F_1F_4F_1]\phi\rp=2\Re  \lp \sqrt(F_4)F_1  \mathcal{N}(\phi)\phi,\sqrt(F_4)F_1\phi\rp,\label{INT-beta}
 $$
 $F_4=F_{4,\beta}(t^{\beta}\gamma\geq 1).$

 Therefore, we see that the factor $\sqrt(F_4)F_1\phi$ can be controlled by the following leading term of the PROB, that is, after using Cauchy-Schwarz

 $$t^{-\eta}\|\sqrt(F_4)F_1\ln^{-3/2}(\jx)\phi\|^2t^{-\alpha-\beta},$$

 and using that $$\|F_1 \mathcal{N}(\phi)\|_{L^{\infty}} \leq ct^{-(3+0)\alpha}, \quad \beta < \alpha.$$
 Next we use a high-low vector-field decomposition to deal with the term
 $$
 \sqrt F_4 F_1 \mathcal{N}(\phi)\phi=  \sqrt F_4 F_1 \mathcal{N}(\phi)\tilde F_1 F_{4,\beta'}\phi+ \sqrt F_4 F_1 \mathcal{N}(\phi)\tilde F_1 \bar F_{4,\beta'}\phi,
 $$
 $$
 \beta< \beta '< \alpha.
 $$
 Consider the generic term in the interaction and decompose it as :
 \begin{align}
 &\sum\phi^{~}\ldots\phi^{~}V(x)= \sum(F_{4,\beta'}\phi+\bar F_{4,\beta'}\phi)^k (F_{4,\beta'}V(x)+\bar F_{4,\beta'}V(x))=\\
 &\sum(\bar F_{4,\beta'}\phi^{~})^{k+1}(\bar F_{4,\beta'}V(x)+\sum(\bar F_{4,\beta'}\phi^{~})^{k}V(x) F_{4,\beta'}\phi^{~}+\\
 &\sum(\bar F_{4,\beta'}\phi^{~})^{k+1} F_{4,\beta'}V(x).
\end{align}
$\phi^{~}$ stands for $\phi$ or its complex conjugate.
$\sum$ stands for the sum of all terms of this type.
All terms which contain a factor with high $\gamma$ acting on $\phi$, are controlled by the leading term in an iteration, using PROBs with $\beta'.$
Since the interaction term decays in time faster than the leading term, this extra term is controlled by a PRES with $\eta$ {\bf larger} than the previous step.
Therefore after finitely many such iterations $\eta$ becomes larger than 1, and then there is nothing to prove.

The term of the form low-low is annihilated by the action of the operator $\sqrt F_4 F_1.$
This is due to a phase-space estimate of the type (see the next subsection)
$$
F_4(\gamma>K)[\bar F_4(\gamma <K/4)\phi]^2 \lesssim K^{-l}.
$$
Similar estimate holds with $A$ replacing $\gamma.$
 Therefore, we need to control the interaction terms in which at least one factor of the form
 $$
  F_{4,\beta '}\phi^{~}
 $$
 is present in the factor $\sqrt(F_4)F_1  \mathcal{N}(\phi)\phi$ of equation \ref{INT-beta}.

 Therefore, by applying Cauchy-Schwarz to the scalar product, one term is controlled by bootstrap against the leading positive term in the PRES, and the other term, which being of the same form but with $ F_{4,\beta'}$ instead of $ F_{4,\beta},$ is controlled by the use of a similar PROB with $\beta' .$
 This can be iterated {\bf only} finally many times, so we can only get estimates with $\beta <\alpha.$

It remains to show that
$$
F_1\sqrt F_4 V(x)( \bar F_{4,\beta'}\phi^{~})^{k+1} =\mathcal{O}(t^{-m}), \quad m>>1.
$$

Using the commutator expansion lemma, we have:
\begin{align}
&F_1\sqrt F_4 V(x) (\bar F_{4,\beta'}\phi^{~})^{k+1}=\\
&F_1 V(x)\sqrt F_4 ( \bar F_{4,\beta'}\phi^{~})^{k+1}+F_1[\sqrt F_4 ,V(x) ] \bar F_{4,\beta'}\phi^{~})^{k+1}=\\
&\sum_n F_1\tilde V(x)^{(n)}( \tilde F_4^{(n)}  \bar F_{4,\beta'}\phi^{~})^{k+1}+\mathcal{O}(t^{-m}+R_{n+1}(F_1\la x\ra^{-n-1}) |\la x\ra^{n+1}V^{(n+1)}|=\\
&\mathcal{O}(t^{-m}).
\end{align}

 The contribution of the interaction term to the derivative of (the expectation value of) the PROB is of the form

 \begin{align}\label{NLD-1}
 &\lp \phi,i[\mathcal{N}(\phi),A_4]\phi\rp  =\lp \phi, A_4\mathcal{N}(\phi)(A_4^-+\bar A_4+\tilde A_4)\phi\rp =\\
 &\lp \phi, A_4\mathcal{N}(\phi) A_4^-\phi\rp +\lp \phi, A_4\mathcal{N}(\phi)\tilde A_4\phi\rp +\\
 &\lp \phi, A_4\mathcal{N}(\phi)\bar A_4 \phi\rp +\\
 &C.C.+R(t).
 \end{align}
 Here we use the decomposition
 \begin{align}\label{NLD-2}
& A_4 \mathcal{N}=F_1F_4\mathcal{N}F_1[F_{4,\beta'}+\bar F_{4,\beta'}+F_4(\gamma t^\beta<-c_0)] \\
& A_4^-=F_1F_4(\gamma t^\beta<-c_0)\\
& R(t)=\mathcal{O}(t^{-\delta})+[\text{similar terms}].
 \end{align}
 $C.C.$ stands for the complex conjugate of the expressions on the RHS.
 $\beta'<\beta.$
 The terms in equation \ref{NLD-1} without $\bar F_4$ are controlled by the leading term of an appropriate PRES using the following PROBs: First one uses $A_4, A_4^-.$ Then one uses $\ln (\jx)^{-2}A_4, \ln (\jx)^{-2}A_4^-.$

 This gives estimates that control the above terms without $\bar A.$

 To estimate the terms with the bar, we use a high-low  spectral decomposition ( w.r.t. $\gamma$ and $A.$)

 The PROBs  $\ln^{-2} (\jx) A_4, \ln^{-2} (\jx) A_4A_4^-$ have the following Heisenberg derivative:
 \begin{equation}
 D_H \left(\ln^{-2} (\jx)A_4+A_4 \ln^{-2} (\jx)\right)=-4\ln (\jx)^{-3} (\jx)^{-1}\gamma A_4 + \ln^{-2} (\jx)D_H(A_4) +C.C.+R(t).
 \end{equation}
The corresponding PRES that follows gives a bound on $\int_{1}^{T}\|\sqrt(F_4)\ln^{-3/2}  (\jx)(\jx)^{-1/2}F_1\phi\|^2 $
in terms of the bound on the same expression without the log factor.

Although this extra estimate may look weaker, it holds on the full support of the function $F_1$ and not just on the support of the derivative of $F_1.$ This is needed to control the non-linear terms.

 We recall here the estimate for the Heisenberg derivative of $A_4$ with respect to the free flow:

     \begin{align}
  D_HA_4=& (D_HF_1)F_4F_1 + F_1F_4(D_HF_1)  + F_1(D_HF_4)F_1 \notag\\
=&F_1' \frac{1}{t^\alpha}\left[2\gamma   -\alpha \frac{\la x\ra }{t}\right]F_4F_1  +F_1F_4\frac{1}{t^\alpha}\left[2\gamma   -\alpha \frac{\la x\ra }{t}\right] F_1'+   F_1([-i\Delta, F_4] +\frac{d F_4}{dt})F_1  \notag\\ &
+ \frac{1}{2t^{2\alpha}} \left[F_1''[[-i\Delta, \jx], \jx] , F_4\right]F_1 +    \frac{1}{2t^{2\alpha}}[F_4,F_1]  F_1''[[-i\Delta, \jx], \jx]
  \end{align}
   We organize the terms into three groups $D_HA_4=I_1+I_2+I_3$, such that
 \begin{align}
 I_1= &F_1' \frac{1}{t^\alpha}\left[2\gamma   -\alpha \frac{\la x\ra }{t}\right]F_4F_1  + F_1F_4\frac{1}{t^\alpha}\left[2\gamma   -\alpha \frac{\la x\ra }{t}\right] F_1' \\
=& \frac{4}{t^\alpha} \sqrt{F_1'F_1} \gamma F_4 \sqrt{F_1'F_1}   -\frac{\alpha}{t}  [F_1'  \frac{\jx }{t^\alpha}F_4F_1  + F_1F_4 \frac{\jx}{t^{\alpha}}F_1'] + R \\
I_2=& F_1(D_HF_4)F_1   =  \frac{\beta}{c_0 t} F_1 \gamma t^{\beta} F_4'F_1, \hspace{1cm} \text{for } t^\alpha\geq 4,\\
I_3=& \frac{1}{2t^{2\alpha}} \left[F_1''[[-i\Delta, \jx], \jx] , F_4\right]F_1 +    \frac{1}{2t^{2\alpha}}[F_4,F_1]  F_1''[[-i\Delta, \jx], \jx] =O(t^{-3\alpha+\beta}).
 \end{align}
Here $R$ is the remainder terms coming from symmetrization, i.e.
\begin{align} & \frac{1}{t^\alpha}\sqrt{G_1}[[\sqrt{G_2}, \gamma F_4], \sqrt{G_3}]\sqrt{G_4} \\
=&   \frac{1}{t^\alpha}\left\{\sqrt{G_1}[\sqrt{G_2}, \gamma ][F_4, \sqrt{G_3}]\sqrt{G_4}  +   \sqrt{G_1}[\gamma , \sqrt{G_3}] [\sqrt{G_2}, F_4]\sqrt{G_4}   +  \sqrt{G_1}\gamma [[\sqrt{G_2}, F_4], \sqrt{G_3}]\sqrt{G_4} \right\} \notag\end{align}
with $G_1, G_2, G_3, G_4\in \{F_1, F_1'\}$, hence $R=O_1(t^{-3\alpha +\beta})$.

 So we have
\begin{equation}D_H A_4 =  \frac{4}{t^\alpha} \sqrt{F_1'F_1} \gamma F_4 \sqrt{F_1'F_1}  +\frac{\beta}{c_0t} F_1 \gamma t^{\beta}F_4'F_1   -\frac{\alpha}{t}  [F_1'  \frac{\jx }{t^\alpha}F_4F_1  + F_1F_4 \frac{\jx}{t^{\alpha}}F_1']  + I_3+R.
\end{equation}
The first two terms are positive, it left to show  $I_3, R\in L^1(dt)$ and also control the third term.
The proof of the above estimates with $F_A(A/t^{\epsilon})$ replacing $F_4$ is similar.
In this case the commutator $[F_A,F_1]=\tilde F_A\tilde F_1 t^{-\epsilon}+\mathcal{O}(t^{-n\epsilon}), \forall \epsilon>0.$
$\tilde F_A, \quad\tilde F_1$ are bounded functions with support given by the support of the derivatives of the functions $\tilde F_A, \, \tilde F_1$ respectively.
\end{proof}

\medskip

Next, we specialize the above analysis to the case of WLS:

\begin{thm}[Slow Expansion and Regularity of WLS]
The high frequency part of WLS is localized in an arbitrarily slowly growing domain in space, in the sense that
\begin{align}
&\lp\pwls, F_1(\frac{|x|}{\ln t}\geq 1)F_2(\gamma \geq a)\pwls\rp_{t_n}\leq \ln^{-k} t_n\\
& \forall a>0, k=1,2,\ldots
\end{align}

Similar estimate holds for $F_2(\gamma\leq -a).$
\end{thm}
\begin{proof}
We compute as before the time derivative of the following PROB:
\begin{align}
&\partial_t \left (\lp\pwls,  \frac{ |x|}{\sqrt t}F_1(\frac{|x|}{\ln t})F_2(\gamma \geq a)F_1\pwls\rp+ C.C.\right )=\\
&\lp\pwls, \left [F_1\gamma F_2F_1 +\sqrt{F_1F'_1}\gamma F_2\sqrt{F_1F'_1}\right ]\pwls\rp/\sqrt{t}+\\
&\lp\pwls, \tilde F_1F_2 \pwls\rp \frac{1}{\sqrt t \ln ^2 t}+\\
&\lp\pwls, F_1F_2\tilde N(\phi) \pwls\rp \frac{1}{\sqrt t ln^m t}+\\
&\frac{-1}{2t}\lp\pwls,\frac{ |x|}{\sqrt t} F_1F_2\pwls\rp.
\end{align}
The first line above on the RHS, consists of leading-order positive terms.
The second line comes from Symmetrization. It is of lower order than the positive leading order. So, it can be estimated by repeated iteration that produces an arbitrary power of $\ln t.$
The third line comes from the interaction terms. On the support of $F_1$ and the assumption of decay for large $x,$ the decay in time is faster than $\sqrt t$ by a few $\ln t$ powers, times $F_1F_2.$ Therefore, this part is dominated by the positive leading term.
Finally, the last term is controlled by the basic property of WLS, namely that $\lp\pwls,\frac{ |x|}{\sqrt t} F_1F_2\pwls\rp\leq \mathcal{O}(1).$ Therefore, this last has an upper bound $1/t.$

The first iterate we multiply the above PROB by $t^{-1/2},$ and we get that the symmetrization term is integrable (like $t^{-1}/\ln^2 t.$) Same for the interaction term, that decays faster.
The last term decays like $t^{-3/2},$ and is also integrable.

Hence we get the PRES that the leading positive term above is integrable with a weight $t^{-1/2}.$
We then redo the above calculation with the above PROB multiplied by $ t^{-1/2}\ln^2 t.$ Then, the $Sym$ term and the interaction term have a factor $t^{-1}$ which is integrable by the previous PRES.
So, now the new PRES has weight  $t^{-1/2}\ln^2 t .$  We can therefore iterate and gain a factor of $\ln^2 t$ in each iterate.
Using then that
\begin{equation}
\int_1^T \frac{1}{\sqrt t ln^k t} dt \leq T^{1/2}/(\ln ^k T),
\end{equation}
the result follows.
\end{proof}
\begin{rem}
The above proof extends to slower growth functions than $\ln t.$
The proof above does NOT imply that the wave operators exist, or that the asymptotic state has no slow growth and no high frequency in the usual sense of limit in $L^2$, but rather, it is true {\bf on a sequence of times.}
It does follow from the statement of Generic Asymptotic Completeness that the limiting bound state has this property.
\end{rem}

\subsection{High Low Vector-Field Decompositions}

It remains to verify the high-low properties of the projections of the operators $\gamma, \,  A.$
It is possible to compute the action of an operator defined as a function of a self-adjoint operator like $\gamma, A,$ since they generate vector-fields.
It can be seen that, in general, unlike the case of Fourier transform (related to the generator of translations $-i\partial_x,$) there are non-vanishing corrections. That is
$$
F_2(t^\beta\gamma\geq 1)\phi_L\psi_L \ne 0,
$$
for $\phi_L,\psi_L$ of the form $\bar F_2(t^{\beta'}\gamma\leq 1)\phi_L (\text{and} \, \psi_L), \beta'> \beta.$
Yet, the corrections are small:
\begin{lem}
For $F_1(|x|/t^{\alpha}\geq 1), \alpha> \beta'>\beta,$ and  $\phi_L,\psi_L$ as above,
we have
$$
F_{1,\alpha}F_{2,\beta}\phi_L\psi_L= \mathcal{O}(t^{-m}).
$$
\end{lem}
\begin{proof}
Since $\gamma^n\phi_L=t^{-n\beta'}\tilde \phi_L,$
by the chain rule, on the support of the function $F_1$
we have that $$
\gamma^n (\phi_L\psi_L)=c_n t^{-n\beta'}\tilde \xi_L,
$$
and so
\begin{align}
&F_1F_{2,\beta}\gamma^{-n}\gamma^{n}\tilde F_1(\phi_L\psi_L)=\\
&c_n\mathcal{O}( t^{n\beta}t^{-n\beta'})+\\
&c_n\mathcal{O}(t^{-n''\alpha} t^{n\beta}t^{-n'\beta'}).
\end{align}
$n'+n''=n.$
$c_n$ grows faster than exponentially in $n.$
\end{proof}

A similar argument can be used for functions of $A/t^{\epsilon}$ but we will use a separate argument,
exploiting the Analytic Projectors as cutoff functions.
\subsubsection{High/Low Scaling Decomposition}

Let $A=\frac{1}{2}(x \cdot p+p \cdot x)$ be the self-adjoint generator of dilations of $L^{2}(\mathbb{R}^{\nu}).$

By the spectral theorem we know that $F(A)$ is well defined for any measurable function $F$.

So, now consider two functions which are localized in the low range of the spectral projection of $A$ :

$$
\begin{aligned}
\phi & =F\left(\frac{|A|}{L} \leq 1\right) \phi \\
&\psi =F\left(\frac{|A|}{M} \leq 1\right) \psi
\end{aligned}
$$

We can diagonalize A. One way to see that is to note that we know the characters of the group generated by $A$, the generalized eigenfunctions of $A$.

$$
A\left(|x|^{-\nu / 2}|x|^{+i \lambda}\right)=\lambda\left(|x|^{-\nu / 2}|x|^{+i \lambda}\right)=: \lambda e(x, \lambda) .
$$

so we have for $f \in \mathcal{S}(\mathbb{R}^{\nu} /\{0\})$

$$
\begin{aligned}
\tilde{f} (\lambda) & =\left\langle|x|^{-\nu / 2}|x|^{+i \lambda}, f(|x| \omega \mid)\right\rangle_{L^{2}\left(|x|^{-1} d x\right)} \\
& =\int_{\mathbb{R}^{+}}|x|^{\nu / 2}|x|^{-i \lambda} f(|x| \omega) \frac{d x}{|x|}, \omega \in S^{\nu-1} .
\end{aligned}
$$

For $f$, which is locally in $L^1$  and vanishes at a certain rate at infinity, and asymptotic to a power of $r$ at the origin, $\tilde f$
is analytic in a strip parallel to the imaginary axis, and the interval in the real part is determined by the asymptotic behavior at zero and infinity, in the variable $r$. Then for $k$ in that interval, we have the Inverse Formula:

$$
f(x)=\frac{1}{2\pi i} \int_{k - i \infty}^{k+i \infty}|x|^{-\lambda} \tilde{f}(\lambda) d \lambda .
$$

In practice, we use logarithmic change variables, and then this is mapped on the standard Fourier transform and Plancherel Theorem.

Next we note that
$$
A|x|^a=-ia|x|^a-i\nu/2.
$$

$$
\begin{aligned}
& A\left(\frac{1}{|x|^m} e(x, \lambda)\right)=\left[A, \frac{1}{|x|^m}\right] e(x, \lambda)+\frac{1}{|x|^m} A e(x, \lambda) \\
& =\left[-i|x| \frac{\partial}{\partial|x|}, \frac{1}{|x|^{m}}\right] e(x, \lambda)+\frac{1}{|x|^m} \lambda e(x, \lambda) \\
& =\left\{\begin{array}{l}
-i\frac{m}{|x|^m} e(x, \lambda)+\lambda \frac{1}{|x|^m} e(x, \lambda)=(\lambda-i m) e(x, \lambda-im) . \\
m=0,1,2, \cdots .
\end{array}\right.
\end{aligned}
$$

We conclude that

$$
\begin{aligned}
& F(A) f(x)=F(A) \frac{1}{2 \pi i} \int_{k-i \infty}^{k+i \infty} \tilde{f}(\lambda)|x|^{-\lambda} d \lambda \\
& = \int_{\R} \tilde{f}(\lambda) F(\lambda)|x|^{-\nu / 2+i \lambda} \frac{d|x|}{|x|} .
\end{aligned}
$$

Next we use the Fourier representation of the Mellin Transform.
Let, for $f\in C_0^{\infty}(\R^{\nu}-{0}),$
\be \label{scale}
Ug(t,\omega)\equiv e^{\frac{1}{2}\nu t}g(e^t\omega).
\ee
Then we have that

\begin{equation}
\tilde \phi(a,\omega)=\frac{1}{\sqrt {2\pi}}\int_{\R}e^{-iat}U\phi(t,\omega) dt.
\end{equation}
\begin{prop}
\be
\widetilde{\phi(x)\psi(x)}= \tilde \phi(.)\ast \tilde\psi(.)(\lambda-i\nu/2).\ee
\end{prop}
\begin{proof}
The proof follows from the basic property of Fourier Transform, applied to \eqref{scale}.
\begin{align}
&\widetilde{\phi(x)\psi(x)}=\frac{1}{\sqrt {2\pi}}\int_{\R}e^{-i\lambda t}U(\phi(t,\omega)\psi(t,\omega)) dt=\\
&\frac{1}{\sqrt {2\pi}}\int_{\R}e^{-iat}U(\phi(t,\omega))U(\psi(t,\omega))e^{\frac{-1}{2}\nu t} dt=\\
&\tilde \phi(.)\ast \tilde\psi(.)(\lambda-i\nu/2)
\end{align}
\end{proof}

$$
\begin{aligned}\label{Mel}
& F_{A}(A \geq N)\phi \psi= \int_{|a-b| \leqslant L,|b| \leqslant M }F_A(a) \tilde{\phi}(a-b-i\nu/2)\tilde{\psi}(b)e(x,a)db\\
&  \text { with } e(x,a) \equiv|x|^{-\nu/ 2+ia}. 
\end{aligned}
$$


Furthermore, 
$$
\int \tilde{\phi}(a)\tilde{\psi}(b) \frac{d(a-b)}{2}=\chi(a+b).
$$
Since $|a+b|\leq L+M<< N,$ due to the exponential localization of $F_A,$

$$
F_A=1+\tanh(\frac{A-N}{R}),
$$
we conclude that the Mellin transform of the product $\phi \psi$ is compactly supported on the real axis.

We conclude that the integral \eqref{Mel}  is of order

$$
\mathcal{O} (e^{-\frac{1}{R}(N-L-M)})
$$

For this we need that $F_{A}$ is analytic in a strip around the real axis, which is larger than (say) $10(\frac{\nu}{2}+1).$

\begin{prop}
For $\phi_L, \psi_L $ with the form
$\phi_L=\bar F_A(A/t^{\epsilon}\leq 1)\phi_L$ and similarly for $\psi_L$,
we have that
$$
P_A(A/t^{\eta}\geq 1)(\phi_L \psi_L)=\mathcal{O}(t^{-m(\eta-2\epsilon)}).
$$
More generally, for $\phi_L=\bar F_A(A\leq N)\phi_L$ and similarly for $\psi_L$
$$
P_A(A\geq M)(\phi_L \psi_L)=\mathcal{O}((M-2N)^{-m}).
$$
\end{prop}

\begin{proof}

\begin{align}
&\phi_L \psi_L= \int_{\R^2} (\bar F_A(|a-b-i\nu/2|\leq N))(\tilde\phi_L(a) \bar F_A(|b|\leq N))\tilde\psi_L(b) e(x,a) \, da db.\\
&P_A(|A|\geq M)(\phi_L \psi_L)=\\
& \int_{\R^2}P_A(|a|\geq M) (\bar F_A(|a-b-i\nu/2|\leq N))(\tilde\phi_L(a) \bar F_A(|b|\leq N))\tilde\psi_L(b) e(x,a)  \,da db=\\
& \int_{\R}P_A(|a'+b+i\nu/2|\geq M) (\bar F_A(|a'|\leq N))(\tilde\phi_L(a'+b+i\nu/2) \bar F_A(|b|\leq N))\tilde\psi_L(b) e(x,a')  \, da' db.
\end{align}
Here, $a'=:a-b-i\nu/2.$ $ P_A$ defined is as
$$
P_A(A)=:1+\tanh(R^{-1}(A-M)).
$$
Since
 $$|a+b|\leq 2N,\quad \text{it follows that}\,\, P_A(a+b)=\mathcal{O}(e^{-R^{-1}(M-2N)} ).
$$
For $\frac{M-2N}{R}>>1$, the addition of $i/2$ to $a+b$, will not change the bound.
\end{proof}

The proof for the case where $t^{\beta}\gamma\leq 1, \beta < \alpha<1/2,$ follows the same argument.
Therefore, the localized solution can only live in the region below, as time goes to infinity:

\begin{equation}
 \bigcup_{0<\alpha\leq 1/2}\{|x|/t^{\alpha}\sim 1; |\gamma|\leq t^{-\alpha+0}\}\bigcup \{|x|\lll \ln t\}.
\end{equation}

\begin{thm}

Under the conditions of the above last theorem, we now consider the case where $0<\alpha< \beta.$
Let $\tilde \delta\equiv \beta-\alpha.$
We have that
\begin{align}
&PS1\equiv F_1(\frac{|x|}{t^{\alpha-0}}\leq 1)F_2(1\leq t^{\beta}\gamma\leq t^{\tilde\delta})\\
&\lp\phi(t), A (PS1) A \phi(t)\rp \lesssim 1.\label{F_1F_2}
\end{align}
\end{thm}
\begin{proof}

\begin{align}
& \lp\phi(t), A (PS1) A \phi(t)\rp=\lp\phi(t), x\cdot p (PS1) p\cdot x\phi(t)\rp +c\lp\phi(t),  (PS1) A+A(PS1) \phi(t)\rp=\\
&c\lp\phi(t), A (PS1) + (PS1)A\phi(t)\rp+\lp\phi(t), x\cdot  (PS1)(p\cdot p+\epsilon)(p\cdot p+\epsilon)^{-1} p  p\cdot x\tilde F_1 \phi(t)\rp +\\
&ct^{-\tilde\delta-0}\lp\phi(t), A (PS1)\phi(t)\rp=\\
&t^{-\tilde\delta-0}+t^{2\alpha-0}\lp\phi(t),   (PS1)(\gamma^2+\epsilon)\mathcal{O}(1) \tilde F_1\phi(t)\rp\lesssim t^{2\alpha-0}(t^{-2\alpha}+\epsilon)\lesssim 1,\\
&\textbf{since we can choose}   \epsilon=t^{-2\alpha}.
\end{align}

\end{proof}

\section[Section Title. Section Subtitle]{{\bf Epilogue-}
   \\ {\large Further Properties of the Solutions, Comments and Conjectures }}

 We begin with a preliminary conclusion from the main theorem:
 \begin{thm}[norm]

 The following limits exist:
 \begin{align}
 &\lim_{t \to \infty}\|\pwls(t)\|_{H^1}=h_{wls}\\
&\lim_{t \to \infty}\|\pwls(t)\|_{L^2}=l_{wls}\\
& l_{wls}+\| \Omega^*_F(\phi(0))\|_{L^2}  =\| \phi(0)\|_{L^2}.\label{norm}
\end{align}
\end{thm}
\begin{proof}
The proof follows from the fact that the scalar product in $L^2,$ of the free wave and the WLS vanishes as $t$ goes to infinity:
$$
\lim_{t \to \infty}\lp \pwls(t), e^{i\Delta t}\Omega^*_F(\phi(0))\rp=0.
$$
This is due to the fact, that the free flow concentrates at distance $|x|\geq t^a, \forall a<1$, while the WLS is localized at
 $|x|\leq t^{1/2}.$
\end{proof}
\subsection{Almost Periodic Solutions}
\begin{thm}[Almost Periodic Solutions]
If the WLS $\pwls(t)$ is an almost periodic function of $t$, then the solution is localized in space:
 $$\|F_1(|x|\geq R)\pwls(t)\|_{L^2}^2\lesssim 1/R.$$
\end{thm}

\begin{proof}
We prove it for the periodic case. The proof in the general case follows from the basic property of almost periodic functions that for each $\epsilon >0$, there is an almost period of finite value.

In the periodic case, since by assumption the initial data is localized, it is then localized on the series of times $n\omega, n=0,1,2,\ldots$

That is $\|F_1(|x|\geq R)\pwls(n\omega)\|_{L^2}^2\lesssim 1/R, \forall n.$
If it is delocalized, then some mass will be outside any ball, along a sequence of times,say $n\omega+a, 0<a<1.$
Hence:
$$
  \|F_1(|x|\geq R)\pwls(n\omega+a)\|_{L^2}^2- \|F_1(|x|\geq R)\pwls(n\omega)\|_{L^2}^2\geq \delta >0.
$$
Then,
\begin{align}
&\frac{\partial}{\partial t} \|F_1(|x|\geq R)\pwls(n\omega+t)\|_{L^2}^2 =\\
&\lp\pwls(s),R^{-1}\tilde F_1(|x|/R \sim 1)\gamma \pwls\rp +C.C.+ \mathcal{O}(t^{-\sigma}).
\end{align}
Integrating the above equation from $t_n \to t_n+a$ we find that
$$
\lp\phi,\tilde F_1 \gamma \phi\rp_{t_{n'}}\geq R\delta,
$$
with $t_n<t_{n'}<t_n+a.$ For $R$ large enough, we get a contradiction.
Since the limit of the above expression exists by the above theorem (\ref{norm}), the bound holds for all times.

The fact that the limit exists follows from the generic AC, which guarantees that the dispersive part moves away, and the localized part has a limiting norm.

Since an almost periodic function returns to itself, up to error $\epsilon$ in finite time (depending only on $\epsilon$),
the result follows.
\end{proof}

\begin{rem}
It does not follow that an asymptotically periodic solution is localized.
\end{rem}

\subsection{Evanescent Solutions}

Consider a possible solution in which the WLS becomes small in a large ball around the origin.
Assume first that the WLS vanishes in $\dot H^1$ along a sequence of times.
Then, since by the theorem (\ref{norm}) the $\dot H^1$ of the WLS converges, it must converge for all times.
So, for all sufficiently long times, we have for all $\epsilon$
$$
\sup_{t\geq T}\|\pwls\|_{ \dot H^1}\leq \epsilon.
$$

Now, for a rather general class of interaction terms, we then have

$\lp \pwls,( \mathcal {N}(\phi)+x\cdot \nabla \mathcal {N}(\phi)) \pwls\rp \lesssim \epsilon^a, \, a>1.$
 This combined with the Virial identity that follows from using the Dilation operator $A$ as PROB
 we find that
 \begin{align}
& \partial_t \lp \pwls, A \pwls\rp =  \lp \pwls, \{D_H(A)+i[N(\phi),A]\} \pwls\rp=\\
 & \lp \pwls ,\{-2\Delta-x\cdot \nabla N(\phi) \} \pwls\rp\geq 2\|\pwls\|_{ \dot H^1}^2(1-c\epsilon^{a-1})\\
 &\geq \|\pwls\|_{ \dot H^1}^2.
 \end{align}
 It is not possible for the Virial quantity above to stay positive for all times, as we know that for some sequence of times
 $$\lp \pwls, A \pwls\rp_{t_n}\lesssim 1.
 $$
 We conclude that the total energy must converge to zero along such a sequence, and therefore due to energy conservation ({\bf only for interaction terms which do not depend explicitly on time}), the only such solutions are \emph{zero energy states.}
 \begin{cor}
 If there exists a sequence of times going to infinity such that the WLS is small on a sufficiently large ball around the origin, then such a state converges to a zero-energy state.
 \end{cor}
 \begin{proof}
 The above Virial argument implies that if the $ \dot H^1$ is small (in $L^p, p\geq 2$)  on a sequence of times, the result will follow.
 Since the WLS is regular, the smallness in $L^p, p\geq 2$ implies that $ \dot H^1$ is also small. It is also true at all times, due to the convergence in $H^1$ of the norm of $\pwls.$

 By assumption, the WLS is small inside a large ball.
  Outside a large ball,we use the fact that there exists a sequence of times on which $A$ is bounded:

  $\lp A\pwls(t_n),\bar F_1(|x|\leq c\sqrt t_n) A\pwls(t_n)\rp\lesssim 1.$

 From this it follows that $\lp \nabla\pwls(t_n),\bar F_1(|x|\leq c\sqrt t_n)F_1 (|x|\geq R)\nabla\pwls(t_n)\rp\lesssim 1/R^2 .$

 Hence $\|\pwls\|_{ \dot H^1}^2 \leq o(1)$ on a sequence of times.
 \end{proof}


    \subsection{The Emergence of Solitons}

    We proved that the WLS consists of a part that is large and localized around the origin and a Halo (or corona) around it.
    This Halo part has a non-zero mass ($L^2$ norm) for all times, but carries \emph{zero energy} to infinity.
    As such it acts like a {\bf zombie state} (ignoring the gravitational energy due to the nonzero mass.)
    Further property of the Zombie part is its localization in the phase-space. We showed it is localized where
\begin{equation}
 \bigcup_{0<\alpha\leq 1/2}\{|x|/t^{\alpha}\sim 1; |\gamma|\sim t^{-\alpha}\}.
\end{equation}
So, this is a function with the property that the radial derivative brings a factor of $t^{-\alpha}$ at the point in space $|x|\sim t^{\alpha}.$
This is exactly how a self-similar function behaves! So we expect to approximate the solution by a function $ t^{-\frac{d}{2}\alpha}  S(|x|/t^{\alpha},t)e^{iEt},$ with $S$ regular and sharply localized. $d$ the dimension.
Looking for a special solution for such $S,$ it shows up as solution of an approximate Elliptic Equation, with the leading nonlinear term for large $t$ as the interaction part. Such solutions are the familiar Solitons.
When the non-linear terms are rational functions of $\phi$ only, then the leading equation will be a monomial.
Such solutions in general blow-up in three dimensions (or vanish), but for a general interaction, other terms which depend on time (or space) may take over to stop the blow-up.
Of course, it is also possible that there exist solutions with time-dependent envelopes, $S.$ For example, solutions which live on different scales for large times, and for large $x.$
\begin{cor}[Emergence Of Self-Similarity]
Suppose that for some $0<\alpha\leq 1/2$ we have a weakly localized state $\psi(t)$ with non-zero mass ($L^2 -norm$) localized inside the space-time region $|x|\thicksim t^{\alpha}.$ Then, asymptotically in time the solution converges to a self-similar solution, at least on some subsequences.
\end{cor}
\begin{proof}
We have shown before that a part of the weakly localized solution that concentrates in $|x|\thicksim t^{\alpha}$
has its momentum concentrates at $|p|\leq t^{-\alpha}\ln t.$ We also know that the expectation of the dilation operator is uniformly bounded in time on such a state.
Therefore, for all $M$ large enough, we have that:

\begin{equation}
0<\epsilon_M \leq \|F_1(\frac{|x|}{t^{\alpha}}=1)  F_2(|p|t^{\alpha}\leq M)\psi(t)\| \leq C<\infty, \quad t>>1. 
\end{equation}

Applying the dilation group generated by $A, U(t)=e^{iAt}$ we get the following.

\begin{equation}
U(\alpha \ln t)   F_1(\frac{|x|}{t^{\alpha}}=1)F_2(|p|t^{\alpha}\leq M) \psi(t) =F_1(|x|\thicksim 1)F_2(|p|\leq M)U(\alpha \ln t)\psi(t).
\end{equation}

Now, the product $F_1(|x|\thicksim 1)F_2(|p|\leq M)$ is a compact operator, and $U(\alpha \ln t)\psi(t)$ is a uniformly bounded set in $L^2$ (but not in $H^1$). Hence, every infinite sequence of time of the above expression is given by a compact operator times a uniformly bounded sequence in $L^2.$ Any such sequence contains a weakly convergent subsequence, and after the action of the compact operator it converges strongly. Hence we have sequential convergence of a channel wave operator, with the dilation group as the asymptotic "dynamics".
That is, the asymptotic behavior along the above subsequence is given by $U(\alpha \ln t_n)\psi_{\infty}(x)$.
\end{proof}
Inserting such a solution into the NLS, we see that this is only possible if the asymptotic behavior of the Laplacian term and the interaction terms scales the same way as $t$ goes to infinity. Since the Laplacian term scales like $t^{2\alpha},$
the interaction term must have a part that scales like a mass critical term, that is, either a nonlinear term with power $4/n$
or a time dependent potential that scales like a mass critical term. Moreover, the original Schr\"odinger equation factorizes, and the resulting elliptic equation has a soliton profile as a solution.

\begin{rem}
If one can show that there are no WLS present in the solution, it follows that Scattering holds. This would be the case for any interaction that satisfies a Morawetz-type estimate (repulsive interactions) together with the existence in $H^1$.
\end{rem}

We end with some conjectures.
\begin{conjec}[ The Petite Conjecture]
A localized solution, which is also regular in space, is an almost periodic function of time.
Notice that WLS does not satisfy this property. Self-similar solutions are not almost periodic.
Our analysis show that the breakdown of this property should come from a high concentration of low frequency modes(in space and time) near zero.
A progress in this direction would be to show that a time-dependent potential with chaotic time behavior, localized in space, will, at least generically, destroy the bound states. See \cite{Kirr, Pyke,BeSof}.
\end{conjec}
\medskip

\begin{conjec}
Time-dependent potentials, localized in space, may have weakly localized solutions which are not localized.
\end{conjec}

\begin{conjec}
Study the stability of weakly localized solutions, by extending the method of Modulation Equations to include the Ansatz that the Halo part is self similar.
\end{conjec}

\section{Appendix A:  Discussion of special cases}
In this section, we work with concrete examples to clarify the validity of our theorems for each of them. 
\begin{enumerate}
\item Example 1: Consider $\textbf{N}(\phi)=-|\phi|^2\phi$ with $\alpha =\frac12+\epsilon$. Now assumption (H2') is fulfilled. 

The $\gamma$ limit theorem~\ref{thm:r-limit} holds. 
\item Example 2: Consider $\textbf{N}(\phi)=- |\phi|^p\phi$, with $p\in (\frac43,4)$.  We take $\alpha>\frac{1}{p}\in (\frac14, \frac34)$, the two endpoints can be reached. 
The $\gamma$ limit theorem~\ref{thm:r-limit} holds with $\alpha=\max{(\frac13+,1/p)}$. 

\item Example 3: Consider  $\textbf{N}(\phi)=- \frac{|\phi|^m}{1+|\phi|^{m-n}}\phi$, with $m>2+\frac43,  n\in(1, \frac43)$. Here we take $\alpha>\frac{1}{m}$, which is $\frac{3}{10}$ or smaller.  The $\gamma$ limit theorem~\ref{thm:r-limit} holds. 
\item Example 4: Consider $\textbf{N}(\phi)=  |\phi|^p\phi +V(x)\phi$, with $|V(x)|\lesssim \frac{1}{(1+|x|)^q}$ such that $p\in (\frac43,4)$ and $q>1$.  We take $\alpha>\max{(\frac{1}{p}, \frac{1}{q})}.$  If $p=q=2$, we take $\alpha=\frac12+.$ If $q\geq3, p\geq 3, $ we take $\alpha=\frac13+$
The $\gamma$ limit theorem~\ref{thm:r-limit} holds for $\alpha=\max{(\frac{1}{p}, \frac{1}{q}, \frac13)}+$
\end{enumerate}

\section{Appendix B: Miscellaneous calculations}
\subsection{commutators}
For the operator 
$$
\gamma_A=-\frac{i}{2}(A(x) \cdot \nabla+\nabla \cdot A(x))
$$
under the condition $A(x)\in\R^3$ and $\partial_j A_i(x)=\partial_i A_j(x)$.  We have the formula 
\[[-i\Delta, \gamma_A]= =-\frac{1}{2}\left(4 \nabla_i \partial_iA_j(x) \nabla_j+\partial_{iij}(A_j(x))\right) \]
In particular, we consider few cases
\begin{enumerate}
\item If $A(x)=\nabla g$, then \[\partial_iA_j=g_{ij},\quad  \partial_{i i j}\left(A_j(x)\right)=\partial_{i i j j} g(x)=\Delta^2 g\]  
\item If $A(x)=xf(r)$, then by direct computation we have 
\end{enumerate}
$$
\begin{aligned}
\partial_i\left(A_j(x)\right)=\left(x_j f(r)\right)_i & =\delta_{i j} f(r)+x_5 f^{\prime}(r) \frac{x_i}{r} \\
& =f(r) \delta_{i j}+\frac{f^{\prime}(r)}{r} x_i x_j
\end{aligned}
$$

Then $$
 \quad \sum_{i} \partial_i\left(A_i(x)\right)=3f(r) +f^{\prime}(r)r
$$
And  
\[\partial_j(\sum_{i} \partial_i\left(A_i(x)\right))= (4 f^{\prime}(r) \frac{1}{r}+f^{\prime \prime}(r) )x_j
\]

Hence 
$$
\begin{aligned}
\sum_{i,j=1}^3(A_i(x))_{i j j}&  =3(\frac{4 f^{\prime}(r)}{r}+f^{\prime \prime}(r)) 
+r \frac{d}{dr}\left(\frac{4 f^{\prime}(r)}{r}+f^{\prime \prime}(r)\right)\\
& =3\left(\frac{4 f^{\prime}(r)}{r}+f^{\prime \prime}(r)\right) \\
& +r\left(\frac{4 f^{\prime \prime}(r)}{r}+\frac{4 f^{\prime}(r)}{r^2}(-1)+f^{\prime \prime\prime}(r)\right)  \\
& =\frac{8 f^{\prime}(r)}{r}+7 f^{\prime \prime}(r)+f^{\prime \prime \prime}(r) r\\
\end{aligned}
$$

 Take special  $ g(x)=\langle x\rangle=\sqrt{1+|x|^2}$, then $g_i=\frac{x_i}{\langle x\rangle}$
$$
\begin{aligned}
& g_{i j}=\frac{\delta_{i j}}{\langle x\rangle}-\frac{x_i x_j}{\langle x\rangle^3} \\
& \Delta g=\frac{3}{\langle x\rangle}-\frac{|x|^2}{\langle x\rangle^3}=\frac{2}{\langle x\rangle}+\frac{1}{\langle x\rangle^3}
\end{aligned}
$$
$$
(\Delta g)_i=\frac{-2}{\langle x\rangle^2} \frac{x_i}{\langle x\rangle}-\frac{3}{\langle x\rangle^4} \frac{x_i}{\langle x\rangle} =\left(-\frac{2}{\langle x\rangle^3}-\frac{3}{\langle x\rangle^5}\right) x_i
$$
$$
\begin{aligned}
& (\Delta g)_{i j}=\left(-\frac{2}{\langle x\rangle^3}-\frac{3}{\langle x\rangle^5}\right) \delta_{i j}+\left(\frac{6}{\langle x\rangle^4} \frac{x_i x_j}{\langle x\rangle}+15 \frac{x_i x_j}{\langle x\rangle^7}\right) \end{aligned}$$
Finally we get 
\[\Delta^2g=  
 =-\frac{6}{\langle x\rangle^3}-\frac{9}{\langle x\rangle^5}+\frac{6|x|^2}{\langle x\rangle^5}+\frac{15|x|^2}{\langle x\rangle^7}= -\frac{15}{\langle x\rangle^7} \]

Remark on radial functions, we can check the leading term is 
\[<\phi, -\partial_i g_{ij} \partial_j \phi>=\int \frac{|\nabla\phi|^2}{\langle x\rangle^3}dx\]
so the decay is too fast. 
 We modify the choice by considering 
$A(x)=xf(r)$ with 
 $f(r)=\left(|x|^{1-\varepsilon}\langle
x\rangle^\varepsilon \right)^{-1}.$
now
$$
\begin{aligned}
& \partial_i A_j(x)
 =\frac{\delta_{i j}}{|x|^{1-\varepsilon}\langle x\rangle^{\varepsilon}}+\left(\frac{(\varepsilon-1)}{|x|^{3-\varepsilon}\langle x\rangle^{\varepsilon}}+\frac{(-\varepsilon)}{|x|^{1-\varepsilon}(x)^{\varepsilon+2}}\right) x_i x_j \\
&
\end{aligned}
$$

check the leading term $<\phi, -\partial_i g_{ij} \partial_j \phi>$
$$
\int \frac{|\nabla \phi|^2}{|x|^{1-\varepsilon}\langle x\rangle^\varepsilon}+\left(\frac{(\varepsilon-1)}{|x|^{3-\varepsilon}\langle x\rangle^\varepsilon}+\frac{(-\varepsilon)}{|x|^{1-\varepsilon}
\langle x\rangle)^{\varepsilon+2}}\right)|x \cdot \nabla \phi|^2
$$
In particular for radial functions
for radial we get

for radial we get
$$
\begin{aligned}
& \int \frac{\left|\nabla \phi\right|^2}{|x|^{1-\varepsilon}\langle x\rangle^{\varepsilon}}+\frac{(\varepsilon-1)}{|x|^{1-\varepsilon}\langle x\rangle)^{\varepsilon}}|\nabla \phi|^2-\frac{\varepsilon|x|^{1+\varepsilon}}{\langle x\rangle^{2+\varepsilon}}|\nabla \phi|^2 \\
= & \int\left( \frac{\varepsilon}{|x|^{1-\varepsilon}\langle x\rangle^{\varepsilon}}-\frac{\varepsilon|x|^{1+\varepsilon}}{\langle x\rangle^{2+\varepsilon}}\right)|\nabla \phi|^2 \\
= & \int \frac{\varepsilon\left(\langle x\rangle^2-|x|^2\right)}{|x|^{1-\varepsilon}\langle x\rangle^{2+\varepsilon}}|\nabla \phi|^2=\int \frac{\varepsilon |\nabla\phi|^2}{|x|^{1-\varepsilon}\langle x\rangle^{2+\varepsilon}  }
\end{aligned}
$$
this gives $\frac{\left.\nabla \phi\right|^2}{|x|^{1-\varepsilon}}$ at the origin.  And it gives $\frac{1}{|x|^3}$ weight at infinity.

A long and tedious calculation gives $(A_j(x))_{iij}\approx \frac{1}{|x|^{3-\varepsilon}\langle x\rangle^{2+\varepsilon}}$, 

 Now we with $\theta=2-\epsilon$, we consider
$$
\begin{aligned}
& g(r)=\int_0^r \frac{s}{\sqrt{s^2+s^ \theta}} d s \\
& \nabla g=\frac{r}{\sqrt{r^2+r^ \theta}} \frac{x}{r}=\frac{x}{\sqrt{r^2+r^\theta}}
\end{aligned}
$$
Now we have
$$
\begin{aligned}
g_{i j} & =\frac{\delta_{i j}}{\sqrt{r^2+r^{2-\epsilon}}}+\frac{x_i\left(-\frac{1}{2}\right)}{\left(r^2+r^\theta\right)^{3 / 2}}\left(2 x_j+\theta r^{\theta-1} \frac{x_j}{r}\right) \\
& =\frac{\delta_{i j}}{\sqrt{r^2+r^{2-\epsilon}}}+\left(-\frac{1}{2}\right) \frac{\left(2+\theta r^{\theta-2}\right)}{\left(r^2+r^\theta\right)^{3 / 2}} x_i x_j
\end{aligned}
$$
 The leading term $<\phi, -\partial_i g_{ij} \partial_j \phi>$ is 
$$
\int \frac{|\nabla \phi|^2}{\sqrt{r^2+r^{2-\epsilon}}}-\frac{1}{2} \frac{\left(2+\theta r^{\theta-2}\right)|x \cdot \nabla \phi|^2}{\left(r^2+r^\theta\right)^{3 / 2}}
$$
For radial function we have 
radial
$$
\begin{aligned}
& \int|\nabla \phi|^2 \left( \frac{1}{\sqrt{r^2+r^{2-\varepsilon}}}-\frac{1}{2} \frac{\left(2 r^2+\theta r^\theta\right)}{\left(r^2+r^\theta\right)^{3 / 2}}\right) \\
= & \int|\nabla \phi|^2 \frac{r^2+r^{2-\varepsilon}-r^2-\theta r^\theta}{\left(r^2+r^{2-\varepsilon}\right)^{3 / 2}}  \\
= & \int \frac{|\nabla \phi|^2 \frac{\varepsilon}{2} r^{2-\varepsilon}}{\left(\sqrt{\left.r^2+r^{2-\varepsilon}\right)^3}\right.} \\
= & \int \frac{\varepsilon}{2}|\nabla \phi|^2 \frac{r^{2-\varepsilon}}{r^{(1-\varepsilon / 2)^3}\left(\sqrt{r^\epsilon+1}\right)^3} \\
= & \int \frac{\varepsilon}{2}|\nabla \phi|^2 \frac{1}{r^{1-\varepsilon / 2}\left(r^\epsilon+1\right)^{3 / 2}}
\end{aligned}
$$
This has better behavior compared to previous one. Near the origin they are the same, while at infinity this one has weight $\frac{1}{|x|^{1+\epsilon}}$

 Computing $\Delta^2 g$ we have that
the leading term is \[-\frac{15}{4}\frac{r^2}{(r^2+r^\theta)^{\frac52}}\].

Now we will use the last choice directly to prove local smoothing estimate. 

\subsection{smoothing estimate}
To prove the optimal smoothing estimate, let's consider $\left\langle\nabla^{\frac{1}{2}} \gamma_g \nabla^{\frac{1}{2}}\right\rangle$
$$
\begin{aligned}
& \left.\left\langle\nabla^{\frac{1}{2}} \gamma_g \nabla^{\frac{1}{2}}\right\rangle\right|_{T_0} ^{\top} \\
& =\int_{T_0}^T \int_{\mathbb{R}^3} \frac{\left|\nabla^{3 / 2} \phi\right|^2}{r^{1-\epsilon / 2}\langle r\rangle^{\frac{3}{2} \varepsilon}}+\frac{\left|\nabla^{\frac{1}{2}} \phi\right|^2}{r^{3-\frac{5}{2} \varepsilon}\langle r\rangle^{5 \epsilon}} d x d t \\
& +2\Im\int_{T_0}^T \int_{\mathbb{R}^3} \nabla^{\frac{1}{2}} \phi \quad \gamma_g \nabla^{\frac{1}{2}}(N(\phi) \phi) d x d t \\
&
\end{aligned}
$$
now nonlinear term is of the from
$$
\begin{aligned}
&\iint \nabla^{3 / 2} \phi \frac{r^\epsilon}{\langle r\rangle^\epsilon} \quad \nabla^{\frac{1}{2}} \phi \quad N(\phi) d x d t\\
\leq &  \frac{\epsilon}{2}\left(\iint \frac{\mid\nabla^{3 / 2} \phi \mid^2}{\left|r^{\frac{1}{2}-\frac{\epsilon}{4}}\right|^2} d x d t+\iint\left|\frac{\nabla^{\frac{1}{2} \phi}}{r^{3 / 2-\frac{5}{4} \epsilon}}\right|^2 d x d r\right) \\
&+\frac{1}{\epsilon} \iint_{r \leqslant 1}\left|N(\phi) r^{2-\frac{1}{2} \epsilon}\right|_{L^{\infty}} d x d t
\end{aligned}
$$

$$
\begin{gathered}
N(\phi)=|\phi|^m \\
|\sqrt{r} \phi|<\|\phi\|_{H^{1}}
\end{gathered}
$$
Another option is to take
$\int_{r \leqslant \delta}$ by take $\delta$ small enough. 

Now  
$$
\begin{aligned}
&\iint_{r \geqslant \delta} \nabla^{3 / 2} \phi \nabla^{1 / 2} \phi \quad N(\phi) d x d t\\
\leq &  \iint \frac{\nabla^{3 / 2} \phi}{r^{1 / 2}} \quad r^{\frac{1}{2}} N(\phi) \nabla^{\frac{1}{2}} \phi \\
& \leq \epsilon_1 \iint\left| \frac{\nabla^{3 / 2} \phi}{r^\frac12}\right|^2 d x d t +\frac{1}{\epsilon_1}\iint\left|r^{\frac{1}{2}} N(\phi)\right|_{L^{\infty}} \left\|\nabla^{\frac{1}{2}} \phi\right\|_{L^2} d xdt . 
\end{aligned}
$$
hence
 this gives us the optimal Smoothing estimate
$$\begin{aligned}
&\int_{T_0}^T \int_{\mathbb{R}^3} \frac{\left|\nabla^{3 / 2} \phi\right|^2}{r^{1-\epsilon / 2}\langle r\rangle^{\frac{3}{2} \varepsilon}}+\frac{\left|\nabla^{\frac{1}{2}} \phi\right|^2}{r^{3-\frac{5}{2} \varepsilon}\langle r\rangle^{5 \epsilon}} d x d t \\
\lesssim&  C(\|\phi\|_{H^1},  |T-T_0|)
\end{aligned}
$$


\begin{thebibliography}{99}



\bibitem{Segev} A. Barak, O. Peleg, C. Stucchio, A. Soffer, and M. Segev.
\emph{Observation of soliton tunneling phenomena and soliton ejection.}  Phys. Rev. Lett. 100, 153901 (2008).



\bibitem{BeSof}M. Beceanu  and A. Soffer.\emph{ A semilinear Schroedinger equation with random potential.} arXiv preprint arXiv:1903.03451 (2019).

\bibitem{NLSresolution} M. Borghese, R. Jenkins, and  K. T.-R. McLaughlin. \emph{Long-time asymptotic behavior of the
focusing nonlinear Schr\"{o}dinger equation}. Annales Inst. H. Poincar\`{e}, Analyse nonlin\`{e}aire
Volume 35, Issue 4, July 2018, Pages 887-920


\bibitem{dKG}
N. Burq, G. Raugel and W. Schlag. \emph{Long time dynamics for damped Klein-Gordon equations.}
Ann. Sci. \'{E}c. Norm. Sup\'{e}r. (4) 50 (2017), no. 6, 1447-1498.





\bibitem{mKdV} G. Chen and J. Liu. \emph{Soliton resolution for the modified KdV Equation. arXiv: 1907.07115 }. arXiv:1903.03855

\bibitem{SG} G. Chen, J. Liu and B. Lu. \emph{Long-time asymptotics and stability for the sine-Gordon equation}. arXiv: 2009.04260


\bibitem{cote} R. C\^{o}te. \emph{On the soliton resolution for equivariant wave maps to the sphere}, Comm. Pure Appl. Math. 68 (2015), no. 11, 1946–2004.


\bibitem{DP} T. Dauxois, M. Peyrard. \emph{Physics of Solitons}, Cambridge Univ. Press, Cambridge, 2006.

\bibitem{mKdVresolution-2} P. Deift, X. Zhou. \emph{A steepest descent method for oscillatory Riemann-Hilbert problems}.
Asymptotics for the MKdV equation. Ann. of Math. (2) 137, no. 2 (1993), 295–368.

\bibitem{4d}  R. Cote, C. Kenig, A. Lawrie and W. Schlag.
\emph{Profiles for the radial focusing 4d energy-critical wave equation}. Comm. Math. Phys. 357 (2018), no. 3, 943--1008.

\bibitem{DJKM}T. Duyckaerts, H. Jia, C. Kenig and F. Merle.  \emph{Soliton resolution along a sequence of times for the focusing energy critical wave equation},  Geometric and Functional Analysis, Vol 27, Issue 4, 2017, 798-862,

\bibitem{DKMM} T. Duyckaerts, C. Kenig, Y. Martel, and F. Merle.  \emph{Soliton resolution for critical co-rotational wave maps and
radial cubic wave equation}.  	arXiv:2103.01293.


\bibitem{DKM1} T. Duyckaerts, C. Kenig, and F. Merle. \emph{Universality of blow-up profile for small radial type II blow-up solutions of the energy-critical wave equation}, J. Eur. Math. Soc. (JEMS) 13 (2011), no. 3, 533–599.

\bibitem{DKM2} T. Duyckaerts, C. Kenig, and F. Merle. \emph{Profiles of bounded radial solutions of the focusing,
energy–critical wave equation}. Geom. Funct. Anal. 22 (2012), 639–698.

\bibitem{DKM} T. Duyckaerts, C. Kenig, and F. Merle. \emph{Classification of radial solutions of the focusing,
energy–critical wave equation}. Cambridge Journ. of Math. 1 (2013), 75–144.


\bibitem{DKM3} T. Duyckaerts, C. Kenig, and F. Merle. \emph{Solutions of the focusing nonradial critical wave equation with the compactness property, }Ann. Sc. Norm. Super. Pisa Cl. Sci. (5) 15 (2016), 731–808.



\bibitem{DKM-o1} T. Duyckaerts, C. E. Kenig, and F. Merle. \emph{Exterior energy bounds for the critical wave equation close to the
ground state}. Comm. Math. Phys., 379:1113–1175, 2020.


\bibitem{DKM-o2} T. Duyckaerts, C. E. Kenig, and F. Merle.  \emph{Decay estimates for nonradiative solutions of the energy-critical
focusing wave equation}. J. Geom. Anal., 2021.

\bibitem{DKM-o3} T. Duyckaerts, C. E. Kenig, and F. Merle. \emph{Soliton resolution for the radial critical wave equation in all odd
space dimensions}. Acta Math., to appear.


\bibitem{KdVresolution} W. Eckhaus, and  P. C. Schuur. \emph{The emergence of solitons of the Korteweg-de Vries equation from
arbitrary initial conditions}. Math. Methods Appl. Sci. 5, 1 (1983), 97–116.


\bibitem{Frank} R.L. Frank, R. Seiringer. \emph{Non-linear ground state representations and sharp Hardy inequalities}, J. Funct. Anal. 255 (12) (2008) 3407–3430.

\bibitem{dNLS} R. Jenkins, J. Liu, P. Perry and C. Sulem. \emph{Soliton Resolution for the Derivative Nonlinear Schrödinger Equation}, Comm. Math. Phys. 363 (2018), no. 3, 1003–1049.

\bibitem{Herbst} I. W. Herbst.  \emph{Spectral theory of the operator $(p^2+m^2)^{\frac12}-Ze^2/r$. }Comm. Math. Phys. 53
(1977), no. 3, 285–294.



\bibitem{HS-JMP} W. Hunziker and I.M. Sigal. \emph{The quantum N-body problem}. J. Math. Phys. 41 (2000), no. 6, 3448–3510.

\bibitem{HS} W. Hunziker and I.M. Sigal.  \emph{Time-dependent scattering theory of n-body quantum systems} Reviews in Mathematical Physics, Vol. 12, No. 8 (2000) 1033–1084.


\bibitem{MV} W. Hunziker, I.M. Sigal, A. Soffer. \emph{Minimal Escape Velocities}, CPDE 24, no11-12, (1999), 2279-2295.


\bibitem{JL} J. Jendrej and A. Lawrie. \emph{Soliton resolution for equivariant wave maps}. 	arXiv:2106.10738.

\bibitem{JK} H. Jia and C. Kenig. \emph{Asymptotic decomposition for semilinear wave and equivariant wave map equations} ,  American Journal of Mathematics 139 (2017), pages 1521-1603

\bibitem{NLWpotential} H. Jia, B. Liu and G. Xu. \emph{ Long time dynamics of defocusing energy critical 3 + 1 dimensional wave equation with potential in the radial case}.
 Communications in Mathematical Physics, Volume 339, Issue 2, pp 353-384.

\bibitem{Keel-Tao} M. Keel, T. Tao, \emph{Endpoint Strichartz estimates.} Amer. J. Math., 120(1998), 955-980
\bibitem{Kirr}Kirr, E., and M. I. Weinstein. \emph{Metastable states in parametrically excited multimode Hamiltonian systems.}Communications in mathematical physics 236.2 (2003): 335-372.



\bibitem{KOPV}R. Killip,  T. Oh,  P. Pocovnicu and  M. Vişan. \emph{Solitons and Scattering for the Cubic–Quintic Nonlinear Schrödinger Equation on ${\mathbb {R}^ 3} $}. Archive for Rational Mechanics and Analysis, 225(1)  (2017), pp.469-548.

\bibitem{Exterior-2}  C. Kenig, A. Lawrie, B. Liu and W. Schlag.
 \emph{Stable soliton resolution for exterior wave maps in all equivariance classes}.
Advances in Math. 285 (2015), 235-300.

\bibitem{Exterior-1}  C. Kenig, A. Lawrie and W. Schlag. \emph{ Relaxation of wave maps exterior to a ball to harmonic maps for all data}.
Geom. Funct. Anal. (GAFA). 24 (2014), no. 2, 610-647.




\bibitem{Kapitula} T. Kapitula  and B. Sandstede. \emph{ Stability of bright solitary-wave solutions to perturbed nonlinear Schrödinger equations.} Physica D: Nonlinear Phenomena, 124(1-3)  (1998), pp.58-103.

\bibitem{CSS} K. Kim, S. Kwon, and S.-J. Oh,   \emph{Soliton resolution for equivariant self-dual Chern-Simons-Schrödinger equation in weighted Sobolev class.} Arxiv: 2202.07314


 \bibitem{KG2}A.I. Komech,  A.A. Komech. \emph{Global attractor for a nonlinear oscillator coupled to the Klein–Gordon field}.
Arch. Ration. Mech. Anal. 185 (2007), 105–142,

\bibitem{Li-Sof1}B. Liu and   A. Soffer. \emph{A General Scattering theory for Nonlinear and Non-autonomous Schroedinger Type Equations-A Brief description.} arXiv preprint arXiv:2012.14382 (2020).

\bibitem{KG1} A.I. Komech. \emph{On attractor of a singular nonlinear U(1)-invariant Klein–Gordon equation}, in Progress in
analysis, Vol. I, II (2001, Berlin), World Sci. Publishing, River Edge, NJ, 2003, 599–611.

\bibitem{Lin-Virial}Lin, T.C., Belić, M.R., Petrović, M.S., Hajaiej, H. and Chen, G., . \emph{The virial theorem and ground state energy estimates of nonlinear Schrödinger equations in $\mathbb {R}^ 2$  with square root and saturable nonlinearities in nonlinear optics.} Calculus of Variations and Partial Differential Equations, 56(5)  (2017), pp.1-20.

\bibitem{Manton} N. Manton and P. Sutcliffe. \emph{Topological Solitons}. Cambridge Monogr. Math. Phys., Cam-
bridge University Press, Cambridge 2004.

\bibitem{Malomed}Hidetsugu Sakaguchi and Boris A. Malomed
Phys. Rev. E 73, 026601 – Published 2 February 2006





\bibitem{Merle}Merle, F. . \emph{Limit behavior of saturated approximations of nonlinear Schrödinger equation.} Communications in mathematical physics, 149(2) 1992, pp.377-414.


\bibitem{Martel5} Y. Martel. \emph{Interaction of solitons from the PDE point of view},  PROC. INT. CONG. OF MATH. – 2018– 2018 Rio de Janeiro, Vol. 3 (2457–2484)


\bibitem{Marzuola}Marzuola, J.L.  \emph{A Class of Stable Perturbations for a Minimal Mass Soliton in Three-Dimensional Saturated Nonlinear Schrödinger Equations.} SIAM journal on mathematical analysis, 42(3)   2010, pp.1382-1403.

\bibitem{Pyke}Pyke, Randall Mitchell, and I. M. Sigal. \emph{Nonlinear wave equations: Constraints on periods and exponential bounds for periodic solutions.} Duke Mathematical Journal 88.1 (1997): 133-180.






\bibitem{RS1} M. Reed and B. Simon. \emph{Methods of Modern Mathematical Physics. I. Functional Analysis}, Academic Press, New York, 1972.

\bibitem{RSS} I. Rodnianski, W. Schlag, and A. Soffer. \emph{ Asymptotic stability of N-soliton states of NLS}, ArXiv e-prints: math/0309114 (2003).



\bibitem{mKdVresolution} P. C. Schuur.   \emph{Asymptotic analysis of soliton problems}. vol. 1232 of Lecture Notes in Mathematics. Springer-
Verlag, Berlin, 1986. An inverse scattering approach.

\bibitem{Segur} H. Segur. \emph{ The Korteweg-de Vries equation and water waves. Solutions of the equation}. I. J. Fluid Mech. 59
(1973), 721–736.

\bibitem{NLSresolution-1} H. Segur, , and M. J.  Ablowitz. \emph{ Asymptotic solutions and conservation laws for the nonlinear Schr\"{o}dinger
equation}. J. Mathematical Phys. 17, 5 (1976), 710–716.



\bibitem{SSAnnals} I. M. Sigal and A. Soffer. \emph{The N-particle scattering problem: asymptotic completeness for short-range systems}, Ann. of Math. (2) 126 (1987), no. 1, 35–108. MR 898052


\bibitem{SSpreprint} I.M. Sigal and A. Soffer.  \emph{ Local decay and propagation estimates for timedependent and time-independent Hamiltonians}. Preprint, Princeton Univ.
(1988) http://www.math.toronto.edu/sigal/publications/SigSofVelBnd.pdf.



\bibitem{SSInvention} I. M. Sigal and A. Soffer. \emph{ Long-range many-body scattering. Asymptotic clustering for Coulomb-type potentials}, Invent. Math. 99 (1990), no. 1, 115–143. MR 1029392
\bibitem{SSDuke} I. M. Sigal and A. Soffer. \emph{Asymptotic completeness for $N \leq 4$ particle systems with the Coulomb-type interactions}, Duke Math. J. 71 (1993), no. 1, 243–298. MR 1230292

\bibitem{SSjams} I. M. Sigal and A. Soffer, \emph{ Asymptotic completeness of N-particle long-range scattering}, J. Amer. Math. Soc. 7 (1994), no. 2, 307–334. MR 1233895



\bibitem{Soffer}  A. Soffer. \emph{ Soliton dynamics and scattering, International Congress of Mathematicians}. Vol. III, Eur. Math. Soc., Z\"{u}rich, 2006, pp. 459–471.

\bibitem{Soffer-monotonic}  A. Soffer. \emph{ Monotonic Local Decay Estimates}, arXiv:1110.6549



\bibitem{SofferW} A. Soffer and M.I. Weinstein. \emph{Multichannel nonlinear scattering for nonintegrable equations}. Comm. Math. Phys. 133 (1990), 119–146.

\bibitem{SofferW2 } A. Soffer  and M.I. Weinstein.  \emph{Multichannel nonlinear scattering, II. The case of anisotropic potentials and data}. J. Differential Equations 98 (1992), 376–390.

\bibitem{SSNLS} C. Sulem, P.-L. Sulem. \emph{The nonlinear Schrödinger equation. Self-focusing and wave collapse} . Appl. Math. Sci. 139, Springer-Verlag, New York 1999.

\bibitem{RR} R. Rajaraman. \emph{Solitons and Instantons: An Introduction to Solitons and Instantons in Quantum Field Theory}. North-Holland Publishing Co., Amsterdam, 1982.

\bibitem{CR1} C. Rodriguez. \emph{
Soliton resolution for equivariant wave maps on a wormhole. }
Comm. Math. Phys. 359 (2018), no. 1, 375–426.

\bibitem{CR2} C. Rodriguez. \emph{
Soliton resolution for corotational wave maps on a wormhole.}
Int. Math. Res. Not. IMRN 2019, no. 15, 4603–4706.


\bibitem{Roy} T. Roy.  \emph{
A weak form of the soliton resolution conjecture for high-dimensional fourth-order Schrödinger equations. }
J. Hyperbolic Differ. Equ. 14 (2017), no. 2, 249–300.


\bibitem{Simon} M. Reed and B. Simon. \emph{Methods of Modern Mathematical Physics. I. Functional Analysis}, Academic Press, New York, 1972

    \bibitem{Su-Su}Sulem, Catherine, and Pierre-Louis Sulem. \emph{The nonlinear Schrödinger equation: self-focusing and wave collapse.} (2007).

\bibitem{Tao04} T. Tao. \emph{On the asymptotic behavior of large radial data for a focusing non-linear Schr\"{o}dinger equation.}
 Dynamics of PDE, Vol.1, No.1, 1-47, 2004

\bibitem{Tao07} T. Tao.  \emph{A (Concentration-)Compact Attractor for High-dimensional Non-linear Schr\"{o}dinger Equations}. Dynamics of PDE, Vol.4, No.1, 1-53, 2007

\bibitem{Tao08} T. Tao.  \emph{A global compact attractor for high-dimensional defocusing nonlinear Schr\"{o}dinger
equations with potential}. Dynamics of PDE, Vol.5, No.2, 101-116, 2008

 \bibitem{Tao-soliton}  T. Tao. \emph{Why are solitons stable?}  Bull. Amer. Math. Soc. (N.S.) 46, 1 (2009), 1–33.




\bibitem{Yang} Y. Yang. \emph{Solitons in Field Theory and Nonlinear Analysis}, Springer Monographs in Mathematics, Springer-Verlag, 2001


 \bibitem{NLSresolution-2} V. E. Zakharov,  A. B. Shabat. \emph{Exact theory of two-dimensional self-focusing and one-dimensional
self-modulation of waves in nonlinear media}. Z. Eksper. Teoret. Fiz. 61, 1 (1971), 118–134.
 \end{thebibliography}
\end{document}